\documentclass[12pt]{article}
\usepackage{graphicx,amsmath,amsthm,amssymb,epsfig,psfrag,color}  
\usepackage[a4paper,margin=2.5cm,top=1.5cm,bottom=4cm]{geometry}

\date{13 November 2010}

\numberwithin{equation}{section}

\theoremstyle{plain}
\newtheorem{thm}{Theorem}[section]
\newtheorem{cor}[thm]{Corollary}
\newtheorem{prop}[thm]{Proposition}
\newtheorem{lem}[thm]{Lemma}

\theoremstyle{definition}

\newtheorem{notation}[thm]{Notation}

\def\R{\mathbb{R}}

\def\N{\mathbb{N}}
\def\P{\mathbb{P}}
\def\E{\mathbb{E}}

\def\I{\infty}

\newcommand{\be}{\begin{equation}}
\newcommand{\ee}{\end{equation}}
\newcommand{\bea}{\begin{eqnarray}}
\newcommand{\eea}{\end{eqnarray}}
\newcommand{\beann}{\begin{eqnarray*}}
\newcommand{\eeann}{\end{eqnarray*}}
\newcommand{\benn}{\begin{equation*}}
\newcommand{\eenn}{\end{equation*}}

\def\ra{\rightarrow}
\def\I{\infty}

\def\writefig#1 #2 #3 {\rlap{\kern #1 truecm
\raise #2 truecm \hbox{\protect{\small #3}}}}

\DeclareMathSymbol{\leqsymb}{\mathalpha}{AMSa}{"36}
\def\leqs{\mathrel\leqsymb}
\DeclareMathSymbol{\geqsymb}{\mathalpha}{AMSa}{"3E}
\def\geqs{\mathrel\geqsymb}

\DeclareMathOperator{\dd}{d}            
\DeclareMathOperator{\e}{e}             
\DeclareMathOperator{\icx}{i}           
\DeclareMathOperator{\Tr}{Tr}           
\DeclareMathOperator{\Id}{Id}           
\DeclareMathOperator{\Cov}{Cov}         
\DeclareMathOperator{\re}{Re}		
\DeclareMathOperator{\im}{Im}		

\def\const{\text{\it const}}

\def\6#1{\dd\!#1}                       
\def\abs#1{\lvert#1\rvert}                      
\def\norm#1{\left\|#1\right\|}          
\def\Order#1{{\mathcal O}(#1)}                  
\def\pscal#1#2{\langle #1,#2 \rangle}	

\newcommand{\cB}{{\mathcal B}}  
\newcommand{\cD}{{\mathcal D}}  
\newcommand{\cO}{{\mathcal O}}  
\newcommand{\cR}{{\mathcal R}}  
\newcommand{\cS}{{\mathcal S}}  

\def\Vbar{\overline V}
\def\vbar{\bar v}

\begin{document}
 
\title{Hunting French Ducks in a Noisy Environment}
\author{Nils Berglund\thanks{MAPMO,
CNRS -- UMR 6628, Universit\'{e} d'Orl\'{e}ans, F\'{e}d\'{e}ration Denis
Poisson -- FR 2964, B.P. 6759, 45067 Orl\'{e}ans Cedex 2, France.} 
\thanks{Supported by ANR project MANDy, Mathematical Analysis of Neuronal
Dynamics, ANR-09-BLAN-0008-01.}
\and Barbara Gentz\thanks{Faculty of Mathematics, University of 
Bielefeld, P.O. Box 10 01 31, 33501 Bielefeld,
Germany.} 
\thanks{Supported by the DFG-funded CRC 701, Spectral Structures and Topological
Methods in Mathematics, at the University of Bielefeld.}
\and Christian Kuehn\thanks{Max Planck Institute for Physics of Complex Systems,
Noethnitzer Str.\ 38, 01187 Dresden, Germany.}}

\maketitle

\begin{abstract}
We consider the effect of Gaussian white noise on fast--slow dynamical systems
with one fast and two slow variables, containing a folded-node singularity. In
the absence of noise, these systems are known to display mixed-mode
oscillations, consisting of alternating large- and small-amplitude oscillations.
We quantify the effect of noise and obtain critical noise intensities above
which the small-amplitude oscillations become hidden by fluctuations.
Furthermore we prove that the noise can cause sample paths to jump away from
so-called canard solutions with high probability before deterministic orbits do.
This early-jump mechanism can drastically influence the local and global
dynamics of the system by changing the mixed-mode patterns.
\end{abstract}



\section{Introduction}
\label{sec:intro}

Our main focus of study are stochastic dynamical systems with multiple time
scales. In particular, we are going to study a special bifurcation (\lq\lq a
folded node\rq\rq) in a three-dimensional fast--slow stochastic differential
equation (SDE) with one fast and two slow variables. The detailed technical
discussion including all relevant definitions and precise statements and proofs
of our results starts in Section \ref{sec:fastslow}. In this section we want to
outline our motivation and state our main results in a non-technical way.
There are two main motivations for our work: 

\begin{itemize}
 \item[(M1)] We want to develop an analogue to the intricate deterministic
bifurcation theory for random dynamical systems by linking stochastic sample-path techniques and the well understood deterministic theory.

 \item[(M2)] A detailed analysis of noise effects in multi-scale stochastic
systems is often crucial in applications; in particular, many biological systems
have widely separated time scales and are influenced by various random effects.
\end{itemize}  
  
We are going to describe our two main motivations in more detail, starting
with (M2). Complex oscillatory patterns have been discovered in many different
applications. Chemical systems \cite{PetrovScottShowalter,DegnOlsenPerram,Koper}
and neuronal dynamics
\cite{RotsteinWechselbergerKopell,GuckCK3,BronsKrupaRotstein} provide ample
examples. Recent work has shown \cite{KuehnMMO,Izhikevich} that fast--slow
systems can be used to model a wide variety of oscillatory patterns. A
classification of local and global fast--slow \lq\lq mechanisms\rq\rq\ can be
used to analyze each pattern. Mixed-mode oscillations (MMO) alternate between
small-amplitude oscillations (SAOs) and large-amplitude oscillations (LAOs).
Figure \ref{fig:fig6} shows a typical MMO time series where the deterministic
time series has been generated by an MMO model proposed in
\cite{BronsKrupaWechselberger}. The time series shows alternations between $L=1$
LAOs and $s=7$ SAOs which is denoted as the MMO pattern $L^s=1^7$.

\begin{figure}[htbp]
\includegraphics[width=1\textwidth,clip=true]{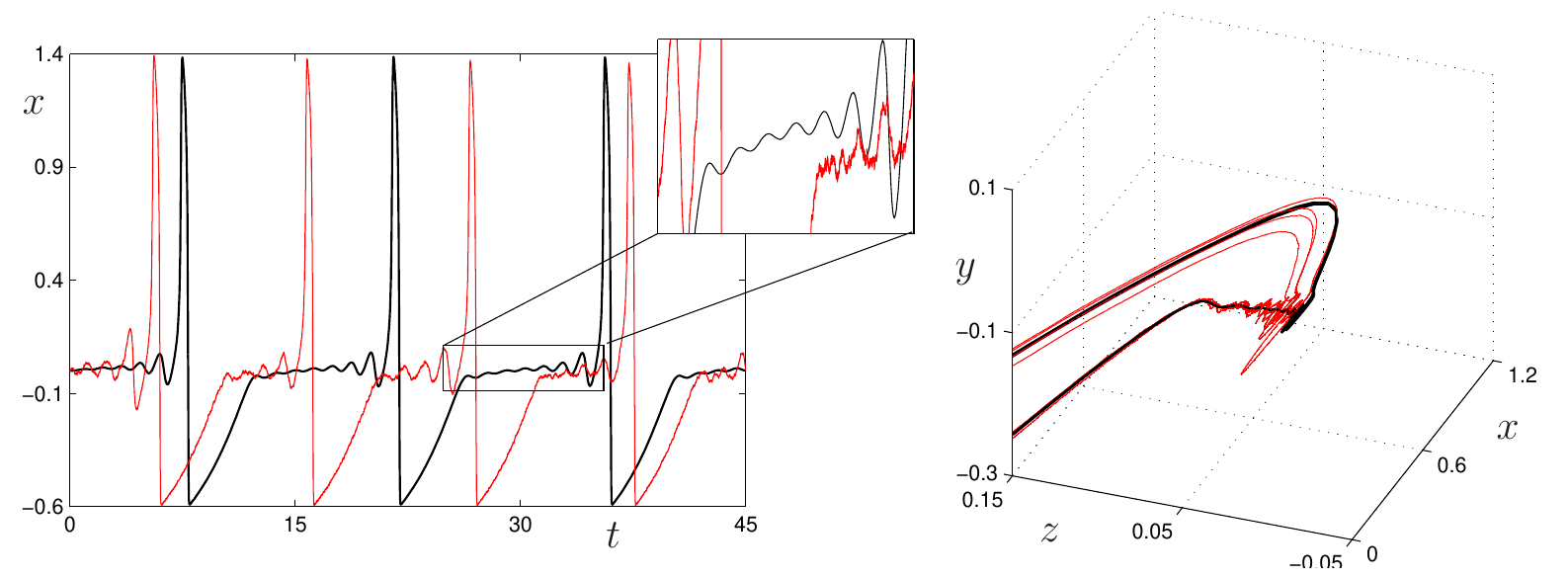}
	\caption{\label{fig:fig6} Comparison of a deterministic (black) and 
	stochastic (red) time trajectory. The left-half of the figure shows 
	the time series exhibiting a $1^7$ MMO; see also the zoom near the
SAOs. 
	The right part shows part of the trajectory in 
	three-dimensional phase space and illustrates the early jumps of the
stochastic 
	sample path.}	
\end{figure}

Although a deterministic model is able to explain a some experiments 
(see, e.g., experimental results by Hudson et al. \cite{HudsonHartMarinko}) it 
fails to accurately model realistic MMOs due to the presence of noise 
(see, e.g., the experiments by Dickson et al. \cite{Dicksonetal1}). 
Figure \ref{fig:fig6} also shows a sample path which is a stochastic 
version of the deterministic orbit perturbed by Gaussian white noise. 
Two important observations that can be made about the stochastic MMO pattern are
\begin{itemize}
\item	Part of the SAOs become indistinguishable from the random fluctuations
so that counting SAOs below a certain amplitude is impossible. 
\item	The stochastic sample path typically jumps before the deterministic 
solution makes an LAO. 
\end{itemize}
In this work, we are going to provide rigorous formulations and the proofs of
both observations.\\

Our motivation (M1) is to contribute to a better understanding of bifurcations in random dynamical systems by relating deterministic and stochastic methods for
fast--slow systems. Within the last two decades substantial progress on
deterministic fast--slow systems has been made. The analysis of hyperbolic fast
dynamics has been completed in a series of works by Fenichel \cite{Fenichel}
(see also \cite{Jones}) at the end of the 1970s; the theory focuses upon
perturbations of normally hyperbolic critical manifolds to nearby slow
manifolds. Near bifurcation points of the fast dynamics major developments in
the 1980s used nonstandard analysis
\cite{CallotDienerDiener78,BenoitCallotDienerDiener} and asymptotic methods
\cite{Eckhaus,BaerErneuxI,BaerErneuxII,MishchenkoRozov}. In the 1990s, two main
geometric methods were introduced into multiple time scale systems. Switching
between hyperbolic fast and slow dynamics was analyzed using the exchange lemma
\cite{JonesKopell,Tin}. The geometry of non-hyperbolic or singular dynamics has
been uncovered using the blow-up technique \cite{DRvdP,Dumortier1}. Since then
the blow-up method has been applied successfully for many singular-perturbation
problems \cite{KruSzm1,KruSzm2,KruSzm3,KruSzm4}. In particular, folded nodes
\cite{Wechselberger,WechselbergerFN,KrupaWechsFSN,MKKR_B} are a very interesting
case since they occur already in generic fast--slow systems with one fast and
two slow variables and have a highly nontrivial unfolding. Near a folded node
the canard phenomenon occurs, i.e., orbits stay near a repelling manifold for a
long time. This effect can generate MMOs
\cite{BronsKrupaWechselberger,KuehnMMO}; see also the discussion for (M2) above
and Figure~\ref{fig:fig6}. \\

Noise acting on a system with multiple time scales can induce new phenomena such
as early transitions~\cite{TM,SHA,SMC,JL,Kuske} and stochastic
resonance~\cite{BSV,Nicolis,McNW,Fox,GMSMP,GHM}. The mathematical theory of fast--slow stochastic differential equations has mainly been developed during the last decade and is still far from being complete. Classical work on random perturbations
of dynamical systems~\cite{FreidlinWentzell}, which mainly focused on
large-deviation aspects, can be applied to situations with a time-scale
separation exponentially large in the noise
intensity~\cite{Freidlin1,Freidlin2,ImkellerPavlyu02,HerrmannImkeller05,
HerrmannImkellerPeithmann2006}. A different approach, based on a detailed
description of sample-path properties, applies to situations with time-scale
separation and noise intensity of comparable magnitude~\cite{BG1,BG2,BG3,BG4}.
This method led to a general theory for the behaviour of sample paths near
normally hyperbolic invariant manifolds~\cite{BG6,BGbook}. Other approaches
include \cite{KabanovPergamenshchikov_2003}, which is based on moment
estimates, and \cite{SchmalfussSchneider},
which adopts the viewpoint of random dynamical systems in
the sense of~\cite{ArnoldSDE}.
The associated methods and results have important applications in climate
dynamics~\cite{BSV,Nicolis,Timmermann03,BG5}, the theory of critical
transitions~\cite{Schefferetal,KuehnCT}, classical and quantum atomic
physics~\cite{AllmanBetz2008,AllmanBetzHairer2010,AguilarBerglund08} and
neuroscience~\cite{Tuckwell,Longtin,Longtin2000,LaingLordBook,BG_neuro09}. In
particular, canards in the stochastic FitzHugh--Nagumo system describing the
action potential of neurons have been considered from the points of view of
large deviations~\cite{DevilleVandeneijndenMuratov2005,DossThieullen2009}, and
of convergence of sample-paths~\cite{Sowers08}. Stochastic MMOs have also been
considered in certain planar
systems~\cite{MuratovVanden-Eijnden,MuratovVanden-EijndenE} and in coupled
oscillators~\cite{YuKuskeLi}. 

The theory of stochastic differential equations with higher-dimensional
singularities and multiple slow variables is not yet as advanced. Here we make a
first step towards bridging this gap between the generic higher-dimensional
deterministic theory and stochastic sample-path analysis. The non-technical
statements of our two main results are:

\begin{enumerate}
\item[(R1)] Sample paths near a folded node stay inside a tubular neighbourhood
of an attracting deterministic solution. The neighbourhood is explicitly defined
by the covariance matrix of a linearized process. The relation between the noise
level, the time scale separation and a system parameter determines when small
oscillations near a folded node become indistinguishable from noisy
fluctuations. This relation can be calculated explicitly to lowest asymptotic
order.  
\item[(R2)] Sample paths near a folded node typically escape from a repelling
deterministic solution earlier than their deterministic counterparts. The typical escape
time can be determined rather precisely and depends on the same parameters
as the relation in (R1). The probability of observing atypical escape times can be shown to be small.
\end{enumerate}

Both results have important implications from theoretical and applied
perspectives. In particular, we show how to control stochastic sample paths
near a multi-dimensional bifurcation point. Therefore it is expected that the
methods we develop have a much wider applicability beyond folded nodes, e.g.\ to
singular Hopf bifurcations \cite{GuckenheimerSH} or other stochastic bifurcation
problems \cite{BGbook,ArnoldRPB}. The precise quantitative estimates on
the relation between noise level and a parameter controlling the number of 
SAOs are immediately useful in applications. Furthermore, the effect of early
jumps could potentially regularize the complicated flow maps near a folded node
\cite{GuckenheimerFNFSN,GuckenheimerScheper} and simplify the local--global
decomposition of return maps \cite{KuehnRetMaps}. \\

The paper is organized as follows. In Section \ref{sec:fastslow} we review the
necessary theory for deterministic fast--slow systems and fix the notation. In
Section \ref{sec:foldednode} we state the known results about folded nodes and
explain why they produce small-amplitude oscillations. In Section
\ref{sec:fn_new} we consider a variational equation around a special canard
solution, called the weak canard. The solution of the variational equation can
be transformed into a \lq\lq canonical form\rq\rq\ which allows us to prove a
result on the spacing of canard solutions  up cross-sections near or including
the folded node. The proof is postponed to
Appendix~\ref{appendix:canonical_form}. In Section \ref{sec:SDEfastslow} we
develop the main setup for stochastic fast--slow SDEs and recall a result on
attracting normally hyperbolic slow manifolds away from bifurcation points.
Several notations that we use throughout are introduced as well. In Section
\ref{sec:SDEfoldednode} we give the rigorous formulation of our main results (R1)--(R2) for stochastic folded nodes. The result (R1) on covariance tubes for the linearized process is
proved in Appendix \ref{appendix:covariance_tubes}. The influence of nonlinear
contributions is dealt with in Appendix \ref{appendix:proof_nonlinear_SDE}. The
result (R2) on early jumps is proven in Appendix \ref{appendix:proof_escape}.
Section \ref{sec:numerics} develops numerical simulations to visualize the
analytical results. We conclude by giving a summary of parameter regimes and
discussing the influence of early jumps on the global return mechanism and LAOs
in Section \ref{sec:discussion}. 

\subsubsection*{Acknowledgments}
It's a pleasure for the authors to thank Mathieu Desroches for inspiring discussions. 
N.B.\ and C.K.\ thank the CRC 701 at University of Bielefeld, B.G.\ thanks the
MAPMO at Universit\'e d'Orl\'eans and C.K.\ thanks Cornell University 
for hospitality and financial support.

\section{Fast--Slow Systems}
\label{sec:fastslow}

We are only going to give a brief introduction to multiple time scale dynamics. 
A detailed reference covering many more topics is currently 
being written \cite{KuehnBook}; other, excellent references are
\cite{MisRoz,Grasman} for 
asymptotic methods and \cite{Fenichel,ArnoldEncy,Jones} for geometric methods. 
Many important discoveries were first made using nonstandard analysis 
\cite{BenoitCallotDienerDiener,DienerDiener}; in particular, many results we 
review in Section \ref{sec:foldednode} were discovered by 
Beno\^{i}t \cite{Benoit4,Benoit1,Benoit5}. However, we are not going to use 
any nonstandard methods and focus on the geometric viewpoint.\\

A \texttt{fast-slow system} of ordinary differential equations (ODEs) is given
by
\be
\label{eq:gen_fast_slow}
\begin{array}{rcrcl}
\epsilon \frac{\6 x}{\dd s}&=& \epsilon \dot{x} &=& f(x,y,\mu,\epsilon)\;,\\
\frac{\6 y}{\6 s}&=& \dot{y} &=& g(x,y,\mu,\epsilon)\;,\\
\end{array}
\ee
where $(x,y)\in\R^m\times \R^n$ are phase-space coordinates, $\mu\in\R^p$ 
are parameters and $0<\epsilon \ll 1$ represents the ratio of time scales. 
We shall assume that $f,g$ are sufficiently smooth. By a rescaling we can 
change from the \texttt{slow time} $s$ to the \texttt{fast time}
$t=s/\epsilon$; 
this transforms \eqref{eq:gen_fast_slow} to 
\be
\label{eq:gen_fast_slow1}
\begin{array}{rcrcl}
\frac{\6x}{\6t}&=& x' &=& f(x,y,\mu,\epsilon)\;,\\
\frac{\6y}{\6t}&=& y' &=& \epsilon g(x,y,\mu,\epsilon)\;.\\
\end{array}
\ee 
\textit{Remark:} The more common notation for the slow time would be $\tau$ 
but we shall reserve $\tau$ for stopping times of stochastic processes; see
Section 
\ref{sec:SDEfastslow}.\\

The first step to analyze fast--slow systems is to consider the 
\texttt{singular limit} $\epsilon\ra 0$. From \eqref{eq:gen_fast_slow1} we obtain
\be
\label{eq:gen_fast_sub}
\begin{array}{rcl}
x' &=& f(x,y,\mu,0)\;,\\
y' &=& 0\;.\\
\end{array}
\ee 
which is an ODE for the fast variables $x$ where the slow variables $y$ act as
parameters. 
We call \eqref{eq:gen_fast_sub} the \texttt{fast subsystem} or \texttt{layer
equations}; 
the associated flow is called the \texttt{fast flow}. Considering $\epsilon\ra
0$ in 
\eqref{eq:gen_fast_slow} we find a differential--algebraic equation for the slow
$y$-variables
\be
\label{eq:gen_slow_sub}
\begin{array}{rcl}
0&=& f(x,y,\mu,0)\;,\\
\dot{y} &=& g(x,y,\mu,0)\;,\\
\end{array}
\ee  
called the \texttt{slow subsystem} or \texttt{reduced system}; the flow induced
by \eqref{eq:gen_slow_sub} is called the \texttt{slow flow}. The slow subsystem
is defined on the \texttt{critical manifold}
\benn
C_0:=\{(x,y)\in\R^{m+n}:f(x,y,\mu,0)=0\}\;.
\eenn 
Observe that $C_0$ can also be interpreted as a manifold of equilibria for the
fast subsystem. Note also that $C_0$ does not have to be a manifold
\cite{KruSzm4} but we only consider the manifold case in this paper. If the
Jacobian matrix $(D_xf)(p)$ has maximal rank at $p\in C_0$ then the implicit-function theorem describes $C_0$ locally as a graph
\benn
h_0:\R^n\ra \R^m, \qquad f(h_0(y),y,\mu,0)=0
\eenn  
near $p$. This allows us to write the slow subsystem more concisely as 
\be
\label{eq:gen_sf_graph}
\dot{y}=g(h_0(y),y,\mu,0)\;.
\ee
We can strengthen the assumption on $(D_xf)(p)$ and require that it is a
hyperbolic matrix, i.e., $(D_xf)(p)$ has no eigenvalues with zero real part. In
this case we say that $C_0$ is \texttt{normally hyperbolic} at $p$. If all the
eigenvalues of $(D_xf)(p)$ have negative (positive) real parts we say that $C_0$
is \texttt{attracting (repelling)} with respect to the fast variables. The
following theorem shows that normal hyperbolicity is the key regularity
assumption for fast--slow systems. 

\begin{thm}[\texttt{Fenichel's Theorem} \cite{Fenichel}]
\label{thm:fenichel1}
Suppose $M_0$ is a compact normally hyperbolic submanifold
(possibly with boundary) of the critical manifold $C_0$ and 
that $f,g \in C^r$, $1\leq r < \infty$. Then for $\epsilon > 0$ sufficiently 
small the following holds:  

\begin{itemize}
\item[(F1)] There exists a locally invariant manifold $M_\epsilon$
diffeomorphic 
to $M_0$. Local invariance means that $M_\epsilon$ can have boundaries through 
which trajectories enter or leave.  
\item[(F2)] $M_\epsilon$ has a Hausdorff distance of $\cO(\epsilon)$ from $M_0$.
\item[(F3)] The flow on $M_\epsilon$ converges to the slow flow as $\epsilon \to
0$.
\item[(F4)] $M_\epsilon$ is $C^r$-smooth and can locally be given as a graph 
$h_\epsilon:\R^n\ra \R^m$.
\item[(F5)] $M_\epsilon$ is normally hyperbolic with the same stability 
properties with respect to the fast variables as $M_0$. 
\item[(F6)] $M_\epsilon$ is usually not unique. In regions that remain at a 
fixed distance from the boundary of $M_\epsilon$, all manifolds satisfying 
(F1)--(F5) lie at a Hausdorff distance $\cO(\e^{-K/\epsilon})$ from each other 
for some $K > 0$ with $K = \cO(1)$. 
\end{itemize} 
\end{thm} 

We call a perturbed manifold $M_\epsilon$ a \texttt{slow manifold}. Sometimes
we refer to \lq\lq the slow manifold\rq\rq\ despite the non-uniqueness (F6) as it
will be often irrelevant which of the $\cO(\e^{-K/\epsilon})$-close manifolds
we pick. 

A simple example where normal hyperbolicity fails is given by
\be
\label{eq:simple_fold_ex}
\begin{array}{rcl}
\epsilon \dot{x}&=& y-x^2\;,\\
\dot{y} &=& \mu-x\;.\\
\end{array}
\ee  
The critical manifold $C_0=\{(x,y)\in\R^2:y=x^2\}$ splits into three parts 
$C_{0}=C_0^a\cup \{(0,0)\}\cup C_0^r$ where
\benn
C^a_0=C_0\cap \{x>0\}\;, \qquad \text{and} \qquad C^r_0=C_0\cap\{x<0\}\;.  
\eenn
$C^a_0$ is attracting and $C^r_0$ is repelling. At $(x,y)=(0,0)$ the 
critical manifold is not normally hyperbolic and has a generic 
\texttt{fold singularity} \cite{KruSzm1}; observe that $(x,y)=(0,0)$ is 
a fold (or saddle--node) bifurcation of the fast subsystem. Figure 
\ref{fig:fig1} illustrates the dynamics near the fold point of 
\eqref{eq:simple_fold_ex}.\\ 

\begin{figure}[htbp]
\includegraphics[width=1\textwidth,clip=true]{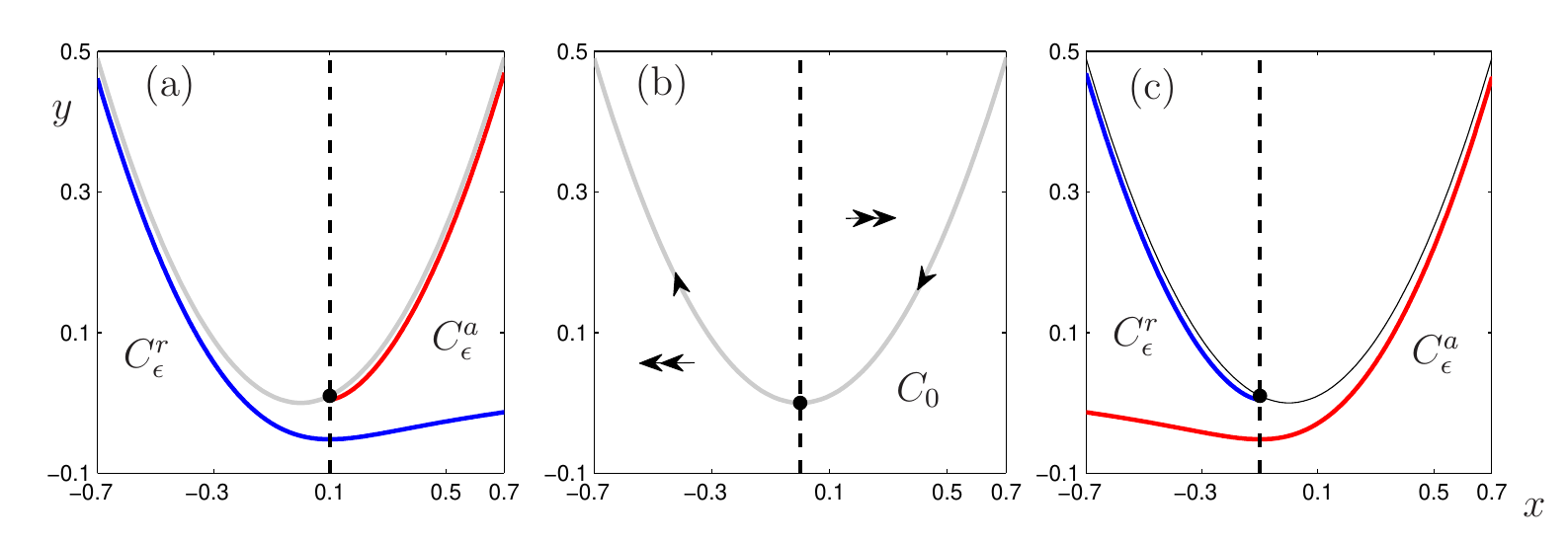}
	\caption{\label{fig:fig1} Planar fold near a singular Hopf bifurcation
of \eqref{eq:simple_fold_ex}; for (a) and (c) we have fixed $\epsilon=0.05$. (a)
$\mu=0.1$: The equilibrium point $(x_0,y_0)=(\mu,\mu^2)$ is determined as the
intersection of $C_0$ (grey) and the nullcline $\{x=\mu\}$ (dashed black). The
equilibrium is stable and the slow manifolds $C^a_\epsilon$ (red) and
$C^r_\epsilon$ (blue) do not interact. (b) $\mu=0$: Only the slow flow on $C_0$
(single arrow) and the fast flow (double arrows) are shown. $C_0$ coincides with
a maximal singular canard. (c) $\mu=-0.1$: After the Hopf bifurcation the slow
manifolds have \lq\lq exchanged sides\rq\rq\ suggesting an intersection for some
$\mu$ near $0$.}	
\end{figure}

To calculate the slow subsystem on $C_0$ we could consider the two graphs 
$x=h_0(y)=\pm \sqrt{y}$ as suggested by \eqref{eq:gen_sf_graph}. For 
\eqref{eq:simple_fold_ex} it is more convenient to differentiate $y=x^2$ 
implicitly with respect to $s$. This gives 
\benn
\dot{y}=2x\dot{x}\qquad \Rightarrow \quad \dot{x}=\frac{\mu-x}{2x}\;,
\eenn  
which shows that the slow flow is undefined at $(0,0)$ if $\mu\neq 0$. 
Fenichel's Theorem provides slow manifolds $C^a_\epsilon$ and $C^r_\epsilon$. 
A major step in the theory of fast--slow systems was to consider the dependence of the dynamics
of \eqref{eq:simple_fold_ex} on the value of $\mu$ 
\cite{DRvdP,BaerErneuxI,KruSzm2}.  Note that for $\mu=0$ the 
slow flow is well-defined and there is a special trajectory that passes from 
$C^a_0$ to $C^r_0$ and that a singular Hopf bifurcation \cite{KruSzm2,Braaksma} 
occurs for $\mu=0$ and $0<\epsilon\ll 1$. The slow manifolds can be extended
under 
the flow into the fold point region. Comparing Figure \ref{fig:fig1}(a) to
Figure 
\ref{fig:fig1}(c) we expect that there is a parameter value $\mu$ for which the 
slow manifolds intersect/coincide. This intersection marks what has become 
known as a canard explosion \cite{KruSzm2,KuehnCanLya}.\\

More generally, suppose that $\gamma_\epsilon$ is a trajectory of a fast--slow 
system \eqref{eq:gen_fast_slow}.
Then we call
$\gamma_\epsilon$ a \texttt{maximal canard} if it lies in the intersection 
of an attracting and a repelling slow manifold; for $\epsilon=0$ we also
refer to $\gamma_0$ as a \texttt{maximal singular canard}. Canards in planar 
fast--slow systems are of codimension one whereas for higher-dimensional 
systems we do not need an additional parameter. In the next section we 
are going to focus on canards in three dimensions.     

\section{Folded Nodes}
\label{sec:foldednode}

A general three-dimensional fast--slow system with one fast variable and two 
slow variables can be written as
\be
\label{eq:3D_general}
\begin{array}{rcl}
\epsilon \dot{x} &=& f(x,y,z,\mu,\epsilon)\;,\\
\dot{y} &=& g_1(x,y,z,\mu,\epsilon)\;,\\
\dot{z} &=& g_2(x,y,z,\mu,\epsilon)\;.\\
\end{array}
\ee
We assume that the critical manifold $C_0=\{(x,y,z)\in\R^3:f(x,y,z,\mu,0)=0\}$ 
of \eqref{eq:3D_general} is a folded surface near the origin; suitable 
non-degeneracy conditions \cite{KuehnMMO,Wechselberger} are
\be
\label{eq:fold_3D}
\begin{array}{ll}
f(0,0,0,\mu,0)=0\;,\qquad & f_x(0,0,0,\mu,0)=0\;,\\
f_y(0,0,0,\mu,0)\neq 0\;,\qquad & f_{xx}(0,0,0,\mu,0)\neq 0\;,\\ 
\end{array}
\ee 
where subscripts denote partial derivatives.\\

\begin{figure}[htbp]
\includegraphics[width=1\textwidth,clip=true]{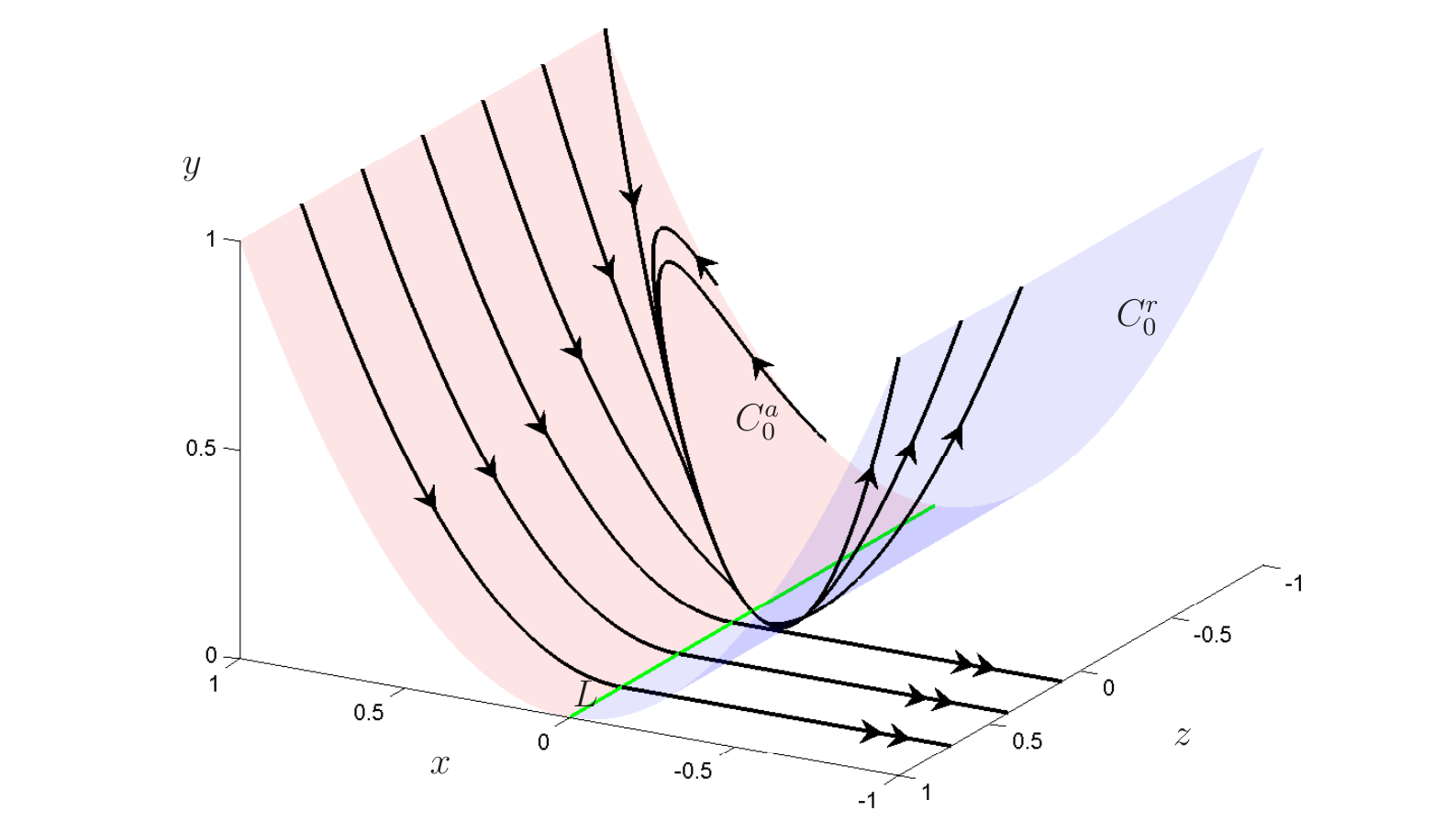}
	\caption{\label{fig:fig2} Singular limit $\epsilon=0$ for the normal
form \eqref{eq:main_ex_eps} with $\mu=0.15$. The attracting manifold $C_0^a$
(red), the fold line $L$ (green) and the repelling manifold $C_0^r$ (blue)
partition the critical manifold. Trajectories of the slow (fast) subsystem are
indicated by single (double) arrows.}	
\end{figure}

The critical manifold again decomposes into three parts
\benn
C_0=C^r_0\cup L \cup C^a_0\;,
\eenn 
where $C^r_0=C_0\cap \{f_x>0\}$ is repelling, $C^a_0=C_0\cap \{f_x<0\}$ is 
attracting and $L=C_0\cap \{f_x=0\}$ is the curve of fold points; see 
Figure \ref{fig:fig2}. Note that the assumption $f_y(0,0,0,\mu,0)\neq 0$ in 
\eqref{eq:fold_3D} implies that the fold curve $L$ can be locally 
parametrized by $z$. To obtain the slow subsystem we again differentiate
$f(x,y,z,\mu,0)=0$ implicitly with respect to $s$
\benn
\dot{x}f_x+\dot{y}f_y+\dot{z}f_z=0\;.
\eenn 
This implies that the slow subsystem is 
\be
\label{eq:sf_reg_3D}
\begin{array}{rcl}
f_x\dot{x}&=&-f_yg_1-f_zg_2\;,\\
\dot{z}&=& g_2\;,\\
\end{array}
\ee
where all functions are evaluated for $p=(x,y,z)\in C_0$ and $\epsilon=0$. 
On $L$ the ODE \eqref{eq:sf_reg_3D} is singular but we can 
rescale time $s\mapsto -s/f_x$ to obtain the \texttt{desingularized 
slow subsystem}
\be
\label{eq:sf_desing_3D}
\begin{pmatrix} \dot{x} \\ \dot{z} \\ \end{pmatrix}=
\left.\begin{pmatrix} f_yg_1+f_zg_2 \\ -f_xg_2 \\
\end{pmatrix}\right|_{p\in C_0}. 
\ee  
Note that the time rescaling has reversed the orientation of trajectories 
of \eqref{eq:sf_reg_3D} on $C^r_0$ and that \eqref{eq:sf_desing_3D} is a 
well-defined planar ODE. We define 
\benn
l(z):=(f_yg_1+f_zg_2)|_{p\in L}
\eenn   
and make the assumptions that
\bea
l(0)&=&0 \label{eq:lfunc1}\\
l(z)&\neq& 0\;, \quad \text{for $z\neq 0$\;.} \label{eq:lfunc2}
\eea
Observe that \eqref{eq:lfunc1} and $f_x(0,0,0,\mu,0)=0$ imply that 
$(x,z)=(0,0)$ is an equilibrium point for \eqref{eq:sf_desing_3D} that 
lies on the fold curve. We say that $(x,y,z)=(0,0,0)$ is a folded 
singularity. Points $(x,y,z\neq0)\in L$ are called jump points as 
trajectories have to make a transition from the slow to the fast 
flow at these points; the condition \eqref{eq:lfunc2} is also called 
the \texttt{normal switching condition}. Generic folded singularities 
can be classified according to their equilibrium type into 
\texttt{folded saddles}, \texttt{folded foci} and \texttt{folded nodes} 
\cite{Wechselberger,Benoit1}. Folded nodes are the most interesting 
folded singularities. Without loss of generality we may assume that 
the folded node is stable for \eqref{eq:sf_desing_3D} with associated 
eigenvalues $\lambda_1, \lambda_2$ for the linearization of 
\eqref{eq:sf_desing_3D}. 

\begin{prop}[\cite{BronsKrupaWechselberger,WechselbergerFN}]
\label{prop:nform_deter}
Suppose that \eqref{eq:3D_general} satisfies \eqref{eq:fold_3D}, 
\eqref{eq:lfunc1}, \eqref{eq:lfunc2} and that we are in the 
folded-node scenario. Then there exist a smooth 
coordinate change and a smooth change of time which bring \eqref{eq:3D_general} 
near $(x,y,z)=(0,0,0)$ into the form
\be
\label{eq:main_ex_eps_old}
\begin{array}{rcl}
\epsilon \dot{x} &=& y-x^2+\cO(yx^2,x^3,xyz)+\epsilon \cO(x,y,z,\epsilon)\;,\\
\dot{y} &=& -(\mu+1)x-z+\cO(y,\epsilon,(x+z)^2)\;,\\
\dot{z} &=& \frac{\mu}{2}\;,\\
\end{array}
\ee
with $\lambda_1=-\mu$ and $\lambda_2=-1$ being the eigenvalues for the linearization of 
\eqref{eq:sf_desing_3D}. 
\end{prop}   

We shall show in Section \ref{sec:fn_new} that the terms in
\eqref{eq:main_ex_eps_old} 
denoted by $\cO(\cdot)$ are indeed higher-order for the analysis near the folded
node. 
Hence we can work with the normal form
\be
\label{eq:main_ex_eps}
\begin{array}{rcl}
\epsilon \dot{x} &=& y-x^2\;,\\
\dot{y} &=& -(\mu+1)x-z\;,\\
\dot{z} &=& \frac{\mu}{2}\;.\\
\end{array}
\ee 
The critical manifold of \eqref{eq:main_ex_eps} is
\benn
C_0=\{(x,y,z)\in\R^3:y=x^2\}\;.
\eenn
It splits into three components
\benn
C_0=C^a_0\cup L \cup C^r_0\;,
\eenn
where $C^a_0=C_0\cap \{x>0\}$ is attracting, $C^r_0=C_0\cap \{x<0\}$ is repelling, 
and $L$ is now a line of fold points. Parametrizing over the slow variables, we 
can also write $C^{a,r}_0=\{x=\pm \sqrt{y}\}$, i.e., $h_0^{a,r}(y,z)=\pm\sqrt{y}$. 

\begin{figure}[htbp]
\includegraphics[width=1\textwidth,clip=true]{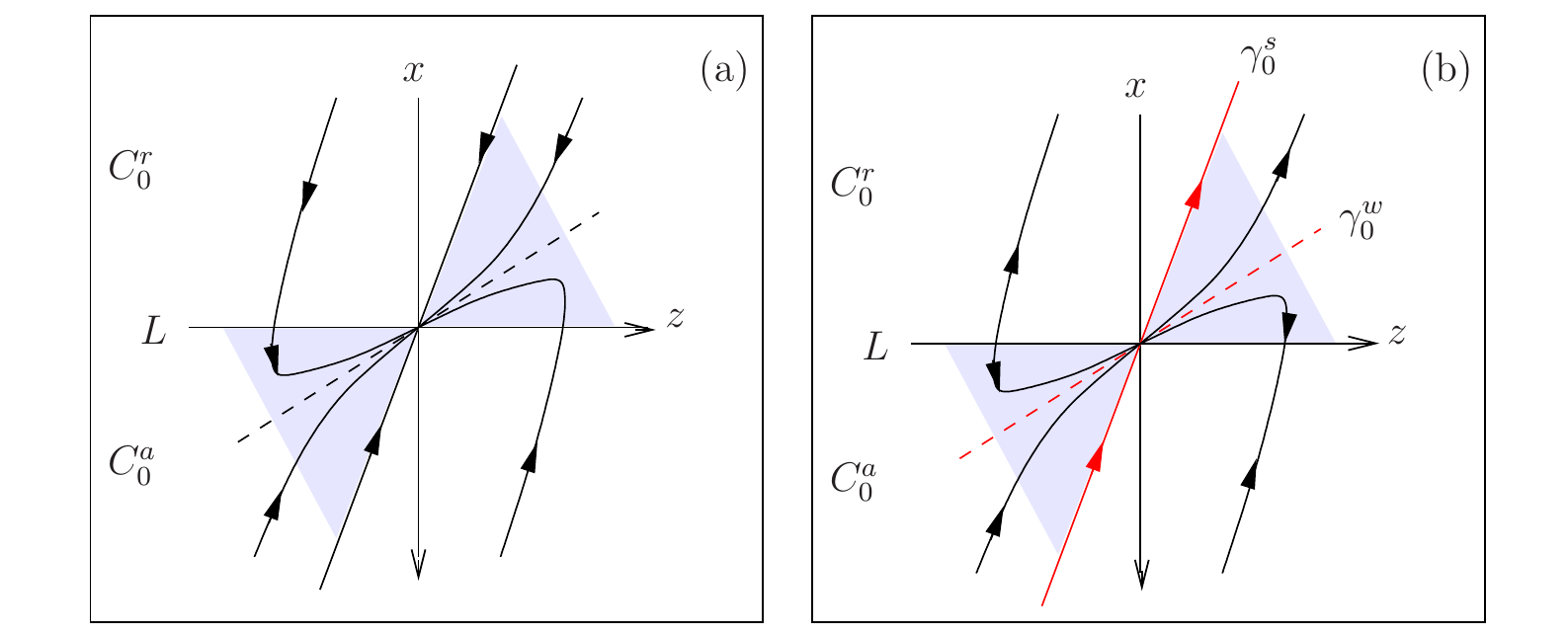}
	\caption{\label{fig:fig3}(a) The desingularized slow flow
\eqref{eq:slowflow_desing_ex} is sketched for some $\mu>0$ with a stable node at
the origin. (b) The slow flow \eqref{eq:slowflow_ex} is illustrated. The strong
eigendirection (solid red) defines the strong singular canard $\gamma^s_0$; the
weak eigendirection (dashed red) defines the weak singular canard
$\gamma^w_0$.}	
\end{figure}

Again differentiating $y=x^{2}$ implicitely, we see that the slow subsystem is given by
\be
\label{eq:slowflow_ex}
\begin{array}{rcl}
2x\dot{x}&=& -(\mu+1)x-z\;,\\
\dot{z}&=& \frac\mu2\;,\\
\end{array}
\ee
and the desingularized slow subsystem (see also Figure \ref{fig:fig3}(a)) is of the form
\be
\label{eq:slowflow_desing_ex}
\begin{array}{lcl}
\dot{x}&=& -(\mu+1)x-z\;,\\
\dot{z}&=& \mu x\;.\\
\end{array}
\ee
The system \eqref{eq:slowflow_desing_ex} is linear with an equilibrium point 
at $(x,z)=(0,0)$, and the eigenvalues are $(\lambda_1,\lambda_2):=(-1,-\mu)$. We 
assume  from now on that 
\benn
\mu\in(0,1)
\eenn 
so that $(0,0)$ is a stable node for the desingularized slow subsystem 
\eqref{eq:slowflow_desing_ex}. Hence we also denote the eigenvalues as 
\benn
\lambda_1=-1=:\lambda_s \qquad\text{and} \qquad \lambda_2=-\mu=:\lambda_w
\eenn
to emphasize the strong and weak eigendirections. Note that
$\mu=\lambda_w/\lambda_s$ 
precisely represents the ratio of eigenvalues and attains all resonances
$\mu^{-1}\in \N$ 
for $\mu\in(0,1)$. The associated (unnormalized) eigenvectors are
\be
\gamma^s_0=(-1/\mu,1)^T\qquad \text{and} \qquad \gamma_0^w=(-1,1)^T\;,
\ee 
which also represent directions for two maximal singular canards; see Figure
\ref{fig:fig3}. 
Observe that the \texttt{singular strong canard} $\gamma^s_0$ and $L$ bound a
funnel region 
on $C^a_0$ of trajectories that all flow into the folded node; see Figure
\ref{fig:fig3}. 
The funnel region has an opening angle $\cos^{-1}(\mu/\sqrt{1+\mu^2})$ which
converges to 
$\pi/2$ as $\mu\ra 0$. The funnel on $C^a_0$ is located in the
$\{x>0,z<0\}$-quadrant and 
the \texttt{singular weak canard} $\gamma_0^w$ is given by the anti-diagonal
$\{z=-x\}$.

\begin{thm}[\cite{Wechselberger,Benoit1,Benoit4,WechselbergerFN}]
\label{thm:canardsR3}
Suppose \eqref{eq:3D_general} has a generic folded node (i.e.\ Proposition 
\ref{prop:nform_deter} applies). Then for $\epsilon > 0$ sufficiently small 
the following holds:
\begin{itemize}
\item[(C1)] The singular strong canard $\gamma_0^s$ always perturbs to a
maximal canard $\gamma^s_\epsilon$. If $\mu^{-1} \not\in \N$, then the singular
weak canard $\gamma^w_0$ also perturbs to a maximal canard $\gamma^w_\epsilon$. 
We call $\gamma^s_\epsilon$ and $\gamma^w_\epsilon$ \texttt{primary canards}.  
 
\item[(C2)] Suppose $k>0$ is an integer such that 
\benn
2k+1<\mu^{-1}<2k+3\qquad \text{and} \qquad \mu^{-1}\neq 2(k+1)\;. 
\eenn
Then, in addition to $\gamma^{s,w}_\epsilon$, there are $k$ other maximal
canards, 
which we call \texttt{secondary canards}. 
\item[(C3)] The secondary canards converge to the strong primary canard as
$\epsilon \ra 0$.
\item[(C4)] The primary weak canard of a folded node undergoes a transcritical
bifurcation for odd $\mu^{-1} \in \N$ and a pitchfork bifurcation for even 
$\mu^{-1} \in \N$. 
\end{itemize}
\end{thm} 

We emphasize the results (C1)--(C3) which will be of major importance for our 
stochastic analysis; (C4) describes the behaviour near resonances and will not 
be considered here. The next theorem provides a geometric viewpoint for the 
generation of maximal canards near a folded node. We say that a \texttt{twist} 
corresponds to a half rotation (i.e.\ a rotation by an angle of $\pi$). 

\begin{thm}[\cite{Wechselberger,WechselbergerFN}]
\label{thm:rotation_det}
Assume\/ $2k+1<\mu^{-1}<2k+3$, for some $k \in \N$, and $\mu^{-1}\neq 2(k+1)$. Then
the following holds: 
\begin{itemize}
 \item[(C5)] The primary strong canard $\gamma^s_0$ twists once around the
primary weak 
 canard $\gamma^w_\epsilon$.
 \item[(C6)] The $j$-th secondary canard $\gamma^j_\epsilon$, $1 \leqs j \leqs
k$, twists 
 $2 j+ 1$ times around the primary weak canard $\gamma^w_\epsilon$. 
 \item[(C7)] The twisting/rotation occurs in an $\cO(\sqrt\epsilon)$
neighborhood of the 
 folded node for \eqref{eq:3D_general}. 
\end{itemize} 
\end{thm}

In particular, the slow manifolds $C^a_\epsilon$ and $C^r_\epsilon$ start 
to spiral near the folded node creating transversal intersections away from
resonances. 
For visualizations of these manifolds in several different contexts see 
\cite{DesrochesKrauskopfOsinga3,KuehnMMO,DesrochesKrauskopfOsinga}.\\

\begin{figure}[htbp]
\includegraphics[width=1\textwidth,clip=true]{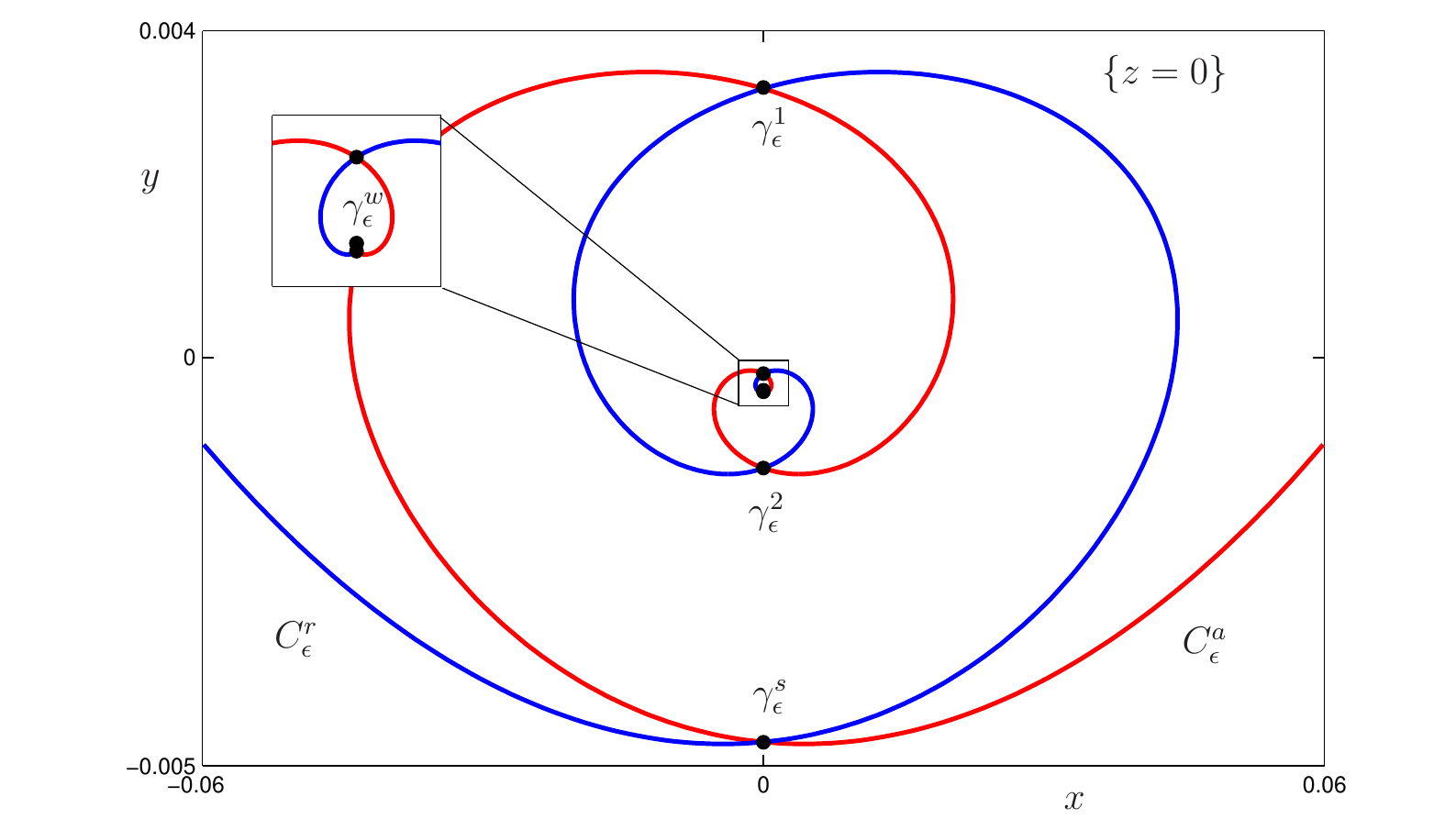}
	\caption{\label{fig:fig4} Canards and slow manifolds near a folded node
for 
	\eqref{eq:main_ex_eps} with $(\mu,\epsilon)=(0.08,0.01)$ on the
cross-section 
	$\{z=0\}$. The primary canard $\gamma_\epsilon^{s}$ and the first two
secondary canards 
	$\gamma^{1,2}_\epsilon$ are labeled. We also show a zoom near the primary 
	weak canard $\gamma_0^w$. All maximal canards are
indicated by 
	black dots.}	
\end{figure}

In Figure \ref{fig:fig4} we show the slow manifolds for \eqref{eq:main_ex_eps}
near 
the folded node on the cross-section $\{z=0\}$. The manifolds have been computed
by forward 
integration and using the symmetry
\be
\label{eq:symmetry}
(x,y,z,s)\mapsto (-x,y,-z,-s)\;.
\ee
The center of rotation is the weak canard $\gamma_0^w$. Since $\mu=0.08$ we
know by (C2) that there are five secondary canards in Figure \ref{fig:fig4}.
Five 
intersections are indeed detected numerically but $\gamma_0^{4,5}$ are very
close to
$\gamma_0^w$ on $\{z=0\}$.  All secondary canards 
$\gamma_0^j$ approach $\gamma_0^w$ when $z<0$ near the folded-node region on 
$C^a_\epsilon$. The canards move away from each other for $z>0$ (see Figure 
\ref{fig:fig2}). The next theorem shows that the maximal canards organize the 
rotational properties of trajectories passing through a folded-node region. 

\begin{thm}[\cite{BronsKrupaWechselberger}]
\label{thm:sectors}
Fix two sections 
\benn
\begin{array}{lcll}
\Sigma^1&:=&\{(x,y,z)\in\R^3:y=K_1\}\qquad & \text{for some $0<K_1=\cO(1)$}\\  
\Sigma^2&:=&\{(x,y,z)\in\R^3:y=K_2\epsilon\} \qquad & \text{for some
$0<K_2=\cO(1)$}\\
\end{array}
\eenn
for the system \eqref{eq:main_ex_eps}. Consider the intersection points 
of maximal canards in $\Sigma^1\cap C^a_\epsilon$. Let $k$ be the number 
of secondary canards. Then, for $\epsilon>0$ sufficiently small, the 
following holds:
\begin{itemize}
 \item[(C8)] The secondary canards are $\cO(\epsilon^{(1-\mu)/2})$ close to 
 the primary strong canard.
 \item[(C9)] There exist $(k+1)$ so-called \texttt{sectors of rotation} $I_j$,
$1\leqs j\leqs k+1$, 
 between the two primary canards labeled in increasing order starting from the 
 strong primary canard. The size of the sectors $I_j$ for $1\leqs j\leqs k$ 
 is $\cO(\epsilon^{(1-\mu)/2})$ while the size of the sector $I_{k+1}$ is
$\cO(1)$.
 \item[(C10)] The Poincar\'{e} map from $\Sigma^1$ to $\Sigma^2$ is a
contraction 
 with rate $\cO(\epsilon^{(1-\mu)/(2\mu)})$.
 \item[(C11)] All maximal canards are separated by $\cO(\sqrt\epsilon)$ 
 in their $z$-coordinate on $\Sigma^2$.
\end{itemize}   
\end{thm}

\noindent
\textit{Remark:} We note that the results from Theorems \ref{thm:canardsR3}, 
\ref{thm:rotation_det} and \ref{thm:sectors} also extend to
higher-dimensional 
fast--slow systems with at least two slow variables and at least one fast
variable 
\cite{Wechselberger1,BronsKrupaWechselberger} but that they do not provide a 
detailed analysis of canards beyond the section $\Sigma^2$.\\
 
Theorem \ref{thm:sectors} provides sectors of rotation that organize the
twisting 
of trajectories near the folded node. Once we know which sector an orbit enters 
we can predict the number of oscillations. Note that the oscillations can be 
classified as \lq\lq small oscillations\rq\rq\ due to (C7). Global returns can induce 
so-called \texttt{mixed-mode oscillations (MMOs)} which are found in a wide 
variety of applications; see \cite{KuehnMMO} for a review of MMO mechanisms 
in multiple time scale systems.
 
\section{Canard Spacing}
\label{sec:fn_new}

Theorem \ref{thm:sectors} describes the spacing of maximal canards away from
the 
folded-node region. Since we are also interested in their spacing on the 
cross-section $\{z=0\}$ depending on $\mu$ we need a refined analysis near 
the folded node. The key component in the proofs of Theorems
\ref{thm:canardsR3}, 
\ref{thm:rotation_det} and \ref{thm:sectors} is a rescaling of 
\eqref{eq:main_ex_eps} near the folded node  
\be
\label{eq:rescale1}
(x,y,z,s)=\left(\sqrt\epsilon \bar{x}, \epsilon \bar{y}, 
\sqrt\epsilon \bar{z}, \sqrt\epsilon \bar{s}\right) 
\ee 
which can also be interpreted as a blow-up transformation 
\cite{Wechselberger,WechselbergerFN,BronsKrupaWechselberger}. We 
shall not introduce the blow-up method here but restrict ourselves 
to the analysis of the rescaled system
\be
\label{eq:main_ex_eps1}
\begin{array}{rcl}
\dot{\bar{x}} &=& \bar{y}-\bar{x}^2+\cO(\sqrt\epsilon)\;,\\
\dot{\bar{y}} &=& -(\mu+1)\bar{x}-\bar{z}+\cO(\sqrt\epsilon)\;,\\
\dot{\bar{z}} &=& \frac{\mu}{2}\;.\\
\end{array}
\ee
Therefore the $\cO(\cdot)$-terms in \eqref{eq:main_ex_eps_old} are 
indeed of higher order for the analysis near the folded node. Neglecting 
the small $\epsilon$-dependent terms and dropping the overbars in 
\eqref{eq:main_ex_eps1} for notational convenience yields  
\be
\label{eq:main_ex}
\begin{array}{rcl}
\dot{x}&=& y-x^2\;,\\
\dot{y}&=& -(\mu+1)x-z\;,\\
\dot{z}&=& \frac\mu2\;.\\
\end{array}
\ee
The ODEs \eqref{eq:main_ex} are our main focus of study in this section. 
Note that (C11) in Theorem~\ref{thm:sectors} implies that the maximal 
canards are all $\cO(1)$-finitely separated for \eqref{eq:main_ex} when they 
are an $\cO(1)$-distance away from $(x,y,z)=(0,0,0)$. Observe that we can 
always solve the last equation 
\benn
z(s)=\frac{\mu}{2}(s-s_0)+z_0
\eenn 
where $z_0=z(s_0)$ denotes the initial $z$-coordinate of the trajectory at 
the initial time $s=s_0$. Hence we can view $z$ as a time variable and 
re-write \eqref{eq:main_ex} if necessary as a planar non-autonomous ODE
\be
\label{eq:main_ex_z}
\begin{array}{rcl}
\mu\frac{\6x}{\6z}&=& 2y-2x^2\;,\\
\mu\frac{\6y}{\6z}&=& -2(\mu+1)x-2z\;.\\
\end{array}
\ee
Our first goal is to quantify the intersections of canard solutions with 
the section $\{z=0\}$. Therefore we are going to focus on the analysis 
of orbits arising as perturbations of slow subsystem trajectories inside 
the funnel on $C^a_0$ and assume
\benn
z_0\leqs z\leqs 0\;.
\eenn 
Let $(x^*(z),y^*(z))$ be any solution of \eqref{eq:main_ex_z} and set
\benn
u=(u_1,u_2):=(x-x^*,y-y^*)
\eenn
to derive the \texttt{variational equation}
\be
\label{eq:main_ex_var1}
\begin{array}{rcl}
\mu\frac{\6u_1}{\6z}&=& -4u_1x^*+2u_2-2u_1^2\;,\\
\mu\frac{\6u_2}{\6z}&=& -2(\mu+1)u_1\;.\\
\end{array}
\ee

A key observation by Beno\^it~\cite{Benoit2} was that there are some special solutions 
to \eqref{eq:main_ex}.

\begin{lem}[\cite{Benoit2}]
\label{lem:Benoit}
The ODE \eqref{eq:main_ex} admits two polynomial solutions
\be
\label{eq:poly_Benoit}
(x(s),y(s),z(s))=\left(\frac{\lambda}{2}s,\frac{\lambda^2}{4}s^2+
\frac{\lambda}{2},\frac{\mu}{2}s\right)
\ee 
with $\lambda\in\{-\mu,-1\}$ corresponding to the two primary singular canards.
\end{lem} 

Lemma \ref{lem:Benoit} can be checked by direct differentiation of 
\eqref{eq:poly_Benoit}. We know that the weak canard is the center of 
rotation, and from \eqref{eq:poly_Benoit} with $\lambda=-\mu$ we find that
the variational equation around the weak canard is given by 
\be
\label{eq:main_ex_var2}
\mu\frac{\6u}{\6z}=
\left(\begin{array}{cc} 4 z & 2 \\ -2(\mu+1) & 0\\ \end{array}\right)
\left(\begin{array}{c}u_1 \\ u_2\\ \end{array}\right) + 
\left(\begin{array}{c} -2u_1^2 \\ 0\\\end{array}\right)\;.
\ee
We are interested in the detailed interaction of other maximal canards 
with the weak canard. 

\begin{prop}
\label{prop:attract}
If $\mu>0$ is sufficiently small and $z<0$ is bounded away from $0$ 
then solutions $u=u(z)$ of \eqref{eq:main_ex_var2} are attracted 
exponentially fast to $\{u_1=0=u_2\}$. 
\end{prop}

\begin{proof}
Augmenting the variational equation \eqref{eq:main_ex_var2} by $\dot{z}=0$ 
gives an autonomous fast--slow system with two fast variables $u=(u_1,u_2)$ 
and one slow variable $z$ (since $\mu$ is sufficiently small). The critical 
manifold is  
\benn
\{(u_1,u_2,z)\in\R^3:u_1=0,u_2=0\}\;.
\eenn
Computing the linearization with respect to the fast variables for $\mu=0$ gives a matrix 
\benn
\left(\begin{array}{cc} 4 z & 2 \\ -2 & 0\\ \end{array}\right)
\eenn
with eigenvalues $2(z\pm\sqrt{z^2-1})$. Hence if $z<0$ both eigenvalues have 
negative real parts and the result follows from Fenichel's Theorem.
\end{proof}

For $0<\mu\ll 1$, Proposition \ref{prop:attract} allows us to reduce the 
study of the nonlinear variational equation \eqref{eq:main_ex_var2} to a linear 
one by dropping the higher-order term $-2u_1^2$. This yields the linear 
non-autonomous ODE
\be
\label{eq:main_ex_var3}
\mu\frac{\6u}{\6z}=
\underbrace{\left(\begin{array}{cc} 
4 z & 2 \\ -2(\mu+1) & 0\\ \end{array}\right)}_{=:A(z)} u=A(z)u\;.
\ee
In particular, we must show what happens to solutions $u=u(u_1,u_2)$ 
near the weak canard and when $z$ is not bounded away from $0$.\\ 

\noindent 
\textit{Remark:} The variational equation \eqref{eq:main_ex_var3} has 
been analyzed \cite{Benoit2,Wechselberger} by re-writing it as a 
second-order equation
\be
\label{eq:Weber}
\frac{\6^2u_1}{\6s^2}-\mu s\frac{\6u_1}{\6s}+u_1=0\;.
\ee
Beno\^{i}t \cite{Benoit2} observed that using the time rescaling 
$s=\tilde{s}/\sqrt{\mu}$ of \eqref{eq:Weber} one gets
\be
\label{eq:Weber2}
\frac{\6^2u_1}{\6\tilde{s}^2}-\tilde{s}\frac{\6u_1}{\6\tilde{s}}+
\frac{1}{\mu}u_1=0
\ee 
which is referred to as \texttt{Weber equation} or 
\texttt{Ricatti-Hermite equation}. The ODE \eqref{eq:Weber2} has 
explicit solutions in terms of Hermite polynomials (see 
\cite{AS}, p.~781). Then the asymptotic properties of Hermite 
polynomials can be used to draw conclusions about the existence 
of maximal canards.\\

We develop two alternative ways to describe the variational equation 
\eqref{eq:main_ex_var3} directly. This provides new quantitative 
information about canard solutions and also gives estimates for the 
spacing of canards on cross-sections near the folded node. Our first 
approach in Section \ref{sec:averaging} uses averaging and provides a 
formal result. This result is then stated in Section 
\ref{sec:diagonal} and proven in Appendix \ref{appendix:canonical_form} 
using coordinate changes that are motivated by 
the formal calculation.

\subsection{Averaging}
\label{sec:averaging}

As shown in \eqref{eq:Weber} we can consider the variational equation 
\eqref{eq:main_ex_var3} as the second-order equation 
\be
\label{eq:Weber3}
\ddot{u}_1+u_1=\mu s\dot{u}_1\;.
\ee 
The form \eqref{eq:Weber3} suggests to view the problem as 
a damped oscillator; we assume that
\be
\label{eq:cond_range}
-\frac{2}{\mu}< s_0\leqs  s\leqs 0\;. 
\ee
Note that $\mu s=-2$ is precisely critical damping and for $\mu s\in(-2,0)$ 
the oscillator is underdamped. Therefore we expect that \eqref{eq:Weber3} 
describes harmonic oscillations with damping/contraction that is 
non-uniform in time. We make a change to polar coordinates
\benn
(u_1,\dot{u}_1)=(r(s)\cos(s+\psi(s)),-r(s)\sin(s+\psi(s)))
\eenn
which yields the ODEs
\be
\label{eq:polar}
\begin{array}{lcl}
\dot{r} &=& \mu s r\sin^2(s+\psi)\;,\\
\dot{\psi} &=& \mu s \cos(s+\psi)\sin(s+\psi)\;.\\
\end{array}
\ee
To simplify the analysis, consider the time rescaling 
$s=-\sqrt{-\tilde{s}}$. This converts \eqref{eq:polar} to
\be
\label{eq:polar1}
\begin{array}{lcl}
\frac{\6r}{\6\tilde{s}} &=& -\frac{\mu  r}{2}\sin^2(-\sqrt{-\tilde{s}}
+\psi)\;,\\
\frac{\6\psi}{\6\tilde{s}} &=& -\frac{\mu}{2}  
\cos(-\sqrt{-\tilde{s}} +\psi)\sin(-\sqrt{-\tilde{s}} +\psi)\;.\\
\end{array}
\ee
We consider \eqref{eq:polar1} on each time subinterval 
\benn
I_j:=[-((j+2)\pi)^2,-(j\pi)^2], \qquad 
\text{for $j\in\{0,2,4,6,\ldots\}$}
\eenn
by viewing \eqref{eq:polar1} as a vector field on $\R^+\times (I_j/\sim)$ 
where the equivalence relation $\sim$ identifies the endpoints of $I_j$. 
Then the vector field is in the form for averaging \cite{Verhulst}. The 
averaged equations are
\begin{align}
\frac{\6r_j}{\6\tilde{s}} &= -\frac{\mu  r_j}{2|I_j|}\int_{I_j}
\sin^2(-\sqrt{-\tilde{s}} +\psi_j)d\tilde{s}
=-\frac{(2j+2)\pi +\sin(2\psi_j)}{2 (4j+4)\pi}\mu r_j\;, \label{eq:polar2a}\\
\frac{\6\psi_j}{\6\tilde{s}} &= -\frac{\mu}{2|I_j|}\int_{I_j}  
\cos(-\sqrt{-\tilde{s}} +\psi_j)\sin(-\sqrt{-\tilde{s}} +\psi_j) d\tilde{s}
=-\frac{\mu \pi \cos(2\psi_j)}{2(4j+4)\pi^2}\;. \label{eq:polar2b}
\end{align}
In particular, we find that if we take a formal limit $j\ra \I$ the equation 
for the radius is 
\be
\label{eq:radius_sol}
\frac{\6r_\I}{\6\tilde{s}}=-\frac{\mu}{4}  r_\I\;.
\ee

\noindent
\textit{Remark:} Observe that one could also view 
\eqref{eq:polar2a}--\eqref{eq:polar2b} as an autonomous vector 
field and formally average over the angle $\psi_j$ to get
\benn
\frac{\6r_j}{\6\tilde{s}}=-\frac{(j+1)}{(4j+4)}\mu r_j\;,\qquad 
\frac{\6\psi_j}{\6\tilde{s}} =0\;,
\eenn
and then take the limit $j\ra \I$.\\

The solution of \eqref{eq:radius_sol} is given by
\be
\label{eq:sol_avg}
r_\I(\tilde{s})=r_\I(\tilde{s}_0)\e^{-\frac14\mu(\tilde{s}-\tilde{s}_0)}=
r_\I(s_0)\e^{\frac14\mu(s^2-s^2_0)}=r_\I((2/\mu)z_0)\e^{(z^2-z^2_0)/\mu}
\ee
which shows that the leading-order behaviour of the solutions 
to the variational equation \eqref{eq:main_ex_var3} for $z<0$ consists 
of a contraction towards the weak canard, given by \eqref{eq:sol_avg}, 
combined with a rotation. If we assume that $\psi_0(0)=0$, 
i.e., the rotation ends at angle $0$ on section $\{z=0\}$, then solving 
\eqref{eq:polar2b} yields 
\benn
\psi_0(\tilde{s})=-\text{tan}^{-1}\left[\text{tanh}\left[ 
\frac{\tilde{s}\mu}{8(1+j)\pi} \right]\right]\;.
\eenn
In principle we can now calculate $\psi_0(-(2\pi)^2)$, use this 
result as an initial condition for $\psi_1$ and then repeat the 
process to get a very detailed description of the rotational properties 
of trajectories near a folded node.

\subsection{Diagonalization}
\label{sec:diagonal}

To give rigorous arguments instead of the above formal calculation, we start by considering 
the variational equation in first-order form \eqref{eq:main_ex_var3}. The 
matrix $A(z)$ has eigenvalues 
\benn
2z\pm 2\icx\omega(z)\;,\qquad \text{where $\omega(z)=\sqrt{1-z^2+\mu}$\;.}
\eenn
We assume that $|z|<1$ so that $\omega(z)$ is real and bounded away 
from $0$. Furthermore $A(z)$ has trace $4z$. We expect that the solution 
$u(z)$ for \eqref{eq:main_ex_var3} consists of a contraction and a rotation 
for 
\be
\label{eq:cond_range_z}
1<z_0\leqs z<0\;.
\ee
Observe that \eqref{eq:cond_range_z} corresponds to the condition 
\eqref{eq:cond_range}. 

\begin{thm}[Canonical form]
\label{thm:var1}
There exists a matrix 
\be
\label{eq:ccA}
S(z)=
\frac{1}{\sqrt{\omega(z)}}
\begin{pmatrix} -z+\omega(z) & 
-z-\omega(z) \\ 1 & 1\\ \end{pmatrix} + \cO(\mu)
\ee
such that the coordinate change $u(z)=S(z)\tilde{u}(z)$ transforms the
variational equation \eqref{eq:main_ex_var3} into canonical form 
\be
\label{eq:canonical_form_diag}
\mu \frac{\6 \tilde{u}}{\6 z}=
\begin{pmatrix}
a(z) & \varpi(z) \\ 
-\varpi(z) & a(z)  
\end{pmatrix}
\tilde{u}\;,
\ee
where
\begin{align}
\nonumber
a(z) &= 2z + \Order{\mu^2}\;,\\
\varpi(z) &= 2\omega(z) + \Order{\mu}\;.
\label{eq:canonical_form_diagB}
\end{align}
As a consequence, we can write the solution of \eqref{eq:main_ex_var3} in the
form 
\be
\label{eq:main_sol_var1}
u(z)=
\e^{\alpha(z,z_0)/\mu}S(z)U(z,z_0)S(z_0)^{-1}u(z_0)
\ee
for $|z|< 1$ and $|z_0|<1$, where 
\be
\label{eq:main_sol_var2}
\alpha(z,z_0) = \int_{z_0}^z a(s)\,\6s
= z^2 - z_0^2 + \Order{\mu^2}
\ee
and $U(z,z_0)$ is the orthogonal matrix 
\be
\label{eq:fund_rot}
U(z,z_0)=
\begin{pmatrix}
\cos(\varphi(z,s)/\mu) & \sin(\varphi(z,s)/\mu) \\
-\sin(\varphi(z,s)/\mu) & \cos(\varphi(z,s)/\mu)
\end{pmatrix}\;,
\qquad 
\varphi(z,z_0)=\int_{z_0}^z \varpi(s)\,\6s+\cO(\mu)\;.
\ee
\end{thm}

The proof is given in Appendix~\ref{appendix:canonical_form}.
We remark that Theorem \ref{thm:var1} easily extends to the linearization around
solutions other than the weak canard. 

\subsection{The Distance Estimate}

From the solution of the variational equation in Theorem \ref{thm:var1} we can
give asymptotic estimates on the distance of the secondary canards to the weak
canard 
on the section
\benn
\Sigma^0:=\{(x,y,z)\in \R^3:z=0\}\;.
\eenn

\begin{thm}[Canard spacing]
\label{thm:dist_canards}
The distance of the $k$-th secondary canard $\gamma^k$ to the weak canard $\gamma^w$ on
$\Sigma^0$ is given by 
$\cO(\e^{-c_{0}(2k+1)^2\mu})$ as $\mu\ra 0$ where $0<c_{0}=\Order{1}$ is a constant. 
\end{thm}

\begin{proof}
By Theorem \ref{thm:rotation_det}, part (C6), the $k$-th secondary canard makes
$(2k+1)/2$ twists around the weak canard for $z$ going from $z_0$ to $0$. Using this
fact together with the rotational part \eqref{eq:fund_rot} of the solution
\eqref{eq:main_sol_var1} in Theorem \ref{thm:var1} we find the condition 
\be
\label{eq:int_above}
\frac{\varphi(0,z_0)}{\mu}=\frac{1}{\mu}\int_{z_0}^0
\varpi(s)\6s\stackrel{!}{=}\pi\frac{2k+1}{2}
\ee 
to leading order. Using concavity of the integral $z\mapsto \int_{z}^0 \omega(s)\6s$ for $z\in(-1,0)$, we find  
\be
\label{eq:int_eval}
-\frac{\pi}{4}z_0\sqrt{1+\mu}\leqs \int_{z_0}^0 \omega(s)\6s\leqs
-z_0\sqrt{1+\mu}\;,
\ee
which enables us to approximate the integral in~\eqref{eq:int_above}. Hence as $\mu\ra 0$ we get that 
\be
\label{eq:z0_order}
z_0\in\left[-(2k+1)\mu,-\frac{\pi}{4}(2k+1)\mu\right]=:[z_0^{(1)},z_0^{(2)}] \;,
\ee
where $z_0^{(1)}$ is the starting-point estimate for maximal canards that start to
spiral around $\gamma^w$ near $z\gtrapprox -1$ and $z_0^{(2)}$ the estimate for
maximal canards that start to spiral around $\gamma^w$ near $z\lessapprox 0$. In
particular, we find that $z_0=\cO((2k+1)\mu)$. Note that the contraction term
towards $\gamma^w$ in \eqref{eq:main_sol_var1} for $-1<z_0\leqs z<0$ has order
\be
\label{eq:contract_term}
\cO\left(\e^{(z^2-z_0^2)/\mu}\right)\;.
\ee
The result follows upon evaluating \eqref{eq:contract_term} on $\Sigma^0$ and substituting \eqref{eq:z0_order}.  
\end{proof}

\section{Stochastic Fast-Slow Systems}
\label{sec:SDEfastslow}

As for deterministic fast--slow systems we only give a brief introduction to 
stochastic fast--slow systems. For a detailed introduction to stochastic 
multiple time-scale dynamics consider \cite{BGbook,KuehnBook}; we are 
also going to assume standard results on stochastic differential 
equations (SDEs) \cite{Oksendal,Kallenberg}. All SDEs and stochastic 
integrals in this paper are considered in the It\^{o} interpretation.\\

We associate to the deterministic fast--slow system 
\eqref{eq:gen_fast_slow} a \texttt{stochastic fast-slow system} given by
\be
\label{eq:gen_fast_slow_SDE}
\begin{array}{lcl}
\6x_s&=& \frac1\epsilon  f(x,y,\mu,\epsilon)\6s
+\frac{\sigma}{\sqrt\epsilon}F(x_s,y_s) \6W_s\;,\\
\6y_s&=& g(x,y,\mu,\epsilon)ds+\sigma'G(x_s,y_s) \6W_s\;,\\
\end{array}
\ee 
where $F:\R^{m+n}\ra \R^{m\times k}$, $G:\R^{m+n}\ra \R^{n\times k}$, 
$\{W_s\}_{s\geqs 0}$ is a $k$-dimensional standard Brownian 
motion on some probability space $(\Omega,\mathcal{F},\P)$ and 
$\sigma,\sigma'>0$ are parameters controlling the \texttt{noise level}
\benn
\sqrt{\sigma^2+(\sigma')^2}\;.
\eenn 
We also define $\rho:=\sigma'/\sigma$ and assume that $\rho$ is bounded above and below by positive constants. The initial conditions 
$(x_0,y_0):=(x_{s_0},y_{s_0})$ are chosen to be square-integrable and independent of $\{W_s\}_{s\geqs s_{0}}$. Furthermore we 
assume that $f,g,F,G$ are sufficiently smooth and satisfy the standard assumptions that guarantee the existence and pathwise uniqueness of strong solutions $(x_s,y_s)$ to 
the SDE \eqref{eq:gen_fast_slow_SDE}. Note that these conditions also imply the existence of a continuous version of $(x_s,y_s)$. 
The law of the process $(x_s,y_s)$, starting at time $s_{0}$ in $(x_0,y_0)$, is denoted by $\P^{s_0,(x_0,y_0)}$ and 
the corresponding expectation by $\E^{s_0,(x_0,y_0)}$. Our approach to 
understand the dynamics of \eqref{eq:gen_fast_slow_SDE} is to analyze 
the time a sample path spends in a given Borel-measurable set 
$\mathcal{A}\subset \R^{m+n}$. Suppose $(x_0,y_0)\in\mathcal{A}$ and 
define the \texttt{first-exit time} from $\mathcal{A}$ as
\benn
\tau_\mathcal{A}:=\inf\{s\in[s_0,\infty):(x_s,y_s)\notin \mathcal{A}\}\;.
\eenn   
In this paper we are always going to choose sets $\mathcal{A}$ so that 
$\tau_{\mathcal{A}}$ is a stopping time with respect to the filtration 
generated by $\{(x_s,y_s)\}_{s\geqs s_0}$. \\

Our first goal is to state an analog of Fenichel's Theorem. This theorem 
is going to describe the typical spreading of sample paths near an attracting critical 
manifold bounded away from folded singularities. Suppose the deterministic 
version of \eqref{eq:gen_fast_slow_SDE} with $\sigma=0=\sigma'$ has a 
compact attracting normally hyperbolic critical manifold 
\benn
C_0=\{(x,y)\in\R^{m+n}:x=h_0(y),y\in\mathcal{D}_0\}\;. 
\eenn
Let $C_\epsilon$ be the slow manifold 
obtained from Fenichel's Theorem \ref{thm:fenichel1}. Our strategy is 
to construct a neighborhood $\mathcal{B}(r)$ for $C_\epsilon$ that 
contains the sample paths with high probability 
\cite{BG6,BGbook}. Define
\be
\label{eq:SDE_deviate}
\xi_s=x_s-h_\epsilon(y_s)\;.
\ee  
Observe that $\xi_s$ measures the deviation of the fast components from 
$C_\epsilon$. Applying It\^{o}'s formula to \eqref{eq:SDE_deviate} gives:
\bea
\label{eq:SDE_deviate1}
\6\xi_s &=& \6x_s-D_yh_\epsilon(y_s)\6y_s+\cO\left((\sigma')^2\right)\6s \\
&=& \frac{1}{\epsilon}\left[ f(h_\epsilon(y_s)+\xi_s,y_s,\mu,\epsilon)
-\epsilon D_yh_\epsilon(y_s)g(h_\epsilon(y_s)+\xi_s,y_s,\mu,\epsilon)
+\cO(\epsilon(\sigma')^2) \right]\6s \nonumber\\
&& +\frac{\sigma}{\sqrt\epsilon}\left[ F(h_\epsilon(y_s)+\xi_s,y_s)
-\rho\sqrt\epsilon D_yh_\epsilon(y_s)G(h_\epsilon(y_s)+\xi_s,y_s)\right]
\6W_s\;. \nonumber
\eea 
From now on, we suppress the arguments $\mu$ and $\epsilon$ for brevity. 
Consider the linear approximation of \eqref{eq:SDE_deviate} in $\xi_s$, 
neglect the It\^{o} term $\cO(\epsilon(\sigma')^2)$ and replace $y_s$ by 
its deterministic version $y^{\det}_s$ to obtain
\be
\label{eq:SDE_deviate_approx}
\begin{array}{lcl}
\6\xi^0_s&=&\frac1\epsilon A_\epsilon(y^{\det}_s)\xi^0_s\6s
+\frac{\sigma}{\sqrt\epsilon}F^0_\epsilon(y^{\det}_s)\6W_s\;,\\
\6y^{\det}_s&=& g(h_\epsilon(y^{\det}_s),y^{\det}_s)\6s\;,\\
\end{array}
\ee
where the two matrices $A_\epsilon$ and $F^0_\epsilon$ are defined as
\benn
\begin{array}{lcl}
A_\epsilon(y)&=& D_xf(h_\epsilon(y),y)-\epsilon 
~D_yh_\epsilon(y)~D_xg(h_\epsilon(y),y)\;,\\
F^0_\epsilon(y) &=& F(h_\epsilon(y),y)-\rho \sqrt{\epsilon}
D_yh_\epsilon(y)G(h_\epsilon(y),y)\;.\\
\end{array}
\eenn
Observe that $A_0(y)=D_xf(h_0(y),y)$ and $F^0_0(y)=F(h_0(y),y)$. To 
solve \eqref{eq:SDE_deviate_approx}, we pick an initial condition on 
the slow manifold $(\xi^0_0,y^{\det}_0)=(0,y^{\det}_0)$. Then the 
solution of \eqref{eq:SDE_deviate_approx} is the It\^{o} integral
\benn
\xi^0_s=\frac{\sigma}{\sqrt\epsilon} \int_0^s U(s,r)
F^0_\epsilon(y^{\det}_r)\6W_r
\eenn
where $U(s,r)$ denotes the principal solution of the homogeneous linear system 
$\epsilon \dot{\nu} = A_\epsilon(y^{\det}_s)\nu$. If we fix a time $s$ then 
$\xi^0_s$ is a Gaussian random variable of mean zero and covariance matrix
\benn
\Cov(\xi^0_s)=\frac{\sigma^2}{\epsilon}\int_0^s U(s,r)
F^0_\epsilon(y^{\det}_s) F^0_\epsilon(y^{\det}_s)^TU(s,r)^T\6r\;.
\eenn
Note carefully that $X_s:=\sigma^{-2}\Cov(\xi^0_s)$ does satisfy a 
fast--slow ODE given by
\be
\label{eq:SDE_deviate_ODE}
\begin{array}{lcl}
\epsilon \dot{X}&=& A_\epsilon(y)X+XA_\epsilon(y)^T+
F^0_\epsilon(y)F^0_\epsilon(y)^T\;,\\
\dot{y}&=& g(h_\epsilon(y),y)\;.
\end{array}
\ee
The system \eqref{eq:SDE_deviate_ODE} has a critical manifold $S^\xi_0$ 
given by the equation
\benn
A_0(y)X+XA_0(y)^T+F^0_0(y)F^0_0(y)^T=0\;.
\eenn
By the remarks above we see that this is equivalent to solving
\be
\label{eq:SDE_deviate_critman}
(D_xf)(h_0(y),y)X+X[(D_xf)(h_0(y),y)]^T+F(h_0(y),y)F(h_0(y),y)^T=0\;.
\ee
Again this manifold can be locally described as a graph 
\benn
S^\xi_0=\{(X,y)\in \R^{m+n}:X=H_0(y)\}\;.
\eenn
The next lemma states that $S^\xi_0$ is normally hyperbolic and attracting.

\begin{lem}[\cite{Bellman}]
Let $M_1$ and $M_2$ be square matrices of dimension $m$ with 
eigenvalues $\lambda_{1,1},\ldots\lambda_{1,m}$ and $\lambda_{2,1},
\ldots\lambda_{2,m}$, respectively. Then the linear map $L:\R^{m\times m}
\ra\R^{m\times m}$ defined by
\benn
L(X)=M_1X+XM_2
\eenn 
has $m^2$ eigenvalues given by $\{\lambda_{1,i}+\lambda_{2,j}\}$ 
for $i,j\in\{1,2,\ldots,m\}$.
\end{lem} 

Therefore Fenichel's Theorem \ref{thm:fenichel1} provides us with a 
slow manifold 
\benn
S^\xi_\epsilon=\{(X,y)\in \cD\subset \R^{m+n}:
X=H_\epsilon(y)=H_0(y)+\cO(\epsilon)\}\;,
\eenn 
which we will now use to describe the typical spreading of sample paths.

\begin{thm}[\cite{BG6,BGbook}] Suppose the norms 
$\|H_\epsilon(y)\|$ and $\|H^{-1}_\epsilon(y)\|$ are uniformly bounded. 
Define the neighborhood $\mathcal{B}(r)$ around the deterministic 
slow manifold as
\benn
\mathcal{B}(r):=\{(x,y)\in \cD:
\pscal{[x-h_\epsilon(y)]}{H^{-1}_\epsilon(y)[x-h_\epsilon(y)]}<r^2\}\;.
\eenn
Then, for $\epsilon>0$ and $\sigma>0$ sufficiently small, sample paths starting 
on $C_\epsilon$ remain in $\mathcal{B}(r)$ 
with high probability; in particular, we have
\be
\label{eq:B(h)} 
\P^{s_0,(x_0,y_0)}\bigl\{\tau_{\mathcal{B}(r)}<s\wedge 
\tau_{\mathcal{D}_0}\bigr\}<K_1(s,\epsilon)\e^{-K_2r^2/2\sigma^2}\;.
\ee
with $K_{1,2}>0$.
\end{thm}

Detailed discussions of the factors $K_{1}(s,\epsilon)$ and $K_2$ can 
be found in \cite{BG6,BGbook}. In particular,
$K_{1}(s,\epsilon)$ grows at most like $s^2$, while $K_2$ does not depend on
time and can be taken close to $1$. The probability in~\eqref{eq:B(h)} thus
remains small on long time spans as soon as we choose $r\gg\sigma$. 
This implies that 
sample paths stay for exponentially long times near an attracting 
slow manifold before they jump away unless they come close to the boundary 
of $C_0$ before, i.e., the $y$-coordinates leave the set $\mathcal{D}_0$.

\section{Stochastic Folded Nodes}
\label{sec:SDEfoldednode}

\subsection{Blow-Up}
\label{ssec:blowup}

Proposition \ref{prop:nform_deter} shows that Equation \eqref{eq:main_ex_eps} 
is a normal form for deterministic fast--slow systems with a folded node. We 
study the associated SDE 
\be
\label{eq:main_ex_SDE}
\begin{array}{lcl}
\6x_s&=& \frac1\epsilon (y_s-x_s^2)\6s + \frac{\sigma}{\sqrt\epsilon}
\6W^{(1)}_s\;,\\
\6y_s&=& \left[-(\mu+1)x_s-z_s\right]\6s+\sigma' \6W^{(2)}_s\;,\\
\6z_s&=& \frac\mu2\6s\;,\\
\end{array}
\ee
where $W^{(1)}_s$, $W^{(2)}_s$ are independent standard Brownian motions; to
simplify the notation we also define $W_s:=(W^{(1)}_s,W^{(2)}_s)^T$. Since $z$
plays the role of a time variable we do not add noise to the $z$-component. We
will always assume that the noise terms are of equal order, i.e., that
$\rho=\sigma'/\sigma$ is bounded above and below by positive constants. Note
that 
\eqref{eq:main_ex_SDE} fits into the framework of a general fast--slow SDE 
\eqref{eq:gen_fast_slow_SDE} with
\benn
F(x,y)=\begin{pmatrix}1 & 0\\ \end{pmatrix}\qquad \text{and} 
\qquad G(x,y)=\begin{pmatrix}0 & 1\\\end{pmatrix}\;.
\eenn
We apply the rescaling/blowup given by \eqref{eq:rescale1} to 
\eqref{eq:main_ex_SDE} to get
\benn
\begin{array}{rcl}
\epsilon^{1/2}\6\bar{x}_{\bar{s}}&=& \epsilon^{1/2}(\bar{y}_{\bar{s}}-
\bar{x}_{\bar{s}}^2)\6\bar{s} + \frac{\sigma}{\sqrt\epsilon} 
\6W^{(1)}_{\epsilon^{1/2}\bar{s}}\;,\\
\epsilon \6\bar{y}_{\bar{s}}&=& \epsilon\left[-(\mu+1)\bar{x}_{\bar{s}}
-\bar{z}_{\bar{s}}\right]\6\bar{s}+\sigma'
\6W^{(2)}_{\epsilon^{1/2}\bar{s}}\;,\\
\6\bar{z}_{\bar{s}}&=& \frac\mu2\6\bar{s}\;.\\
\end{array}
\eenn
We use the scaling law of Brownian motion, divide the first equation by 
$\sqrt\epsilon$ and the second one by $\epsilon$ and drop the overbars for 
notational convenience, obtaining
\be
\label{eq:SDE_bu}
\begin{array}{lcl}
\6x_s&=& (y_s-x_s^2)\6s + \frac{\sigma}{\epsilon^{3/4}} \6W^{(1)}_{s}\;,\\
\6y_s&=& \left[-(\mu+1)x_s-z_s\right]\6s+\frac{\sigma'}{\epsilon^{3/4}} 
\6W^{(2)}_{s}\;,\\
\6z_s&=& \frac\mu2\6s\;.\\
\end{array}
\ee
Therefore rescaling the noise intensities as 
\be
\label{eq:rescale_noise}
(\sigma,\sigma')=(\epsilon^{3/4}\bar{\sigma},\epsilon^{3/4}\bar{\sigma}')
\ee
removes $\epsilon$ from the equations and yields
\be
\label{eq:SDE_bu1}
\begin{array}{lcl}
\6x_s&=& (y_s-x_s^2)\6s + \sigma \6W^{(1)}_{s}\;,\\
\6y_s&=& \left[-(\mu+1)x_s-z_s\right]\6s+\sigma' \6W^{(2)}_{s}\;,\\
\6z_s&=& \frac\mu2\6s\;,\\
\end{array}
\ee
where the overbars from \eqref{eq:rescale_noise} have again been dropped. To
study \eqref{eq:SDE_bu1} we consider the variational equation around a
deterministic solution (which may be the weak primary canard, a secondary
canard, or even any other solution with initial condition $(x_0,y_0,z_0)$
sufficiently close to $C^a_0$). This approach is a generalization of Section
\ref{sec:SDEfastslow} where we considered the variation around the slow
manifold. Viewing \eqref{eq:SDE_bu1} as a planar non-autonomous system with time
$z=(\mu/2)s$ and setting
\benn
(x_z,y_z)=(x^{\det}_z+\xi_z,y^{\det}_z+\eta_z)\;,
\eenn
we get the SDE
\be
\label{eq:vareq_SDE}
\begin{array}{lcl}
\6\xi_z&=&\frac2\mu(\eta_z-\xi_z^2-2x_z^{\det} \xi_z)\6z
+\sigma\sqrt{\frac2\mu} \6W^{(1)}_z\;,\\
\6\eta_z&=&-\frac2\mu(\mu+1)\xi_z\6z
+\sigma'\sqrt{\frac2\mu}\6W^{(2)}_z\;.\\
\end{array}
\ee

\subsection{Covariance Tubes}
\label{ssec:covariance_tubes}

We linearize \eqref{eq:vareq_SDE} and denote by $\zeta^0_z=(\xi^0_z,\eta_z^0)^T$ the solution 
which satisfies the SDE
\be
\label{eq:lin_var_SDE}
\6\zeta^0_z=\frac{1}{\mu}
\underbrace{\begin{pmatrix} -4x_z^{\det} & 2 \\ 
-2(\mu+1) & 0 \\\end{pmatrix}}_{=:A(x^{\det}_z)} \zeta^0_z\6z 
+ \frac{\sigma}{\sqrt{\mu}}
\underbrace{\begin{pmatrix} \sqrt{2} & 0 \\ 0 & \sqrt{2}\rho
\\ 
\end{pmatrix}}_{=:F^0} \6W_z\;.
\ee 
Then \eqref{eq:lin_var_SDE} is solved by the following Gaussian 
process
\benn
\zeta^0_z=U(z,z_0)\zeta^0_{z_0}+\frac{\sigma}{\sqrt{\mu}}\int_{z_0}^z
U(z,r)F^0\6W_r \;,
\eenn
where $U(z,r)$ is the principal solution to the 
deterministic homogeneous non-autonomous linear system 
$\mu\dot{\zeta}=A(x^{\det}_z)\zeta$. The two-by-two covariance 
matrix $\Cov(\zeta^0_z)=:\Cov(z)$ is given by
\be
\label{eq:cov_matrix}
\Cov(z)=
\frac{\sigma^2}{\mu}\int_{z_0}^z U(z,r)(F^0)(F^0)^TU(z,r)^T\6r\;.
\ee
Differentiating $V(z):=\sigma^{-2}\Cov(z)$ we find that it 
satisfies the ODE
\be
\label{eq:cov_matrix1}
\mu\frac{\6V}{\6z}
=A(x^{\det}_z)V+VA(x^{\det}_z)^T+(F^0)(F^0)^T
\ee
with initial condition $V(z_0)=0$. 
The following result describes the behaviour of the solutions
of~\eqref{eq:cov_matrix1}, in particular as $z$ approaches $0$. 

\begin{thm}[Behaviour of the covariance matrix]
\label{thm:covariance_tubes}
Fix an initial time $z_0<0$. There exist constants $c_+>c_->0$ such that, for all $0<\mu\ll 1$,
the solution of~\eqref{eq:cov_matrix1} with initial condition $V(z_0)=0$
satisfies
\begin{align}
\label{eq:cov_matrix2A}
\frac{c_-}{\abs{z}} &\leqs V_{11}(z), V_{22}(z) \leqs \frac{c_+}{\abs{z}} 
&& \text{for $z_0+\Order{\mu\abs{\log\mu}} \leqs z \leqs -\sqrt{\mu}$\;,} \\
\frac{c_-}{\sqrt{\mu}} &\leqs V_{11}(z), V_{22}(z) \leqs
\frac{c_+}{\sqrt{\mu}} 
&& \text{for $-\sqrt{\mu} \leqs z \leqs \sqrt{\mu}$\;,} 
\label{eq:cov_matrix2B}
\end{align}
and
\begin{align}
\nonumber
\abs{V_{12}(z)} = \abs{V_{21}(z)} &\leqs c_+
&& \text{for $z_0\leqs z \leqs \sqrt{\mu}$\;,} \\
\abs{V_{22}(z)-V_{11}(z)} &\leqs c_+
&& \text{for $z_0\leqs z \leqs \sqrt{\mu}$\;.} 
\label{eq:cov_matrix3}
\end{align}
Furthermore, let $\Vbar(z)$ be any solution of~\eqref{eq:cov_matrix1} with
positive definite initial condition $\Vbar(z_0)$. More precisely, we require
both $\Vbar(z_0)$ and $\Vbar(z_0)^{-1}$ to have all elements uniformly bounded
in $\mu$. Then the matrix elements of $\Vbar(z)$ satisfy~\eqref{eq:cov_matrix2B}
and~\eqref{eq:cov_matrix3}, and~\eqref{eq:cov_matrix2A} holds for \textbf{all}
$z\in[z_0,-\sqrt{\mu}\,]$.
\end{thm}

The proof is given in Appendix~\ref{appendix:covariance_tubes}, where we also
give some additional information on how the covariance can be approximated by
asymptotic expansions.

Equations~\eqref{eq:cov_matrix2A} and~\eqref{eq:cov_matrix2B} show that the
variances of $\xi^0_z$ and $\eta^0_z$ grow like $\sigma^2/\abs{z}$ up to time
$-\sqrt{\mu}$, and then stay of order $\sigma^2/\sqrt{\mu}$ up to time
$\sqrt{\mu}$. Thus we expect the fluctuations of stochastic sample paths around
the deterministic solution to increase when the fold at $z=0$ is approached. 
The restriction $z\geqs z_0+\Order{\mu\abs{\log\mu}}$ is due to the fact that
the variances are initially equal to zero, and need some time to build up. 

Equations~\eqref{eq:cov_matrix3} show that the covariance of $\xi^0_z$ and
$\eta^0_z$ remains bounded, of order $\sigma^2$, up to time $\sqrt{\mu}$, and
the same holds true for the difference between the variances. This implies that
fluctuations become more isotropic as $z$ approaches $0$. \\

We now turn to the analysis of the full nonlinear SDE~\eqref{eq:vareq_SDE}
satisfied by the difference $\zeta_z=(\xi_z,\eta_z)$ between stochastic sample
paths and deterministic solutions. This SDE can be written in vectorial form as 
\begin{equation}
\label{eq:nonlin01}
\6\zeta_z = \frac{1}{\mu}\left[A(x^{\det}_z)\zeta_z 
+ b(\zeta_z) \right] \6z + \frac{\sigma}{\sqrt{\mu}} F^0\6W_z\;, 
\end{equation} 
where $A$ and $F^0$ have been defined in~\eqref{eq:lin_var_SDE}, and
$b(\zeta)^T=(-\xi^2,0)$ denotes the nonlinear term. We expect the covariance
matrix of $\zeta_z$ to be close to $\Cov(\zeta^0_z)=\sigma^2 V(z)$, where $V(z)$
is the solution of~\eqref{eq:cov_matrix1} with initial condition $V(z_0)=0$.
Thus sample paths should be concentrated in a tube surrounding the deterministic
solution, with elliptical cross-section determined by $V(z)$. Note that the
elliptical cross-section becomes close to circular as $z$ approaches $0$, since
the variances $V_{11}, V_{22}$ are then of larger order than the covariance
$V_{12}$ and the difference $V_{22}-V_{11}$.

The fact that $V(z_0)$ is not invertible causes some technical complications.
Therefore, in the following we let $\Vbar(z)$ be the solution
of~\eqref{eq:cov_matrix1} with an initial condition $\Vbar(z_0)$ which is
positive definite. Observe that
the difference between $\Vbar(z)$ and $V(z)$ decreases exponentially fast. We
fix a $z_0<0$ and define the covariance-tube as 
\begin{equation}
 \label{eq:nonlin02}
\cB(r) = \bigl\{(x,y,z) \colon z_0 \leqs z \leqs \sqrt{\mu}, 
\pscal{[(x,y)-(x^{\det}_z,y^{\det}_z)]}
{\Vbar(z)^{-1}[(x,y)-(x^{\det}_z,y^{\det}_z) ] } < r^2
\bigr\} \;.
\end{equation} 
The cross-section of $\cB(r)$ at any plane $\{z=\const\}$ is an ellipsoid
whose axes are determined by $\Vbar(z)$, while the scaling parameter $r$
controls the size of the tube.

\begin{thm}[Concentration of sample paths in the covariance tube]
\label{thm:nonlinear_SDE}
There exist constants $\Delta_0, r_0,\mu_{0}>0$ such that for all $0<\Delta<\Delta_0$, all
$\sigma<r<r_0\mu^{3/4}$ and all $\mu\le\mu_{0}$, 
\begin{equation}
 \label{eq:nonlin03}
\P \bigl\{ \tau_{\cB(r)} < z \bigr\} 
\leqs 
C_+(z,z_0) \e^{-\kappa_0 r^2/2\sigma^2}
\end{equation}  
holds for all $z\leqs \sqrt{\mu}$, where the exponent $\kappa_0$ satisfies 
\begin{equation}
 \label{eq:nonlin04A} 
\kappa_0 =  1 - \Order{\Delta}  - \Order{r\mu^{-3/4}} 
- \Order{\sigma^2/r^2}\;,
\end{equation} 
and the prefactor is given by 
\begin{equation}
 \label{eq:nonlin04B}
C_+(z,z_0) =1+ \frac{\const}{\Delta\mu} 
\biggl( \frac{r}{\sigma} \biggr)^2
\int_{z_0}^{z} x^{\det}_s
\6s\;. 
\end{equation} 
\end{thm}

The proof is given in Appendix~\ref{appendix:proof_nonlinear_SDE}. This result
shows that the probability that sample paths leave the covariance-tube before
time $z$ is small, provided we choose $r \gg \sigma(\log(C_+(z,z_0)))^{1/2}$. Indeed, for
these values of $r$, the exponential term dominates the polynomial prefactor. We
thus say that sample paths are concentrated in the covariance tube~$\cB(r)$ for
$r$ slightly larger than $\sigma$, or that the typical spreading of sample
paths is given by~$\cB(\sigma)$. 

The condition $\sigma<r<r_0\mu^{3/4}$ implies that the theorem only applies to
noise intensities smaller than $\Order{\mu^{3/4}}$. What happens for
$\sigma\geqs\mu^{3/4}$ is that fluctuations become large already some time
before the fold line is reached, which completely smears out the small
oscillations present in the deterministic case. In fact, it is possible to show
that if $\sigma\geqs\mu^{3/4}$, the bound~\eqref{eq:nonlin03} still holds true
for $r\leqs r_0\abs{z}^{3/2}$, and thus sample paths are localised up to times
$z\ll-\sigma^{2/3}$. 

\subsection{Small-Amplitude Oscillations and Noise}
\label{ssec:SAOs}

Our knowledge of the size of the covariance tubes now allows to determine when they start 
to overlap; from this we can deduce consequences for the existence of small-amplitude oscillations (SAOs) in the presence of noise. Recall from Theorem
\ref{thm:canardsR3} 
that there are two primary 
canards $\gamma^{s,w}_\epsilon$ and $K-1$ secondary canards $\gamma^k_\epsilon$ 
for $2K-1<\mu^{-1}<2K+1$. We want to consider the rotations around
$\gamma^w_\epsilon$ 
and denote the strong canard by $\gamma^0_\epsilon$, i.e., $k=0$. By Theorem
\ref{thm:dist_canards} 
the distance from the weak canard is given by $\cO(\exp(-c_0(2k+1)^2\mu))$ for
some 
positive constant $c_0\in[\pi/4,1]$. Theorem \ref{thm:covariance_tubes} implies
that on $z=0$ the width of the covariance tubes is 
given  by $\cO(\mu^{-1/4}\sigma)$.  Therefore the covariance tubes for sample 
paths starting on the $k$-th canard start to overlap with the weak canard 
for $\sigma\approx \mu^{1/4}  \exp(-c_0(2k+1)^2\mu)$. 
We define the functions 
\benn
\sigma_k(\mu):=\mu^{1/4}e^{-c_0(2k+1)^2\mu}.
\eenn

\begin{figure}[htbp]
\includegraphics[width=1\textwidth,clip=true]{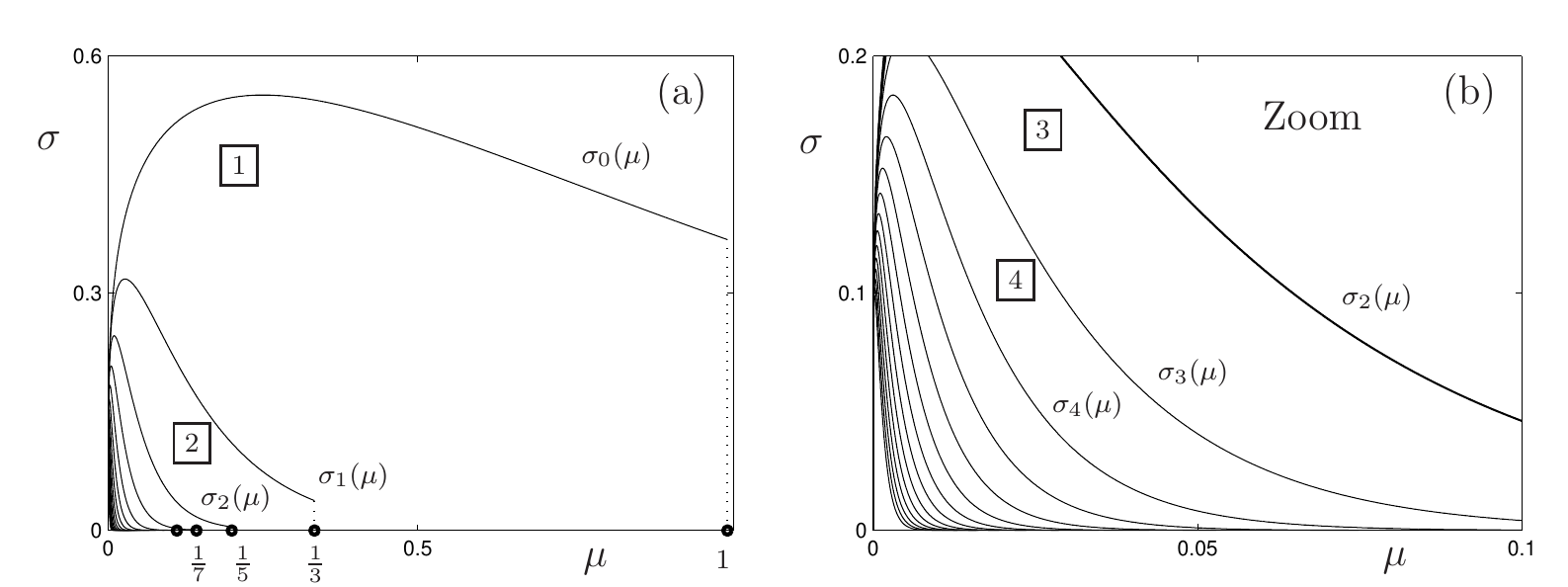}
	\caption{\label{fig:fig8}$(\mu,\sigma)$-parameter plane. (a) 
	The curves $\sigma_k(\mu)$ are shown for $c_0=1$. Above the $k$-th
curve 
	we cannot distinguish the oscillation induced by the $k$-th 
	canard from noisy fluctuations. The regions of different numbers 
	of SAOs are delimited by the curves $\sigma_k(\mu)$ and the 
	existence requirements for the $k$-th canard. The number of 
	visible canards is indicated in rectangular boxes. 
	(b) A zoom of (a) is shown that illustrates the structure of thin
regions 
	as $\mu\ra 0$ and $\sigma\ra 0$.}	
\end{figure}

The previous considerations yield the following result.

\begin{cor}[Noisy SAOs] 
\label{cor:noisySAOs}
On the section $\{z=0\}$ the covariance tubes of the $k$-th 
canard overlap if 
\be
\label{eq:noise_SAOs}
\sigma>\sigma_k(\mu)\;.
\ee
Therefore, depending on the noise level $\sigma$, the canard number and 
the parameter $\mu$, deterministic SAOs with $2k+1$ twists become
indistinguishable 
from noisy fluctuations whenever \eqref{eq:noise_SAOs} holds. 
\end{cor}

In Figure \ref{fig:fig8} we show the curves $\sigma_k(\mu)$ and 
indicate in which regions of the $(\mu,\sigma)$-parameter plane 
one can distinguish which number of canards. The curves $\sigma_k(\mu)$ 
and the existence conditions for canards enclose bounded regions where 
precisely $k+1$ canards can be distinguished which yields $(2k+1)/2$ twists 
up to the section $\{z=0\}$; see also Theorem \ref{thm:rotation_det}. We can 
also study Figure \ref{fig:fig8} for fixed $\sigma$. In this case decreasing 
$\mu$ first increases the number of visible SAOs and then decreases it again.

\subsection{Early Jumps}
\label{ssec:escape_canards}

We now turn to the behaviour for times $z>\sqrt{\mu}$. For definiteness, we
let $(x^{\det}_z,y^{\det}_z) = (-z,z^2-\mu/2)$ be the weak canard solution, and
define the set 
\begin{equation}
 \label{eq:esc_canard01}
\cD(\eta) = \bigl\{ (x,y,z) \colon z\geqs\sqrt{\mu}, 
(x-x^{\det}_z)^2 + (y-y^{\det}_z)^2 < \eta^2 z
\bigr\} \;.
\end{equation} 
$\cD(\eta)$ is a tube centred in the weak canard, whose width grows like
$\sqrt{z}$. The following result, which is proved in 
Appendix~\ref{appendix:proof_escape}, shows that sample paths are unlikely to
stay for very long in $\cD(\eta)$. 

\begin{thm}[Escape of sample paths from the primary canard]
\label{thm:escape_canards} 
There exist constants $\kappa=\kappa(\eta)>0$, $C_0>0$ and $\gamma_{1},\gamma_{2}>0$ such that,
whenever $\sigma\abs{\log\sigma}^{\gamma_{1}} \leqs \mu^{3/4}$, 
\begin{equation}
 \label{eq:esc_canard02}
\P \bigl\{ \tau_{\cD(\eta)} > z \bigr\}
\leqs C_0 \abs{\log\sigma}^{\gamma_{2}} \e^{-\kappa(z^2-\mu)/(\mu\abs{\log\sigma})}\;.
\end{equation} 
\end{thm}

The probability that a sample paths stays in $\cD(\eta)$ thus becomes small as
soon as 
\begin{equation}
 \label{eq:esc_canard03}
z \gg \sqrt{\mu\abs{\log{\sigma}}/\kappa}\;. 
\end{equation} 
Unless the noise intensity $\sigma$ is exponentially small in $\mu$, the
typical time at which sample paths jump away from the canard is slightly (that
is, logarithmically) larger than $\sqrt{\mu}$. 

\section{Numerics and Visualization}
\label{sec:numerics}

In this section we briefly discuss how to compute canard solutions and their 
associated covariance tubes. Furthermore we visualize the early jumps after 
passage near a folded node in phase space for a model system with global
returns. 
We also compute the probability density 
of escaping trajectories on a cross-section for this example. SDEs have been 
integrated numerically by a standard Euler--Maruyama 
scheme \cite{Higham,KloedenPlaten}. Deterministic solutions have been computed 
using a stiff ODE solver \cite{ShampineReichelt,HairerWannerII}.

\subsection{Covariance Tubes}
\label{sec:visual1}

The maximal canards and their associated covariance tubes can be computed. 
Figure \ref{fig:fig5} shows an example for these computations where we used 
the blown-up normal form \eqref{eq:SDE_bu1} with $\mu=0.08$.\\ 

\begin{figure}[htbp]
\includegraphics[width=1\textwidth,clip=true]{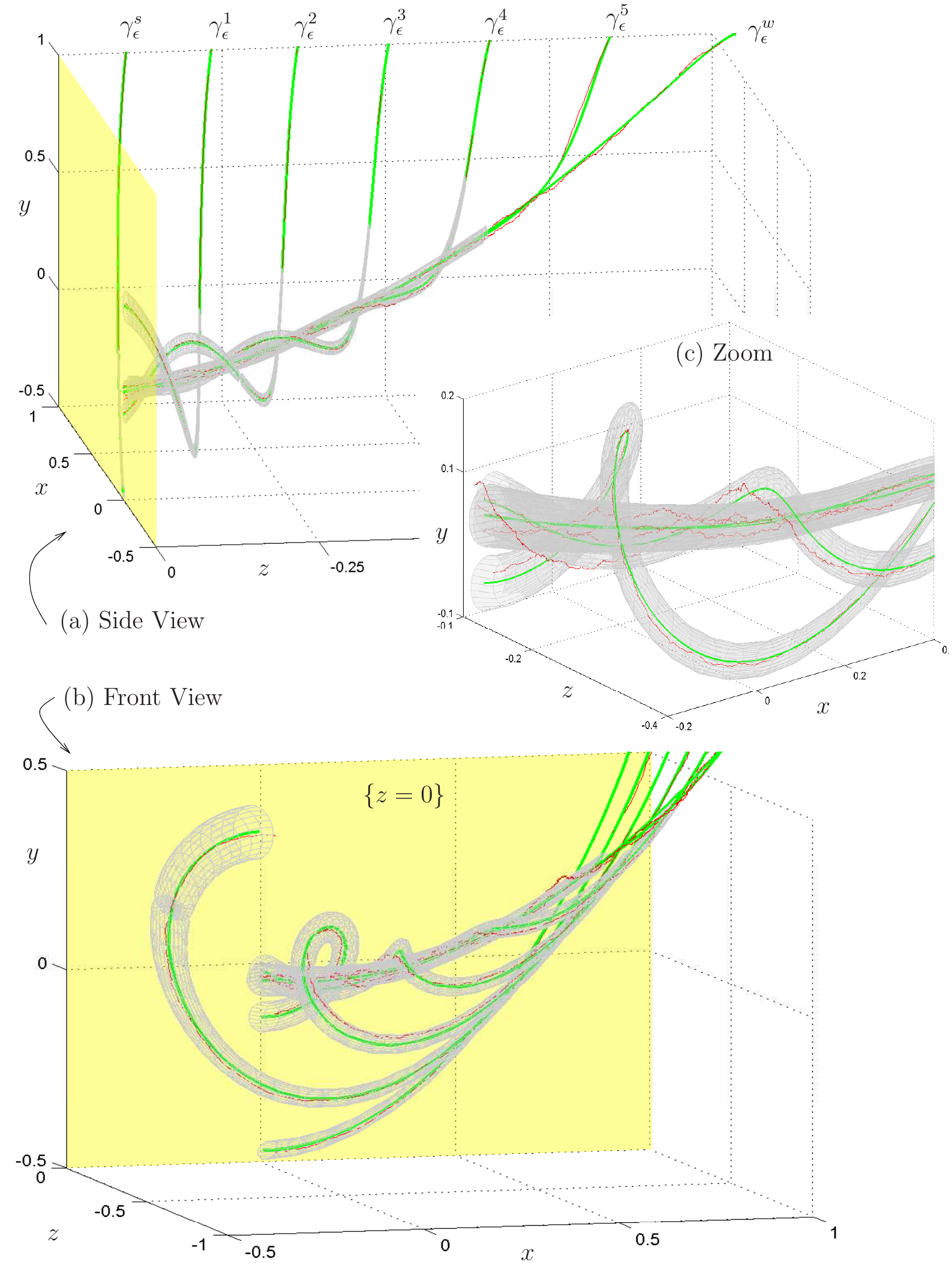}
	\caption{\label{fig:fig5} Computation of the canards and covariance
tubes 
	near a folded node in System \eqref{eq:SDE_bu1}. A detailed description
of 
	the figure can be found in Section \ref{sec:visual1}.}	
\end{figure}

First we compute the deterministic slow manifolds $C^a_\epsilon$ and
$C^r_\epsilon$ by 
forward respectively backward integration (see \cite{KuehnMMO}) up to the
section 
\benn
\Sigma^0=\{(x,y,z)\in\R^3:z=0\}\;.
\eenn 
The primary and secondary maximal canards have been computed as intersections 
of the slow manifolds $C^a_\epsilon\cap C^r_\epsilon$; see also Section 
\ref{sec:foldednode} and \cite{GuckHaiduc,DesrochesKrauskopfOsinga2}. The 
resulting maximal canards (thick green curves) are shown in 
Figure \ref{fig:fig5}. We also computed a sample path (thin red curves) for 
each maximal canard starting at the same point as the maximal canard with 
$z_0=-1$; the noise values were fixed at $\sigma=0.008=\sigma'$. The tubes 
defined by the covariance are shown in grey and have been computed using 
integration of the covariance differential equation \eqref{eq:cov_matrix1}; the 
section $\Sigma^0$ is drawn in yellow for better orientation.\\
 
Figure \ref{fig:fig5}(a) shows a side view that illustrates how the different 
primary canards $\gamma^{s,w}_\epsilon$ and secondary canards
$\gamma^j_\epsilon$ 
are organized with respect to $z$. We have only started to draw the covariance 
tubes $\cB(r)$ with $r^2=0.02$ a bit beyond the initial values at $x_0=1$. It is
clearly visible how the canards 
and their tubes are attracted towards the weak canard and then start to rotate 
around it. Figure \ref{fig:fig5}(b) shows a front view towards the section 
$\Sigma^0$. This view shows nicely how the tubes grow with increasing
$z$-values 
and that the ellipses defined by the covariance matrix are indeed close to 
circular near $z=0$. Furthermore we can see how the canards are organized on $\Sigma^0$; 
the maximal canard tubes for $\gamma^{s,1,2}_\epsilon$ do 
not overlap while all other tubes overlap near the weak canard. Figure 
\ref{fig:fig5}(c) shows a zoom that illustrates the twisting and also shows 
how the sample paths are indeed \lq\lq trapped\rq\rq\ inside the covariance tubes with 
very high probability. 

\subsection{Early Jumps}
\label{sec:visual2}

To visualize the effect of early jumps, we consider a folded node with global
returns 
given by 
\be
\label{eq:BKW_mod}
\begin{array}{lcl}
\6x &=& \frac{1}{\epsilon}(y-x^2-x^3) \6s + \frac{\sigma}{\sqrt{\epsilon}}
\6W_s^{(1)}\;, \\  
\6y &=& \left[-(\mu+1)x-z \right] \6s + \sigma' \6W^{(2)}_s\;, \\
\6z &=& \left[\frac{\mu}{2}+ax+bx^2\right] \6s\;, \\
\end{array}
\ee
which is a slight modification of a model system for folded-node MMOs
\cite{BronsKrupaWechselberger}. 
The critical manifold is cubic-shaped (or S-shaped) and given by
\benn
C_0=\{(x,y,z)\in\R^3:y=x^2+x^3\}=C^{a,-}_0\cup L_-\cup C^r\cup L^+ \cup
C^{a,+}\;
\eenn
where $C^{a,+}_0=C_0\cap \{x<-2/3\}$, $C^{r}_0=C_0\cap \{-2/3<x<0\}$,
$C^{a,+}_0=C_0\cap \{x>0\}$, 
$L_-=C_0\cap\{x=-2/3\}$ and $L_+=C_0\cap\{x=0\}$. The parameters $(a,b)\in\R^2$
help to adjust the global 
return mechanism. If $a,b$ are $\cO(1)$ then 
they do not influence the local behaviour of a folded node at the origin
$(x,y,z)=(0,0,0)$. Figure 
\ref{fig:fig7} shows the effect of early jumps after passage through a folded-node region. Parameters 
for the simulation are:
\be
\label{eq:paras_jumps}
\epsilon=0.01\;, \quad \mu=0.143\;, \quad a=0.2\;, \quad b=-1.1\;, \quad
\sigma=0.005\;, \quad \sigma'=0\;.
\ee
In Figure \ref{fig:fig7}(a) a deterministic trajectory (thick blue curve) has
been computed for $\sigma=0$. Then 
an SDE sample path for \eqref{eq:BKW_mod} has been started on a point (green
dot) of the deterministic solution 
and integrated forward. We define a cross-section
\benn
\Sigma^J:=\{(x,y,z)\in\R^3:x=-0.3\}\;.
\eenn

\begin{figure}[htbp]
\includegraphics[width=.9\textwidth,clip=true]{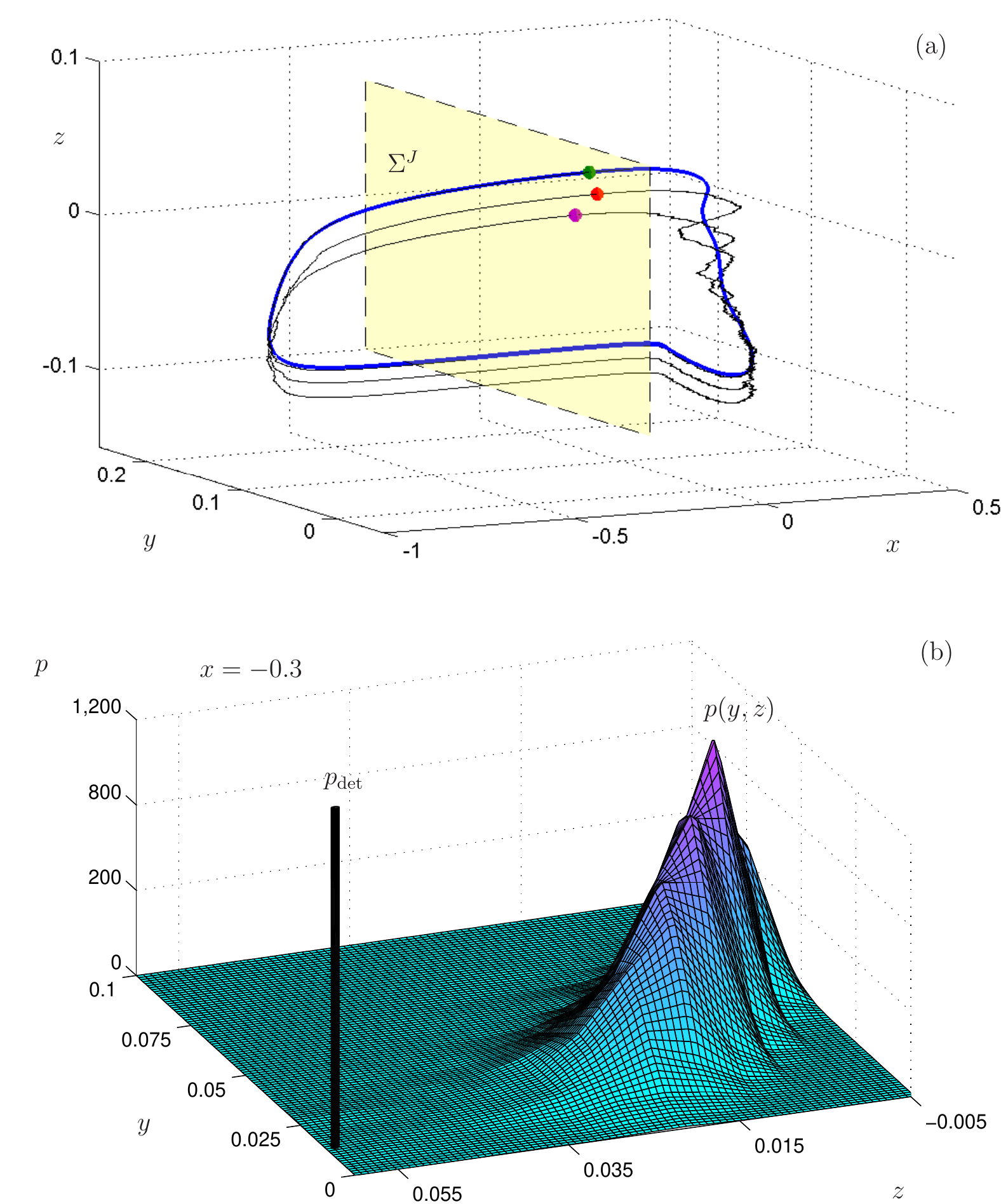}
	\caption{\label{fig:fig7} (a) Phase space plot of simulations for
\eqref{eq:BKW_mod}. The deterministic 
	solution (blue) and a stochastic sample path (black) are shown;
parameters are given by \eqref{eq:paras_jumps} 
	and for the deterministic solution we have $\sigma=0=\sigma'$.
Intersections with the cross-section $\Sigma^J$ (yellow) 
	are shown as thick dots; the SDE sample path is started on the
deterministic solution and on $\Sigma^J$ (green dot). 
	The next two intersections are shown as well (violet and red dots). (b)
Probability density $p(y,z)$ for 4000 
	sample paths on $\Sigma^J$ of sample path escapes from the folded node.
The deterministic point-mass density is 
	indicated as a bar $p_{\det}$ (black).}	
\end{figure}

Escapes of sample paths from the folded-node region are recorded on $\Sigma^J$.
The next two returns are also 
shown as points (violet and red) on the cross-section $\Sigma^J$. It is clearly
visible from 
Figure \ref{fig:fig7}(a) how in this realization the SDE sample path jumps before the deterministic
solution. 
Note that this causes the path to get re-injected into the folded-node region
after a large 
excursion at a point slightly different from the deterministic solution. Hence
the global return 
mechanism can potentially act as a control mechanism for the noise. To
investigate the early jumps 
further we show in Figure \ref{fig:fig7}(b) the probability density 
\benn
p(y,z)\qquad \text{on $\Sigma^J$\;.}  
\eenn
The density has been computed by recording the intersections with $\Sigma^J$
after passage through the 
folded node for $4000$ sample paths that have been integrated for a time
$s\in[0,20]$. The corresponding 
deterministic point measure $p_{\det}$ has been indicated as well. The density
$p(y,z)$ clearly shows 
that paths are likely to jump before the deterministic solution if we consider
the $z$-coordinate 
distance from the folded node. We also see that the density $p(y,z)$ is quite
concentrated and shows a 
multi-modal structure. This structure can be explained from the fact that sample
paths 
exit early but between different exit points they can make additional
deterministic small oscillations. 
The possible different numbers of these oscillations correspond to the different maxima of
$p(y,z)$. 

\section{Final Remarks}
\label{sec:discussion}

In Section \ref{sec:SDEfoldednode} we stated our results on the relation between
the noise level, 
the parameter $\mu$ and the typical spreading of sample paths. Note that we proved 
and stated our results in blown-up (or re-scaled) 
coordinates removing the $\epsilon$-dependence. In particular, we worked in a
neighbourhood of the 
folded node that is of size $\cO(\sqrt\epsilon)$ in original coordinates. To
obtain the results in original 
coordinates one has to apply a blow-down transformation. First, we replace
$(x,y,z,\sigma,\sigma')$ by 
$(\bar{x},\bar{y},\bar{z},\bar{\sigma},\bar{\sigma}')$ (recall: we dropped the
overbars for notational convenience), cf.~\eqref{eq:rescale1} and \eqref{eq:rescale_noise}. 
Then the identity
\benn
(\bar{x},\bar{y},\bar{z},\bar{\sigma},\bar{\sigma}',\mu)=
(\epsilon^{-1/2}x,\epsilon^{-1}
y,\epsilon^{-1/2}z,\epsilon^{-3/4}\sigma,\epsilon^{-3/4}\sigma',\mu)
\eenn
provides the required blow-down transformation. This implies, e.g., that the
interaction of canards in 
Corollary \ref{cor:noisySAOs} is given, in original coordinates, by relations
of the form 
\benn
\sigma\approx \epsilon^{3/4}\mu^{1/4} e^{-c_0(2k+1)^2\mu}
\eenn
or that sample paths are likely to escape for
\benn
z\gg \sqrt{\frac{\mu\epsilon}{\kappa}\left|\log\sigma-\frac34\log\epsilon
\right|}
\eenn
as shown in Theorem \ref{thm:escape_canards}. Obviously one also has to
translate the assumptions in a similar 
way, e.g., $\bar{\sigma}\ll \mu^{3/4}$ becomes $\sigma\ll (\mu\epsilon)^{3/4}$.\\

\begin{figure}[tbp]
\includegraphics[width=1\textwidth,clip=true]{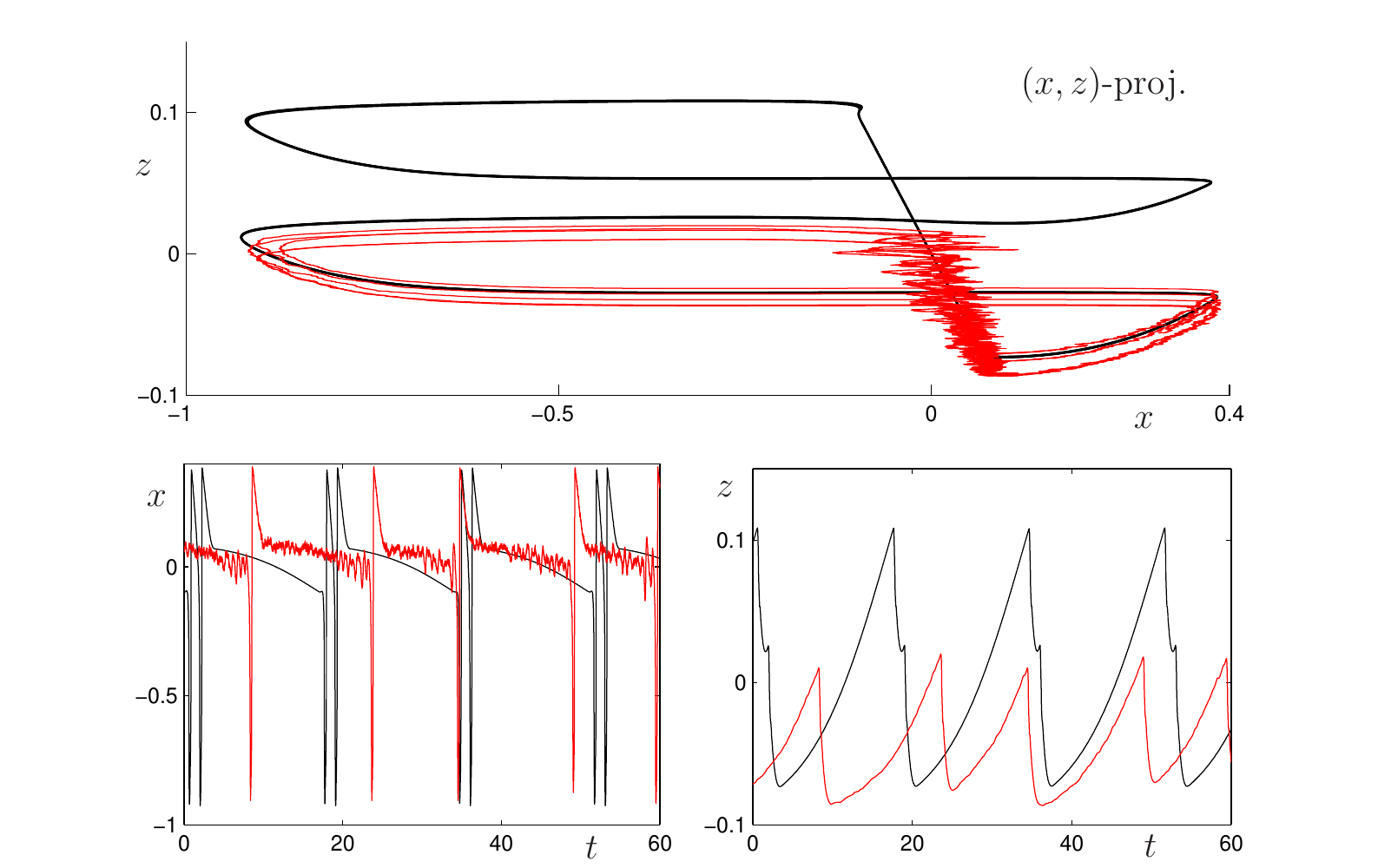}
	\caption{\label{fig:fig9} Simulation for the model system
\eqref{eq:BKW_mod} with parameter 
	values \eqref{eq:paras_jumps_2to1}. The upper plot shows a projection of
a deterministic solution (black) 
	and a stochastic sample path (red) into the $(x,z)$-plane. The lower two
plots show the associated 
	time series.}	
\end{figure}

Another important point is that we focused on the detailed analysis near the
folded node and 
did not consider different types of global return mechanisms. In Figure
\ref{fig:fig9} we show 
the interaction between global returns and noise-induced early jumps on
oscillatory  patterns. 
The simulation in Figure \ref{fig:fig9} has been carried out using the model
system \eqref{eq:BKW_mod} with parameter 
values 
\be
\label{eq:paras_jumps_2to1}
\epsilon=0.01\;, \quad \mu=0.029\;, \quad a=-0.1\;, \quad b=-0.5\;, \quad
\sigma=0.005\;, \quad \sigma'=0.005\;.
\ee
The deterministic solution in Figure \ref{fig:fig9} is an MMO with pattern $2^s$
where the number 
of small oscillations is difficult to count from the numerical results as $\mu$
is already extremely 
small. The SDE sample paths jump significantly earlier to $C^{a,-}_\epsilon$
than the deterministic solution as expected from our results. 
However, we see that the deterministic solution makes an additional large-amplitude oscillation (LAO)
given by the passage 
\benn
C^{a,-}_\epsilon\ra \text{jump near $L_-$ } \ra C^{a,+}_\epsilon \ra \text{jump
near $L_+$ } \ra C^{a,-}_\epsilon\;.
\eenn  
In particular, the early jumps change the number of LAOs in the MMO pattern from
$L=2$ to $L=1$. The noise also 
influences the number of SAOs but the crucial point is that it can also have a
global effect. Hence 
we end up with a stochastic-resonance-type mechanism for MMO patterns.\\

Furthermore, one could think about extending our results to capture the effect
of a global return map 
on the escape density after the passage through a folded-node region; see also
the brief discussion in 
Section \ref{sec:visual2}. Let $\Sigma^J$ denote a cross-section on which we
record the escape from 
the folded-node region and let $\Sigma^R$ denote a cross-section to the
deterministic flow slightly 
before the re-entry to the folded-node region; see also \cite{KuehnRetMaps}.
Then we have a global return map 
induced by the deterministic flow
\be
M: \Sigma^J\ra \Sigma^R\;.
\ee
If we are given a probability density $p$ on $\Sigma^J$, we can then consider the
induced density $p\circ M^{-1} / \abs{\det DM\circ M^{-1}}$ on $\Sigma^R$. 
This should allow us to calculate a distribution for different MMO patterns,
i.e.,
we can hope to assign a probability to each combination $L^s$ after one return
from 
$\Sigma^J$ to $\Sigma^J$. Although this approach seems possible it is beyond the
local analysis 
we focused on here. A solution of this problem crucially depends on the form of
the global returns which are described by 
the map $M$. 
   

\newpage

\appendix

\section{Proof of Theorem \ref{thm:var1} (Canonical Form)}
\label{appendix:canonical_form}

We are going to need the following lemma for the proof of Theorem 
\ref{thm:var1}.

\begin{lem}
\label{lem:Smale}
Consider two non-autonomous vector fields $F,G:\R\times \R^N\ra \R^N$ for 
$(z,X)\in \R\times \R^N$. Suppose that on an open set $\mathcal{D}$, containing $(0,X_0)$,
both are continuous in $z$ and $C^1$ in $X$. Furthermore, suppose 
for all $(z,X)$ in $\mathcal{D}$ we have
\benn
\|F(z,X)-G(z,X)\|<\mu\;.
\eenn 
Let $K$ be a Lipschitz constant for $F$ in $X$, which is uniform in $z$. Suppose
$X(z)$ and $Y(z)$ solve $\frac{\6X}{\6z}=F(z,X)$ and $\frac{\6Y}{\6z}=G(z,Y)$
with $X(0)=X_0=Y(0)$, then
\be
\label{eq:Smale}
\|X(z)-Y(z)\|\leqs \frac{\mu}{K}\left(\e^{K|z|}-1\right)\;.
\ee
\end{lem}

\begin{proof}
A direct Gronwall-lemma argument suffices; for details see
\cite{HirschSmaleDevaney}, 
p.~399--400.
\end{proof}

The proof of Theorem \ref{thm:var1} proceeds in several steps that aim at
bringing Equation~\eqref{eq:main_ex_var3} into diagonal
form~\cite{Berglund_Thesis,BK2}.

\begin{proof}[Proof of Theorem \ref{thm:var1}]
As a first step we apply a rescaling 
\benn
u=\exp\left[\frac{1}{2\mu}\int_0^z \text{Tr}(A(s))\6s\right]u^{(0)}
=\e^{z^2/\mu}u^{(0)}
\eenn
where $u^{(0)}=(u_1^{(0)},u_2^{(0)})$ are new coordinates. This yields 
\be
\label{eq:var1_pf1}
\mu\frac{\6u^{(0)}}{\6z}=A_0(z)u^{(0)},\qquad \text{with }
A_0(z)=\begin{pmatrix} 2 z & 2 \\ -2(\mu+1) & -2z\\ \end{pmatrix}
\ee
where now $\text{Tr}(A_0(z))=0$. Therefore the principal solution has
determinant 
$1$ and is area preserving. We are going to show that the solution of 
\eqref{eq:var1_pf1} is a rotation up to a small error using a sequence of 
$z$-dependent coordinate transformations $S_j(z)$. We set
\benn
u^{(0)}=S_0(z)u^{(1)}, \qquad \text{where }
S_0(z)=\begin{pmatrix} \frac{-z-\icx\omega(z)}{1+\mu} & 
\frac{-z+\icx\omega(z)}{1+\mu} \\ 1 & 1\\ \end{pmatrix}\;.
\eenn
Then $S_0^{-1}A_0S_0$ is diagonal with entries $\pm 2\icx\omega(z)$. We get
\benn
\mu\frac{\6u^{(1)}}{\6z}=A_1(z)u^{(1)}\;,
\eenn
where the new matrix $A_1$ is given by
\benn
A_1(z)=S_0^{-1}A_0S_0-\mu S_0^{-1}\frac{\6S_0}{\6z}=
\begin{pmatrix} 2\icx\omega(z)+\mu \frac{\icx-\omega'(z)}{2\omega(z)} 
& \mu\frac{\icx+\omega'(z)}{2\omega(z)} \\
\mu\frac{-\icx+\omega'(z)}{2\omega(z)} 
& -2\icx\omega(z)+\mu\frac{-\icx-\omega'(z)}{2\omega(z)}\\ \end{pmatrix}\;.
\eenn
As above we want to have a zero trace so we compute
\benn
\frac{1}{2\mu}\int_0^z \text{Tr}(A_1(s))\6s=
-\frac{1}{2}\int_0^z\frac{\omega'(s)}{\omega(s)}\6s=
-\frac{1}{2}\log\omega(z) + \const\;.
\eenn
Hence we consider the scaling
\benn
u^{(1)}=S_1(z)u^{(2)}=\frac{1}{\sqrt{\omega(z)}}u^{(2)}
\eenn
which yields the equation
\benn
\mu\frac{\6u^{(2)}}{\6z}=A_2(z)u^{(2)},\qquad \text{with }
A_2(z)=\begin{pmatrix} \icx\omega_2(z) & \mu\bar{\rho}_2(z) \\ 
\mu\rho_2(z) & -\icx\omega_2(z)\\ \end{pmatrix}
\eenn
where overbar denotes complex conjugate and 
\be
\label{eq:omega2}
\omega_2(z)=2\omega(z)+\frac{\mu}{2\omega(z)}\;, \qquad \rho_2(z)
=\frac{-\icx+\omega'(z)}{2\omega(z)}\;.
\ee
The next transformation 
\benn
u^{(2)}=S_2(z)u^{(3)}
\eenn
yields
\be
\label{eq:var1_pf3}
\mu\frac{\6u^{(3)}}{\6z}=A_3(z)u^{(3)},\qquad A_3(z)=
S_2^{-1}\left[A_2S_2-\mu\frac{\6S_2}{\6z}\right]\;.
\ee
Instead of a given transformation we now impose the form of the matrices
\be
\label{eq:var1_pf_aux}
S_2(z)=\begin{pmatrix} 1 & \mu \bar{v}(z)\\ \mu v(z) & 1\\ 
\end{pmatrix},\qquad 
A_3(z)= \begin{pmatrix} \rho_1(z) & 0 \\ 0 & \bar{\rho}_1(z) 
\\ \end{pmatrix}
\ee
so that $S_2=\Id+\cO(\mu)$ and $A_3$ is diagonal. Substituting 
\eqref{eq:var1_pf_aux} into \eqref{eq:var1_pf3} leads to the equations 
\benn
\begin{array}{rcl}
0 & = & \icx\omega_2(z)+\mu^2\bar{\rho}_2(z)v-\rho_1(z)\\
\mu \frac{\6v}{\6z}&=& \rho_2(z)-i\omega_2(z)v-v\rho_1(z)\\
\end{array}
\eenn
and their complex conjugates. The first equation determines 
$\rho_1(z)$, and thus the second one becomes
\be
\label{eq:Ricatti}
\mu \frac{\6v}{\6z}=-2\icx\omega_2(z)v-\mu^2\bar{\rho}_2(z)v^2+\rho_2(z)\;.
\ee
If we can show that \eqref{eq:Ricatti} has a bounded solution for times 
$|z|<1$ then our prescribed coordinate change $S_2$ exists. 
Now let $a(z)=2z+\re \rho_1(z)=2z+\Order{\mu^2}$ and define $\alpha(z,z_0)$
by~\eqref{eq:main_sol_var2}. Then a last transformation 
\begin{equation}
\label{eq:complextoreal}
 u^{(3)}=S_3(z) \tilde u =
\e^{-\alpha(z,0)/\mu}
\frac{1}{1+\icx}
\begin{pmatrix}
\icx  & 1 \\ 
1 & \icx  
\end{pmatrix}\tilde u
\end{equation} 
will bring the equation into canonical form~\eqref{eq:canonical_form_diag},
with $\varpi(z)=\im\rho_1(z)=2\omega(z)+\Order{\mu}$. 
Composing all transformations, we get $u=S(z)\tilde u$, where
$S(z)=\e^{z^2/\mu}S_0(z)S_1(z)S_2(z)S_3(z)$ is indeed of the
form~\eqref{eq:ccA}.

To prove the existence of bounded solutions for 
\eqref{eq:Ricatti} note that $\rho_2(z)$ and $\omega_2(z)$ are bounded 
away from zero and that they are also have bounded norms for $|z|<1$. Now set
\benn
F(z,v):=-\frac{2\icx\omega_2(z)}{\mu}v+\frac{\rho_2(z)}{\mu}\;, \qquad
G(z,v):=-\frac{2\icx\omega_2(z)}{\mu}v-\mu\bar{\rho}_2(z)v^2+\frac{\rho_2(z)}{
\mu}\;.
\eenn
Note that $F$ is a vector field with bounded solutions on an $\cO(1)$-time
scale. 
Considering $F,G$ as real planar vector fields on $\R^2$, we can apply Lemma 
\ref{lem:Smale} to conclude that \eqref{eq:Ricatti} admits bounded solutions 
on time intervals of length $1$.
\end{proof}

\section{Proof of Theorem \ref{thm:covariance_tubes} (Covariance Matrix)}
\label{appendix:covariance_tubes} 

In this Appendix, we discuss the equation 
\be
\label{eq:cov_matrix_repeated}
\mu\frac{\6V}{\6z}
=A(x^{\det}_z)V+VA(x^{\det}_z)^T+(F^0)(F^0)^T
\ee
describing the evolution of
$V(z)=\{V_{ij}(z)\}_{i,j\in\{1,2\}}:=\sigma^{-2}\Cov(z)$,  
the covariance matrix of the linearized variational equation 
around a deterministic solution of \eqref{eq:SDE_bu1}. Before proving Theorem
\ref{thm:covariance_tubes} on the small-$\mu$ asymptotics of the solutions, we
provide different approaches yielding information on the behaviour of $V(z)$. A
formal method based on iterative computations of a slow manifold is developed in
Section \ref{sec:iteration} to understand the asymptotics of
\eqref{eq:cov_matrix_repeated}. In Section \ref{sec:lyapunov} we provide
rigorous bounds using a Lyapunov function. In Section \ref{sec:DelayedHopf} we
refine the previous results by a transformation to real canonical form and
results about delayed Hopf bifurcation, thereby proving the main theorem. 

\begin{prop}
\label{prop:compVarB}
Let 
\begin{equation}
\label{eq:def_cov}
v(z) = 
\begin{pmatrix}
v_1(z) \\ v_2(z) \\ v_3(z) 
\end{pmatrix}\;,
\qquad
\text{where }
\begin{cases}
v_1(z) = V_{11}(z)\;, \\
v_2(z) = V_{22}(z)\;, \\
v_3(z) = V_{12}(z) = V_{21}(z)
\end{cases}
\end{equation}
denote the two variances and the covariance. Then $v(z)$ satisfies the ODE
\be
\label{eq:varBz}
\mu\frac{\6v}{\6z}=\underbrace{\begin{pmatrix} 
-8x(z) & 0 & 4 \\ 0 & 0 & -4(\mu+1) \\ -2(\mu+1) & 2 & -4x(z) \\
\end{pmatrix}}_{=:B(z)}v+\underbrace{\begin{pmatrix} 
2 \\ 2\rho^2 \\ 0 \\ \end{pmatrix}}_{=:E}
\ee
where we have abbreviated $x^{\det}(z)=:x(z)$. 
\end{prop}

\begin{proof} 
Using the definitions from \eqref{eq:lin_var_SDE} and equation 
\eqref{eq:cov_matrix_repeated} we get
\begin{align*}
\mu\frac{\6V}{\6z}&=A(x^{\det}_z)V+VA(x^{\det}_z)^T+(F^0)(F^0)^T\\
&= \begin{pmatrix} -4x^{\det}_z & 2 \\ 
-2(\mu+1) & 0 \\\end{pmatrix} V+V 
\begin{pmatrix} -4x^{\det}_z & -2(\mu+1) \\ 
2 & 0 \\\end{pmatrix}+
\begin{pmatrix} \sqrt{2} & 0 \\ 0 & \sqrt{2}\rho \\\end{pmatrix}
\begin{pmatrix} \sqrt{2} & 0 \\ 0 & \sqrt{2}\rho \\ \end{pmatrix}\\
&=
\begin{pmatrix}4V_{12}-8V_{11}x^{\det}_z +2 & 
2V_{22}-4V_{12}x^{\det}_z-2V_{11}(1+\mu) \\ 
2V_{22}-4V_{12}x^{\det}_z-2V_{11}(1+\mu) &
-4V_{12}(1+\mu) +2\rho^2\\ \end{pmatrix}\;.
\end{align*}
Therefore the result follows.
\end{proof}

Our goal is to 
analyze \eqref{eq:varBz} for a given maximal canard solution $x(z)$. 
Observe that $B(z)$ has eigenvalues 
\benn
-4x(z),\qquad -4x(z)\pm 4i\sqrt{1-x(z)^2+\mu}=-4x(z)\pm 4\icx\omega_0(-x(z))\;.
\eenn
We assume that $z_0, z_1$ are chosen so that  
\be
\label{eq:condB}
-1<z_0< 0< z_1<1 \qquad\text{and}\qquad 1-x(z)^2+\mu>0
\qquad\forall z\in[z_0,z_1]\;.
\ee
In particular, the assumptions \eqref{eq:condB} are satisfied for 
any maximal canard solution approaching the folded-node region from 
the slow manifold $C^a_\epsilon$ for some $z_0<0<z_1$ of order $1$ and $\mu$ 
sufficiently small. 

\subsection{Iteration and Asymptotics}
\label{sec:iteration}

\begin{notation}
Henceforth, we write $x(z,\mu)\asymp y(z,\mu)$ if 
\begin{equation}
 \label{eq:asympt_notation}
c_-  y(z,\mu) \leqs x(z,\mu) \leqs c_+  y(z,\mu)
\end{equation} 
holds for all $z$ with positive constants $c_\pm$ independent of $z$
and $\mu$. 
\end{notation}

A formal derivation for the asymptotics as $\mu\ra 0$ for \eqref{eq:varBz} 
can be carried out using an iterative scheme \cite{Neishtadt1,Berglund_Thesis}. 
We set $v(z)=V_0^*(z)+V_1(z)$ where 
\be
\label{eq:iter1}
V_0^*(z):=-B(z)^{-1}E= 
-\frac{1}{4(1+\mu)x(z)}
\begin{pmatrix}
1+\mu+\rho^2 \\ 
(1+\mu)^2+(1+4x(z)^2+\mu)\rho^2\\
2\rho^2 x(z)\\  
\end{pmatrix}\;
\ee
defines the critical manifold for \eqref{eq:varBz} when 
viewed as a slowly time-dependent system. We get
\benn
\mu\frac{\6V_1}{\6z}=B(z)V_1+\mu E_1(z),\qquad E_1(z)=-\frac{\6}{\6z}V^*_0(z)\;.
\eenn
The same change procedure also works for any $n\geqs 1$ by setting
\be
\label{eq:slow_varB}
v(z)=\sum_{j=0}^n \mu^j V^*_j(z)+V_{n+1}(z)\;.
\ee
Then $V_{n+1}(z)$ satisfies the equation
\be
\label{eq:slow_varC}
\mu\frac{\6V_{n+1}}{\6z}=B(z)V_{n+1}+\mu^{n+1}E_{n+1}\;,
\ee
where $V^*_n$ and $E_n$ are given inductively by
\benn
V^*_{n+1}(z)=B(z)^{-1}\frac{\6}{\6z}V^*_n(z), \qquad 
E_{n+1}(z)=\frac{\6}{\6z}[B(z)^{-1}E_n(z)]\;.
\eenn

\noindent 
\textit{Remark:} Observe that \eqref{eq:slow_varB} is the asymptotic 
expansion for the slow manifold of \eqref{eq:varBz}. The iterative 
scheme we use here is very convenient for slowly time-dependent 
systems. Many other methods to calculate slow manifolds for general 
fast--slow systems have been explored; see 
\cite{ZagarisKaperKaper1,ZagarisKaperKaper2} and references therein.

\begin{prop}
\label{prop:asympB}
Assume that the deterministic maximal canard solution $x(z)$ satisfies 
\eqref{eq:condB}. Then the asymptotic expansion 
\eqref{eq:slow_varB} of $v(z)$ for $n\geqs 0$ has components of order
\be
\label{eq:asymp_mainB}
\mu^n V^*_{n,1}(z)\asymp \mu^n V^*_{n,2}(z)\asymp 
\frac{\mu^n}{|z|^{2n+1}}, \qquad \mu^n V^*_{n,3}(z)\asymp 
\frac{\mu^n}{|z|^{2n}}\;.
\ee
\end{prop}

\begin{proof}
First observe that the symmetry \eqref{eq:symmetry} implies that $x(0)=0$ 
for any maximal canard. Using this fact and the form of the slow flow 
\eqref{eq:slowflow_desing_ex}, we find that $x(z)$ must have a Taylor 
expansion with non-vanishing linear term, i.e.,
\be
\label{eq:TaylorB}
x(z)=x_1 z+x_2z^2+\cdots
\ee
with $x_1<0$. The proof of \eqref{eq:asymp_mainB} then proceeds by 
induction as follows: The base step $n=0$ holds by formulas \eqref{eq:iter1}
and 
\eqref{eq:TaylorB}. The induction step from $n-1$ to $n$ follows from 
direct differentiation
\benn
\frac{\6}{\6z}\left(\frac{1}{z^{2n-1}}\right)=-\frac{2n-1}{z^{2n}}\;,\qquad
\eenn
and the calculation of $B(z)^{-1}$
\benn
B(z)^{-1}=- 
\begin{pmatrix}\frac{1}{8x(z)} & 
\frac{1}{8(1+\mu) x(z)} & 0 \\ \frac{1+\mu}{8x(z)} 
& \frac{1+\mu}{8x(z)}+\frac{x(z)}{2(1+\mu)} & 
\frac12 \\ 0 & \frac{1}{4(1+\mu)} & 0 
\end{pmatrix}
\eenn
almost immediately; we just have to observe the block structure of 
$B(z)^{-1}$.
\end{proof}

Proposition \ref{prop:asympB} is a formal asymptotic result. The 
asymptotic series \eqref{eq:slow_varB} becomes \lq\lq disordered\rq\rq\ for 
$|z|=\cO(\sqrt\mu)$, because in this case all the terms for the 
coordinates $v_1$ and $v_2$ are of order $1/\sqrt\mu$, while all 
terms for $v_3$ are of order $1$. Therefore we conjecture that 
\be
\label{eq:conjecB}
v_1=\cO\left(\frac{1}{\sqrt\mu}\right)\;, \qquad
v_2=\cO\left(\frac{1}{\sqrt\mu}\right)\;, \qquad
v_3=\cO(1)\;,
\ee
for $-\sqrt{\mu}<z_0\leqs z\leqs \sqrt{\mu}$.

\subsection{Lyapunov Function}
\label{sec:lyapunov}

The results in this section are not as sharp as the results obtained 
by coordinate changes in Section \ref{sec:DelayedHopf}, but they are 
obtained by a completely different technique which is of interest on 
its own in the context of folded nodes. Next 
we are going to establish an auxiliary result needed below in the proof of Proposition \ref{prop:Lyapunov}.

\begin{lem}
\label{lem:aux_Lya}
Consider the linear non-autonomous differential equation on $\R$ given by
\be
\label{eq:aux_lem_eq}
\mu\frac{\6X}{\6z}=k_1zX+k_2\frac{\mu^n}{(-z)^{2n}}
\ee
where $k_{1,2}=\cO(1)$ are two positive constants, $\mu>0$, and 
either $n=0$ and $z\geqs z_0$ or $n\geqs1$ and $z_0\leqs z<0$. 
Then 
\benn
X(z)\asymp 
\begin{cases}
\mu^n \abs{z}^{-(2n+1)}
& \text{for $z_0 + \Order{\mu\abs{\log\mu}}\leqs z\leqs -\sqrt{\mu}$\;,} \\
\mu^{n-1} \abs{z}^{-(2n-1)}
& \text{for $-\sqrt{\mu} \leqs z < 0$ if
$n\geqs1$\;,} \\
\mu^{-1/2}
& \text{for $-\sqrt{\mu} \leqs z \leqs \sqrt{\mu}$ if
$n=0$\;.} 
\end{cases}
\eenn
\end{lem}

\begin{proof}
The solution of \eqref{eq:aux_lem_eq} can be written as 
\be
\label{eq:auxl1}
X(z)=X(z_0)\e^{k_1(z^2-z_0^2)/2\mu}+k_2I_n(z)
\ee
where 
\benn
I_n(z)=\mu^{n-1}\int_{z_0}^z \e^{k_1(z^2-t^2)/2\mu}\frac{1}{(-t)^{2n}}\6t\;.
\eenn
For $z\in[z_0,-\sqrt\mu\,]$ the leading-order asymptotics of \eqref{eq:auxl1}
is 
given by $I_n(z)$. Using integration by parts, we get an upper bound
\begin{align}
\nonumber
k_1I_n(z)&=k_1\mu^{n-1}\int_{z_0}^z
\left(-\frac{\mu}{k_1t}\frac{1}{(-t)^{2n}}\right) 
\left(-\frac{k_1t}{\mu}\e^{k_1(z^2-t^2)/2\mu}\right)\6t\\
\nonumber
&=\left. \mu^n\frac{1}{(-t)^{2n+1}}\e^{k_1(z^2-t^2)/2\mu} \right|_{z_0}^z -
\mu^n\int_{z_0}^z \frac{\6}{\6t}\left[\frac{1}{(-t)^{2n+1}} \right] 
\e^{k_1(z^2-t^2)/2\mu}\6t\\
\label{eq:auxl1B} 
&=
\frac{\mu^n}{(-z)^{2n+1}}-\frac{\mu^n}{(-z_0)^{2n+1}}\e^{k_1(z^2-z_0^2)/2\mu}
-(2n+1)I_{n+1}(z)\\
&\leqs \frac{\mu^n}{(-z)^{2n+1}}\;.
\nonumber
\end{align}
The lower bound follows by inserting the upper bound for $I_{n+1}(z)$ in~\eqref{eq:auxl1B}. Note that here the condition $z\geqs z_0 + \Order{\mu\abs{\log\mu}}$ is needed to make the term
$\e^{k_1(z^2-z_0^2)/2\mu}$ small. Let us also remark that, in particular, we showed $X(-\sqrt{\mu})\asymp1/\sqrt{\mu}$. 

Finally, to describe the behaviour for $-\sqrt{\mu}<z<0$, we 
replace $z_0$ by $-\sqrt{\mu}$ in~\eqref{eq:auxl1}. Then all
exponential terms are of order $1$, and the integral can be estimated directly.
\end{proof}

\begin{prop}
\label{prop:Lyapunov}
Suppose \eqref{eq:condB} holds, and let $x(z)$ be a maximal canard 
solution. Then solutions to the variational equation \eqref{eq:varBz} 
remain bounded by $\cO(1/(\abs{z}+\sqrt\mu))$ for $z_0\leqs z\leqs\sqrt{\mu}$. 
\end{prop}

\begin{proof}
Throughout the proof we are going to introduce several positive constants 
$c_j=\cO(1)$ for $j\in \N$, whose actual values do not influence the asymptotic 
result. As a first step, we want to find a symmetric matrix $M(z)$ such that
\be
\label{eq:Mmat_def}
B(z)^TM(z)+M(z)B(z)=-x(z)\Id\;.
\ee
This can simply be accomplished by solving the six algebraic equations 
\eqref{eq:Mmat_def}. We find that 
\benn
\left.M(z)\right|_{\mu=0}=
\frac{1}{64(1+3x^2)}
\begin{pmatrix}
7+12x^2 & 1+12x^2 & -12x \\
1+12x^2 & 7+64x^2+48x^4 & -16x(1+3x^2) \\
-12x & -16x(1+3x^2) & 4(3+18x^2)\\
\end{pmatrix}\;,
\eenn
where we have abbreviated $x=x(z)$. Since \eqref{eq:TaylorB} holds 
for maximal canards, it is straightforward to check that 
the matrix $M(z)$ is positive definite, uniformly in $\mu$ and $z$. 
Therefore it defines a family of quadratic forms
\be
\label{eq:qform}
Y_n:=V_n^TM(z)V_n
\ee
where $V_n$ is defined by \eqref{eq:slow_varB}. The quadratic form 
\eqref{eq:qform} satisfies
\benn
c_1 \|V_n\|^2\leqs Y_n(z)\leqs c_2\|V_n\|^2\;.
\eenn 
for some constants $c_1,c_2>0$. Essentially $Y_n$ will act as a 
Lyapunov function to bound $\|V_n\|$. To show this, we compute the 
derivative. Using \eqref{eq:Mmat_def} and~\eqref{eq:slow_varC}, we get
\begin{align}
\label{eq:Lya_diff}
\mu\frac{\6Y_n}{\6z} &= \mu\frac{\6V_n^T}{\6z}M(z)V_n+ \mu V_n^TM(z)
\frac{\6V_n}{\6z}+\mu V_n^T\frac{\6M}{\6z}V_n \nonumber\\
&=-x(z)V_n^TV_n+\mu V_n^T \frac{\6M}{\6z}V_n+\mu^n\left[
E_n(z)^TM(z)V_n+V_n^TM(z)E_n(z)\right]\;.
\end{align}
Since $\|E_n\|=\cO(|z|^{-2n})$ and $\left\|\frac{\6M}{\6z}\right\|$ is bounded,
we can find 
constants $c_3,c_4>0$ such that \eqref{eq:Lya_diff} implies
\benn
\mu\frac{\6Y_n}{\6z}\leqs c_3(-x(z)+\mu)Y_n+c_4
\frac{\mu^n}{(-z)^{2n}}\sqrt{Y_n}\;.
\eenn
Setting $Y_n=Z_n^2$, we find that the last inequality is equivalent to 
\be
\label{eq:lin_tb_solved}
\mu\frac{\6Z_n}{\6z}\leqs c_5(-x(z)+\mu)Z_n+c_6\frac{\mu^n}{(-z)^{2n}}\;.
\ee
Using \eqref{eq:TaylorB} and Lemma \ref{lem:aux_Lya} we obtain that for 
$z_0$ of order $-1$,
\be
Z_n(z)\leqs c_7\frac{\mu^n}{\abs{z}^{2n+1}} \qquad 
\text{for $z_0\leqs z< 0$\;.}
\label{eq:lin_tb_solved++}
\ee
Since $Z_n$ is equivalent to $\|V_n(z)\|$, this shows that~\eqref{eq:slow_varC}
is indeed an asymptotic expansion in powers of $\mu/z^2$ for $z\leqs
-\sqrt{\mu}$, and in particular
all components of $v(-\sqrt{\mu})$ are of order $1/\sqrt{\mu}$. To complete the
proof up to time $z=\sqrt{\mu}$, we simply apply~\eqref{eq:lin_tb_solved++} in
the particular case $n=0$ (that is, for $V_0=v$ and $E_0=E$). 
\end{proof}

In view of our conjecture \eqref{eq:conjecB}, the bound on the covariance 
provided by Proposition~\ref{prop:Lyapunov} is not sharp since 
we have not yet shown that $v_3=\cO(1)$.

\subsection{Delayed Hopf Bifurcation}
\label{sec:DelayedHopf}

To obtain a sharp bound on the covariance we consider a similar coordinate 
change idea as in Section \ref{sec:diagonal}. This procedure will give a 
variational equation for the covariance that has desirable symmetry 
properties. 

\begin{lem}
\label{lem:SDE_canonical}
There exists a linear coordinate change $\zeta^0=S(z)\tilde\zeta^0$
transforming the linearized SDE~\eqref{eq:lin_var_SDE} into
\begin{equation}
 \label{eq:SDE_canonical01}
\6\tilde\zeta^0_z = \frac1\mu \widetilde A(z) \tilde\zeta^0_z \6z 
+ \frac{\sigma}{\sqrt{\mu}} \widetilde F(z) \6W_z\;,
\end{equation} 
where $\widetilde A(z)$ is in canonical form 
\begin{equation}
 \label{eq:SDE_canonical02}
\widetilde A(z) =  
\begin{pmatrix}
a(z) & \varpi(z) \\ -\varpi(z) & a(z) 
\end{pmatrix}\;,
\end{equation} 
with $a(z)=-2x(z)+\Order{\mu^{2}}$ and $\varpi(z)=2\omega(z)+\Order{\mu}$. 
The matrix $\widetilde F(z)$ is positive definite, with eigenvalues bounded
below and above uniformly in $z$.  
\end{lem}
\begin{proof}
It suffices to apply the coordinate change $\zeta^0_z=S(z)\tilde\zeta^0_z$
constructed in the proof of Theorem~\ref{thm:var1} (with an obvious modification
due to the fact that $x(z)$ is not necessarily given by the weak canard). The
new diffusion coefficient is then given by $\widetilde F(z)=S(z)^{-1}F^0$. 
\end{proof}

A computation  analogous to the one in the proof
of Proposition~\ref{prop:compVarB} then yields 

\begin{lem}
\label{lem:real}
The covariance matrix of $\zeta^0_z$ is given by $\sigma^2 \widetilde V(z)$, 
where the matrix elements of  $\widetilde V(z)$ satisfy the system 
\be
\label{eq:cov_matrix1real}
\mu\frac{\6\tilde{v}}{\6z}=\underbrace{\begin{pmatrix} 2a(z) & 0 &
2\varpi(z) \\
0& 2a(z) & -2\varpi(z) \\ -\varpi(z) & \varpi(z) & 2a(z) \\
\end{pmatrix}}_{=:\widetilde{B}(z)} \tilde{v}+\widetilde{E}
\ee
where $\widetilde{E}$ vector in $\R^3$ with $\cO(1)$-components.
\end{lem}

It is already apparent from the form of \eqref{eq:cov_matrix1real} that 
the analysis of the variational equation simplifies. We can now prove
Theorem~\ref{thm:covariance_tubes}, which we restate as follows for convenience.

\begin{thm}[Theorem~\ref{thm:covariance_tubes}]
\label{thm:covariance} 
Suppose \eqref{eq:condB} holds. Then the solution $v=(v_1,v_2,v_3)$ for 
\eqref{eq:varBz} satisfies the following asymptotics as $\mu\ra 0$
\be
\label{eq:thm_asympB}
v_1\asymp\frac{1}{\abs{z}+\sqrt\mu}, \qquad v_2\asymp\frac{1}{\abs{z}+\sqrt\mu},
\qquad v_3=\cO(1), \qquad 
(v_1-v_2)=\cO(1)
\ee
for $z_0 + \Order{\mu\abs{\log\mu}} \leqs z \leqs \sqrt{\mu}$. 
\end{thm}

\begin{proof}
We work in the coordinates provided by Lemma \ref{lem:real}. Summing the first
two equations of \eqref{eq:cov_matrix1real} 
we get 
\be
\label{eq:delayedHopf_eq0} 
\mu \frac{\6}{\6z}(\tilde{v}_1+\tilde{v}_2) = 2a(z)(\tilde{v}_1+\tilde{v}_2)  
+\tilde{e}_1+\tilde{e}_2  \;,
\ee
and we already know from Proposition \ref{prop:Lyapunov} (resp.\  
Lemma \ref{lem:aux_Lya}) that this yields 
\benn
(\tilde{v}_1+\tilde{v}_2)(z)\asymp\frac{1}{\abs{z}+\sqrt{\mu}}
\eenn
for $z_0 + \Order{\mu\abs{\log\mu}} \leqs z \leqs \sqrt{\mu}$. The difference of
the first two equations 
in \eqref{eq:cov_matrix1real} and the third equation can be combined as 
\be
\label{eq:delayedHopf_eq}
\mu \frac{\6}{\6z} \begin{pmatrix} \tilde{v}_1-\tilde{v}_2 \\ 
\tilde{v}_3 \\ \end{pmatrix} =
\begin{pmatrix} 2a(z) & 4\varpi(z) \\ -\varpi(z) & 2a(z) \\
\end{pmatrix}
\begin{pmatrix} \tilde{v}_1-\tilde{v}_2 \\ \tilde{v}_3 \\ 
\end{pmatrix} +
\begin{pmatrix} \tilde{e}_1 - \tilde{e}_2 \\ \tilde{e}_3 \\ 
\end{pmatrix}\;.
\ee
Considering \eqref{eq:delayedHopf_eq} as a fast--slow system with slow 
variable $z$ we find that the critical manifold is given by the equation  
\benn
\begin{pmatrix} (\tilde{v}_1-\tilde{v}_2)^* \\ \tilde{v}_3^*\\
\end{pmatrix} = -\begin{pmatrix} 2a(z) & 4\varpi(z) \\ 
-\varpi(z) & 2a(z) \\ \end{pmatrix}^{-1} \begin{pmatrix} 
\tilde{e}_1 - \tilde{e}_2 \\ \tilde{e}_3 \\ \end{pmatrix}\;.
\eenn
which is of order $1$. Observe that \eqref{eq:delayedHopf_eq} undergoes a 
delayed (or dynamic) Hopf bifurcation at $z=0$. Thus Neishtadt's theorem on 
delayed Hopf bifurcations \cite{Neishtadt1} applies, and shows that solutions 
of the variational equation satisfy 
\benn
(\tilde{v}_1-\tilde{v}_2)(z)=(\tilde{v}_1-\tilde{v}_2)^*(z) + \cO(\mu), \qquad
\tilde{v}_3(z)= v_3^*(z)+\cO(\mu)
\eenn  
for $z_0+\cO(\mu\abs{\log\mu})\leqs z\leqs \cO(1)$.
Now the result~\eqref{eq:thm_asympB}  follows from 
$V(z)=S(z)\widetilde V(z)S(z)^T$, by writing  $\widetilde V(z)$ as the sum of a
leading term proportional to the identity matrix and a remainder of order $1$.
\end{proof}

\section{Proof of Theorem~\ref{thm:nonlinear_SDE} (Staying in covariance tubes)}
\label{appendix:proof_nonlinear_SDE} 

Applying the transformation of Lemma~\ref{lem:SDE_canonical} to the nonlinear
equation~\eqref{eq:nonlin01} and dropping the tildes yields the system 
\begin{equation}
 \label{eq:p_nlSDE:01} 
\6\zeta_z = \frac{1}{\mu} 
\bigl[ A(z) \zeta_z + b(\zeta_z,z) \bigr] \6z
+ \frac{\sigma}{\sqrt{\mu}} F(z) \6W_z\;,
\end{equation} 
where 
\begin{equation}
 \label{eq:p_nlSDE:02}
A(z) = 
\begin{pmatrix}
a(z) & \varpi(z) \\ -\varpi(z) & a(z)
\end{pmatrix}\;,
\end{equation} 
and $b(\zeta,z) = \Order{\norm{\zeta}^2}$. 
The solution of~\eqref{eq:p_nlSDE:01} with initial condition $\zeta_{z_0}=0$
can be written as 
\begin{equation}
 \label{eq:p_nlSDE:03}
\zeta_z = \frac{\sigma}{\sqrt{\mu}} \int_{z_0}^z U(z,s) F(s) \6W_s 
+ \frac{1}{\mu} \int_{z_0}^z U(z,s) b(\zeta_s,s) \6s
=: \zeta^0_z + \zeta^1_z\;, 
\end{equation} 
where $U(z,s)$ denotes the principal solution of the time-dependent linear
system $\mu\dot\zeta=A(z)\zeta$. Owing to the particular form of $A(z)$, we
have the explicit expression 
\begin{equation}
 \label{eq:p_nlSDE:04}
U(z,s) = \e^{-\alpha(z,s)/\mu}
\begin{pmatrix}
\cos(\varphi(z,s)/\mu) & \sin(\varphi(z,s)/\mu) \\
-\sin(\varphi(z,s)/\mu) & \cos(\varphi(z,s)/\mu)
\end{pmatrix}\;,
\end{equation} 
where
\begin{equation}
 \label{eq:p_nlSDE:05}
\alpha(z,s) = \int_s^z -a(u)\6u\;, 
\qquad
 \varphi(z,s) = \int_s^z \varpi(u)\6u\;.
\end{equation} 
Note in particular that since $-a(z) \asymp x(z)\asymp -z$ near $z=0$, we have
$\alpha(z,s)\asymp s^2 - z^2$. 

For a two-by-two matrix $M$, let $\norm{M}$ denote its $L^2$-operator norm,
i.e., $\norm{M}^2$ is the largest eigenvalue of $MM^T$. 

\begin{lem}
\label{lem:p_nlSDE01} 
Let 
\begin{equation}
 \label{eq:p_nlSDE:06}
\Theta(z) = \frac{1}{\mu} \int_{z_0}^z \norm{U(z,s)} \6s \;.
\end{equation} 
Then 
\begin{equation}
 \label{eq:p_nlSDE:07}
\Theta(z) = \cO\left(\frac{1}{\abs{z} + \sqrt\mu}\right) 
\end{equation} 
for all $z_0\leqs z\leqs\sqrt{\mu}$. 
\end{lem}

\begin{proof}
Since $U(z,s)U(z,s)^T = \e^{-2\alpha(z,s)/\mu}\Id$, we have 
\begin{equation}
 \label{eq:p_nlSDE:08}
\Theta(z) = \frac{1}{\mu} \int_{z_0}^z \e^{-\alpha(z,s)/\mu} \6s\;,
\end{equation} 
and the result follows from Lemma~\ref{lem:aux_Lya}.
\end{proof}

The next lemma provides bounds on the norms of $\Vbar(z)$ and $\Vbar(z)^{-1}$. 

\begin{lem}
 \label{lem:p_nlSDE02}
Let 
\begin{equation}
 \label{eq:p_nlSDE:09}
K_+(z)^2 = \norm{\Vbar(z)}\;,
\qquad
K_-(z)^2 = \norm{\Vbar(z)^{-1}}\;. 
\end{equation}  
Then 
\begin{equation}
 \label{eq:p_nlSDE:10}
K_+(z)^2 =  \cO\left(\frac{1}{\abs{z} + \sqrt\mu}\right) \;, 
\qquad
K_-(z)^2 = \Order{\abs{z} + \sqrt\mu}
\end{equation} 
for all $z_0\leqs z\leqs \sqrt\mu$. 
\end{lem}

\begin{proof}
Note that 
\begin{equation}
 \label{eq:p_nlSDE:11}
\Vbar(z)\Vbar(z)^T = 
\begin{pmatrix}
\vbar_1 & \vbar_3 \\ \vbar_3 & \vbar_2
\end{pmatrix}^2
= 
\begin{pmatrix}
\vbar_1^2 + \vbar_3^2 & (\vbar_1+\vbar_2)\vbar_3 \\
(\vbar_1+\vbar_2)\vbar_3 & \vbar_2^2 + \vbar_3^2
\end{pmatrix} 
\end{equation} 
has eigenvalues given by 
\begin{equation}
 \label{eq:p_nlSDE:12}
\frac{1}{2} 
\left[
\vbar_1^2 + \vbar_2^2 + 2\vbar_3^2 
\pm (\vbar_1+\vbar_2)\sqrt{(\vbar_1-\vbar_2)^2 + 4\vbar_3^2}\,
\right] \;.
\end{equation} 
The larger eigenvalue is equal to $K_+(z)^4$, while the smaller one is
equal to $K_-(z)^{-4}$. The result thus follows from the bounds obtained in
Theorem~\ref{thm:covariance}.
\end{proof}

We can now prove a local version of Theorem~\ref{thm:nonlinear_SDE}, on a small 
interval $[s,t]\subset[z_0,z]$. 

\begin{prop}
\label{prop:p_nlSDE01} 
Fix times $z_0\leqs s < t \leqs z$ such that $\alpha(t,s)\leqs\Order{\mu}$.
Then, for all $\mu$ and $\sigma$ small enough, for any $0<\gamma<1$, 
\begin{multline}
 \label{eq:p_nlSDE:13}
\P 
\left\{
\sup_{s\leqs u\leqs t} \pscal{\zeta_u}{\Vbar(u)^{-1}\zeta_u} \geqs r^2
\right\} \\
\leqs
\frac{1}{1-\gamma} \exp
\left\{
-\frac{\gamma r^2}{2\sigma^2}
\left[
1 - \cO\biggl((\abs{s}+\sqrt{\mu}\,)\frac{t-s}{\mu}\biggr) 
- \cO\biggl(\frac{r}{(\abs{t}+\sqrt{\mu})^{3/2}}\biggr) 
\right]
\right\}\;.
\end{multline}
\end{prop}

\begin{proof}
The proof is adapted from \cite[Section~5.1.2]{BGbook}, and we use almost the
same notations as there. Let $\Upsilon_u=U(s,u)\zeta_u$. Then 
\begin{equation}
 \label{eq:p_nlSDE:14}
 \pscal{\zeta_u}{\Vbar(u)^{-1}\zeta_u}
=
\pscal{\Upsilon_u}{\underbrace{U(u,s)^T\Vbar(u)^{-1}U(u,s)}_{=:Q_s(u)^2}
\Upsilon_u}
= \norm{Q_s(u)\Upsilon_u}^2\;.
\end{equation} 
We can decompose $\Upsilon_u=\Upsilon_u^0+\Upsilon_u^1$, where 
\begin{align}
\nonumber
\Upsilon_u^0 &= \frac{\sigma}{\sqrt{\mu}}
\int_{z_0}^u U(s,v)F(v)\6W_v\;, \\
\Upsilon_u^1 &= \frac{1}{\mu} 
\int_{z_0}^u U(s,v) b(\zeta_v,v)\6v\;.
\label{eq:p_nlSDE:15}
\end{align}
The process $\Upsilon_u^0$ is a Gaussian martingale. 
Lemma~5.1.8 in~\cite{BGbook} can thus be applied and provides the bound 
\begin{equation}
 \label{eq:p_nlSDE:16}
\P\left\{
\sup_{s\leqs u\leqs t}\norm{Q_s(t)\Upsilon^0_t} \geqs R_0
\right\} 
\leqs 
\frac{1}{1-\gamma} 
\exp\left\{-\gamma\frac{R_0^2}{2\sigma^2}\right\}\;.
\end{equation} 
For this bound to be useful, we need to show that $Q_s(u)$ and $Q_s(t)$ are
close to each other. Observe that 
\begin{equation}
 \label{eq:p_nlSDE:17}
\mu\frac{\6}{\6u} Q_s(u)^{-2} 
= U(s,u)F(u)F(u)^TU(s,u)^T\;, 
\end{equation} 
as a consequence of the definition of $U(s,u)$ and the differential equation
satisfied by $\Vbar(u)$. Integrating from $u$ to $t$ and multiplying on the
left by $Q_s(u)^2$, we get 
\begin{equation}
 \label{eq:p_nlSDE:18}
Q_s(u)^2Q_s(t)^{-2} - \Id = 
Q_s(u)^2 \frac{1}{\mu} \int_u^t  U(s,v)F(v)F(v)^TU(s,v)^T \6v\;.
\end{equation} 
Now $\norm{U(s,v)}\leqs\Order{1}$ owing to the assumption
$\alpha(t,s)=\Order{\mu}$. Thus the integral has order $t-u \leqs t-s$.
Furthermore,
\begin{equation}
 \label{eq:p_nlSDE:19}
\norm{Q_s(u)^2} \leqs \norm{U(u,s)}^2 K_-(u)^2 = \Order{K_-(u)^2}\;.
\end{equation} 
As a consequence, we get 
\begin{equation}
 \label{eq:p_nlSDE:20}
Q_s(u)^2 = Q_s(t)^2 
\left[
\Id + \cO \left( 
K_-(s)^2 \frac{t-s}{\mu}
\right)
\right] \;.
\end{equation} 
Thus there exists an $R=r[1-\Order{K_-(s)^2(t-s)/\mu}]$ such that 
\begin{equation}
 \label{eq:p_nlSDE:21}
 \P 
\left\{
\sup_{s\leqs u\leqs t} \pscal{\zeta_u}{\Vbar(u)^{-1}\zeta_u} \geqs r^2
\right\} 
\leqs 
\P 
\biggl\{
\sup_{s\leqs u\leqs t\wedge\tau_{\cB(r)}} \norm{Q_s(t)\Upsilon_u} \geqs R
\biggr\}\;.
\end{equation}
For any decomposition $R=R_0+R_1$ with $R_0,R_1>0$, we can bound the above
probability by $P_0+P_1$, where 
\begin{equation}
 \label{eq:p_nlSDE:22}
P_i =  \P 
\biggl\{
\sup_{s\leqs u\leqs t\wedge\tau_{\cB(r)}} \norm{Q_s(t)\Upsilon^i_u} \geqs R_i
\biggr\}\;,
\qquad
i=0, 1\;.
\end{equation}  
$P_0$ has already been estimated in~\eqref{eq:p_nlSDE:16}. 
Now we want to choose $R_1$ in such a way that $P_1=0$. For any $u \leqs t\wedge
\tau_{\cB(r)}$, we have 
\begin{align}
\nonumber
\norm{Q_s(t)\Upsilon^1_u}
&\leqs \const \sup_{u\in[s,t]}K_-(u)
\frac{1}{\mu} \int_{z_0}^u \norm{U(s,v)b(\zeta_v)} \6v \\
\nonumber
&\leqs \const \sup_{u\in[s,t]}K_-(u)\Theta(u)
\sup_{v\in[z_0,t\wedge\tau_{\cB(r)}] } \norm{\zeta_v}^2 \\
\nonumber
&\leqs \const \sup_{u\in[s,t]}K_-(u)\Theta(u)
\sup_{v\in[z_0,t] } K_+(v)^2r^2 \\
&\leqs \const \left( \sqrt{\mu}+\abs{t} \right)^{-3/2} r^2\;.
\label{eq:p_nlSDE:23}
\end{align}
We can thus achieve $P_1=0$ by simply choosing $R_1$ as a
sufficiently large constant times $(\sqrt{\mu}+\abs{t})^{-3/2} r^2$. This
determines $R_0$, and the result then follows from~\eqref{eq:p_nlSDE:16}. 
\end{proof}

We can now complete the proof of Theorem~\ref{thm:nonlinear_SDE}, which will
follow directly from

\begin{thm}
\label{thm:nonlinear_SDE_repeated}
There exist constants $\Delta_0, r_0,\mu_{0}>0$ such that for all $0<\Delta<\Delta_0$,
all $0<\sigma<r<r_0\mu^{3/4}$, $0<\mu<\mu_{0}$ and all $0<\gamma<1$, 
\begin{equation}
 \label{eq:nonlin03R}
\P \bigl\{ \tau_{\cB(r)} < z \bigr\} 
\leqs 
C_+(z,z_0) \exp \biggl\{ -\gamma \frac{r^2}{2\sigma^2} 
\bigl[ 1 - \Order{\Delta}  - \Order{r\mu^{-3/4}} \bigr]
\biggr\}
\end{equation}  
holds for all $z\leqs \sqrt{\mu}$, where
\begin{equation}
 \label{eq:nonlin04R}
C_+(z,z_0) = \frac{\const}{(1-\gamma)\Delta\mu} \int_{z_0}^{z} x^{\det}_s
\6s\;. 
\end{equation} 
\end{thm}
\begin{proof}
Let $z_0=s_0 < s_1 < \dots < s_N=z$ be a partition of $[z_0,z]$. Then
\begin{equation}
 \label{eq:p_nlSDE:24}
\P \bigl\{ \tau_{\cB(r)} < s \bigr\}  \leqs \sum_{k=1}^N P_k\;, 
\end{equation} 
where 
\begin{equation}
 \label{eq:p_nlSDE:25}
P_k =  \P 
\biggl\{
\sup_{s_{k-1}\leqs u\leqs s_k} \pscal{\zeta_u}{\Vbar(u)^{-1}\zeta_u} \geqs r^2
\biggr\}
\end{equation} 
can be estimated by Proposition~\ref{prop:p_nlSDE01}. We want to choose the
partition in such a way that the error terms in $P_k$ are bounded uniformly in
$k$. A convenient choice is to define the $s_k$ by 
\begin{equation}
 \label{eq:p_nlSDE:26}
\begin{cases}
\alpha(s_{k+1},s_k) = \Delta\mu 
& \text{when $s_{k+1}<-\sqrt{\mu}$\;,} \\
s_{k+1} - s_k = \Delta\sqrt{\mu}
& \text{when $\abs{s_{k+1}}\leqs\sqrt{\mu}$\;.} 
\end{cases} 
\end{equation} 
Using the fact that $\alpha(t,s)\asymp\abs{s+t}(t-s)$ and applying
Proposition~\ref{prop:p_nlSDE01}, one indeed checks that 
\begin{equation}
 \label{eq:p_nlSDE:27}
P_k \leqs \frac{1}{1-\gamma} \exp
\Bigl\{
-\frac{\gamma r^2}{2\sigma^2} 
\left[
1 - \Order{\Delta} - \Order{r\mu^{-3/4}}
\right]
\Bigr\}
\qquad
\forall k=1,\dots, N\;,
\end{equation} 
where the error terms are uniform in $k$. 
It remains to estimate the number $N$ of elements of the partition, which will
give the prefactor $C_+$. In the case $z\leqs -\sqrt{\mu}$, we simply have
\begin{equation}
 \label{eq:p_nlSDE:28}
\alpha(z,z_0) = N\Delta\mu
\quad\Rightarrow\quad
N = \biggl\lceil\frac{\alpha(z,z_0)}{\Delta\mu}\biggr\rceil\;. 
\end{equation} 
In the case $-\sqrt{\mu} \leqs z \leqs \sqrt{\mu}$, we have
\begin{equation}
 \label{eq:p_nlSDE:29}
 N = \biggl\lceil\frac{\alpha(z,z_0)}{\Delta\mu}\biggr\rceil +
\biggl\lceil\frac{z-(-\sqrt{\mu})}{\Delta\sqrt{\mu}}\biggr\rceil\;,
\end{equation} 
and the result follows from the fact that 
$\alpha(z,-\sqrt{\mu}) \asymp \sqrt{\mu}(z-(-\sqrt{\mu}))$. 
\end{proof}

Theorem~\ref{thm:nonlinear_SDE} is just a reformulation of this result, in which
we have chosen $\gamma=1-\sigma^2/r^2$. 


\section{Proof of Theorem \ref{thm:escape_canards} (Early Jumps)}
\label{appendix:proof_escape} 

We consider again the equation for the difference $\zeta_z$ between stochastic
sample paths and a deterministic reference solution, this time given by the
weak canard. In canonical form, we have 
\begin{equation}
 \label{eq:p_escape:01} 
\6\zeta_z = \frac{1}{\mu} 
\bigl[ A(z) \zeta_z + b(\zeta_z,z) \bigr] \6z
+ \frac{\sigma}{\sqrt{\mu}} F(z) \6W_z\;,
\end{equation} 
where 
\begin{equation}
 \label{eq:p_escape:02}
A(z) = 
\begin{pmatrix}
a(z) & \varpi(z) \\ -\varpi(z) & a(z)
\end{pmatrix}\;,
\qquad
a(z) = 2z + \Order{\mu^2}\;,
\end{equation} 
and $b(\zeta,z) = \Order{\norm{\zeta}^2}$. The proof of Theorem~\ref{thm:escape_canards} is split into several
parts. In Subsection~\ref{appendix:proof_escape_noise}, we show that sample
paths are likely to leave a neighbourhood of order slightly (that is,
logarithmically) larger than $\sigma/\sqrt z$ of the weak canard in a time $z$ of
order $\sqrt{\mu\abs{\log\mu}}$. Subsection~\ref{appendix:proof_escape_drift}
analyses the dynamics in a larger neighbourhood of the weak canard, in which
the drift term dominates. Subsection~\ref{appendix:proof_escape_Laplace}
combines the two results to prove the main theorem.


\subsection{Diffusion-Dominated Escape}
\label{appendix:proof_escape_noise} 

We assume from now on that $\sigma\ll\mu^{3/4}$, because otherwise stochastic
sample paths are no longer localised near deterministic solutions when
$z=\sqrt{\mu}$. We define the set 
\begin{equation}
 \label{p_esc_n:01}
\cS(h) = \left\{ (\zeta,z) \colon z\geqs\sqrt{\mu}, \norm{\zeta} < h
\hat\rho(z) \right\}\;, 
\end{equation} 
where 
\begin{equation}
 \label{p_esc_n:02}
\hat\rho(z)^2 = \frac{\Tr(F(z)F(z)^T)}{4z}\;. 
\end{equation} 
The following result is an adaptation of~\cite[Proposition~4.7]{BG1} to the
two-dimensional case.

\begin{prop}
\label{prop:escape_noise}
Let $h, \nu>0$, with $\nu$ of order $1$, satisfy the conditions 
\begin{equation}
 \label{p_esc_n:03}
\frac{\sigma}{h} \leqs c_0
\qquad\text{and}\qquad
\biggl( \frac{h}{\sigma} \biggr)^{2+\nu}
\biggl[\log \biggl( 1+\nu+\frac{h^2}{\sigma^2} \biggr)\biggr]^{1/2}
\leqs c_1  \frac{\mu^{3/4}}{\sigma} 
\end{equation}
for some $c_0, c_1>0$. If $c_0$ and $c_1$ are small enough, then 
there exist $T>0$ and $C(\nu)>0$ such that for any $(z_0,\zeta_0)\in\cS(h)$
with $z_0<T$,
\begin{equation}
 \label{p_esc_n:04}
\P^{(\zeta_0,z_0)} \bigl\{ \tau_{\cS(h)} \geqs z \bigr\}
\leqs C(\nu) \biggl( \frac{h}{\sigma} \biggr)^{2\nu} 
\exp \biggl\{ -\kappa(\nu) \frac{z^2 - z_0^2}{\mu}\biggr\} 
\end{equation} 
holds for all $\sqrt{\mu} \leqs z_0 \leqs z \leqs T$, where 
\begin{equation}
 \label{p_esc_n:05}
\kappa(\nu) = \frac{2\nu}{1+\nu}
\biggl[ 1 - \cO \biggl( \frac{1}{\nu\log(h/\sigma)} \biggr) \biggr]\;. 
\end{equation} 
\end{prop}

We shall choose the value of the parameter $\nu$ later on, while $h$ will be
taken of the form $h=c\sigma\abs{\log\sigma}$. Then
Condition~\eqref{p_esc_n:03} reduces to 
\begin{equation}
 \label{p_esc_n:06}
\sigma \abs{\log\sigma}^{2+\nu} \sqrt{\log\abs{\log\sigma}} \leqs 
\Order{\mu^{3/4}}\;, 
\end{equation} 
which is slightly stronger than requiring $\sigma\ll\mu^{3/4}$. The exponent
$\kappa(\nu)$ in~\eqref{p_esc_n:05} becomes optimal in the limit $\nu\to\infty$,
but Condition~\eqref{p_esc_n:06} becomes more stringent as $\nu$ grows large.
Note that we have to choose a finite $\nu$ anyhow.

\begin{proof}[Proof of Proposition~\ref{prop:escape_noise}]
Let $\alpha(t,s)=\int_s^t a(u)\6u$. 
We define a partition $z_0=s_0 < s_1 < \dots < s_N = z$ of $[z_0,z]$ by
\begin{equation}
 \label{p_esc_n:07}
\alpha(s_k,s_{k-1}) = \Delta\mu\;,
\qquad
k=1,\dots,N-1\;,
\end{equation}  
where $\Delta>0$ will be chosen later. The Markov property implies that 
\begin{equation}
 \label{p_esc_n:08}
\P \bigl\{ \tau_{\cS(h)} \geqs z \bigr\}
\leqs \prod_{k=1}^{N-1} P_k\;,
\end{equation} 
where
\begin{equation}
 \label{p_esc_n:09}
P_k = \sup_{\zeta\colon \norm{\zeta}\leqs h\hat\rho(s_{k-1})} 
\P^{\zeta,s_{k-1}} \biggl\{ \sup_{s_{k-1}\leqs s\leqs s_k}
\frac{\norm{\zeta_s}}{\hat\rho(s)} < h \biggr\}\;.
\end{equation} 
We shall derive a uniform bound $q(\Delta)$ for all $P_k$,  $1\leqs k\leqs
N-1$. Then~\eqref{p_esc_n:08} and the
definition~\eqref{p_esc_n:07} of the partition imply 
\begin{equation}
 \label{p_esc_n:10} 
\P \bigl\{ \tau_{\cS(h)} \geqs z \bigr\}
\leqs q(\Delta)^{-1} \exp \biggl\{ -\frac{\alpha(z,z_0)}{\mu}
\frac{\log q(\Delta)^{-1}}{\Delta}\biggr\}\;,
\end{equation} 
and the result will follow from an appropriate choice of $\Delta$. 

The process $\zeta_s$ starting at time $s_{k-1}$ in $\zeta$ can be decomposed
as $\zeta_s = \zeta_s^{k,0} + \zeta_s^{k,1}$, with
\begin{align}
\nonumber
\zeta_s^{k,0} &= U(s,s_{k-1})\zeta + \frac{\sigma}{\sqrt{\mu}}
\int_{s_{k-1}}^s U(s,u)F(u)\6W_u\;, \\
\zeta_s^{k,1} &= \frac{1}{\mu}
\int_{s_{k-1}}^s U(s,u)b(\zeta_u,u)\6u\;,
\label{p_esc_n:11}
\end{align}
where the principal solution  $U(s,u)$ of the linear system is given
in~\eqref{eq:p_nlSDE:04}. For any decomposition $h=H_0-H_1$, we have 
\begin{equation}
 \label{p_esc_n:12}
P_k \leqs  \sup_{\zeta\colon \norm{\zeta}\leqs h\hat\rho(s_{k-1})} 
\left[ P_{k,0}(\zeta,H_0) + P_{k,1}(\zeta,H_1) \right]\;,
\end{equation} 
where
\begin{align}
\nonumber
P_{k,0}(\zeta,H_0) &=  \P^{\zeta,s_{k-1}} 
\biggl\{ \sup_{s_{k-1}\leqs s\leqs s_k}
\frac{\norm{\zeta^{k,0}_s}}{\hat\rho(s)} < H_0 \biggr\}\;,\\
P_{k,1}(\zeta,H_1) &=  \P^{\zeta,s_{k-1}} 
\biggl\{ \sup_{s_{k-1}\leqs s\leqs s_k}
\frac{\norm{\zeta^{k,1}_s}}{\hat\rho(s)} \geqs H_1 , 
\sup_{s_{k-1}\leqs s\leqs s_k}
\frac{\norm{\zeta_s}}{\hat\rho(s)} < h \biggr\}\;.
\label{p_esc_n:13}
\end{align}
We start by bounding $P_{k,0}(\zeta,H_0)$, using the end-point estimate  
\begin{equation}
 \label{p_esc_n:14}
 P_{k,0}(\zeta,H_0) \leqs \P^{\zeta,s_{k-1}}
\bigl\{ \norm{\zeta^{k,0}_{s_k}} < H_0\hat\rho(s_k) \bigr\}
\leqs \frac{\pi H_0^2
\hat\rho(s_k)^2}{\sqrt{(2\pi)^2\det\Cov(\zeta^{k,0}_{s_k})}}
\;.
\end{equation} 
The last inequality follows from the fact that the random variable
$\zeta^{k,0}_{s_k}$ is Gaussian, and we have bounded its density by the
normalizing constant. 
We denote the diagonal matrix elements of $V=\sigma^{-2}\Cov(\zeta^{k,0}_{s_k})$
by $v_1$ and $v_2$, and the off-diagonal element by $v_3$. Then 
\begin{equation}
 \label{p_esc_n:15}
 \det\Cov(\zeta^{k,0}_{s_k}) = \sigma^4 (v_1v_2 - v_3^2) 
= \sigma^4 \biggl[ \frac14 ((\Tr V)^2 - (v_1-v_2)^2) - v_3^2 \biggr]\;.
\end{equation} 
As already remarked in the proof of Theorem~\ref{thm:covariance}, the quantities
$v_1-v_2$ and $v_3$ remain of order $1$ up to some $z=T$ of order $1$, as a
consequence of Neishtadt's result on delayed Hopf bifurcations. In order to
estimate the trace $\Tr V$, we use the fact that $\hat\rho(z)$ is decreasing in
$(0,T]$ for $T$ small enough, owing to the fact that $\Tr(F(z)F(z)^T)$ is
bounded below by a positive constant, and has a derivative bounded in absolute
value. Thus we have, cf.~\eqref{eq:delayedHopf_eq0}, 
\begin{align}
\nonumber
\Tr V &= \frac{1}{\mu} \int_{s_{k-1}}^{s_k} \e^{2\alpha(s_k,u)/\mu} 
\Tr(F(u)F(u)^T) \6u \\
\nonumber
&= \e^{2\Delta} \int_{s_{k-1}}^{s_k} \frac{4u}{\mu} 
\e^{-2\alpha(u,s_{k-1})/\mu} \hat\rho(u)^2 \6u \\
&\geqs \hat\rho(s_k)^2 \left[ \e^{2\Delta}-1 \right]\;.
\label{p_esc_n:16}
\end{align}
Substituting~\eqref{p_esc_n:16} in~\eqref{p_esc_n:15} and then
in~\eqref{p_esc_n:14} yields 
\begin{equation}
 \label{p_esc_n:17}
P_{k,0}(\zeta,H_0) 
\leqs \frac{H_0^2}{\sigma^2}
\frac{1}{\e^{2\Delta}-1}
\biggl[ 1 +  \cO \biggl( \frac{1}{\e^{4\Delta}\hat\rho(s_k)^4} \biggr)
\biggr]\;. 
\end{equation} 
Next we estimate $P_{k,1}$. We first obtain the bound 
\begin{align}
\nonumber
\norm{\zeta^{k,1}_{s\wedge\tau_{\cS(h)}}} 
&\leqs \frac{1}{\mu} \int_{s_{k-1}}^{s\wedge\tau_{\cS(h)}}
\norm{U(s,u)} \norm{b(\zeta_u,u)} \6u \\
\nonumber
&\leqs \const \frac{1}{\mu} \int_{s_{k-1}}^{s\wedge\tau_{\cS(h)}}
\e^{\alpha(s,u)/\mu} h^2 \hat\rho(u)^2 \6u  \\
&\leqs \const\; h^2 \frac{\hat\rho(s_{k-1})^2}{2s_{k-1}} \e^\Delta\;,
\label{p_esc_n:18}
\end{align}
where we have used the fact that the function $u\mapsto\hat\rho(u)^2/2u$ is
decreasing, and bounded the integral of $(2u/\mu)\e^{\alpha(s,u)/\mu}$ by
$\e^\Delta$. Using a Taylor expansion of $\hat\rho(s)^2$ and the definitions
of $\hat\rho$ and of the partition, one finds 
\begin{equation}
 \label{p_esc_n:19}
\frac{\hat\rho(s_{k-1})^2}{\hat\rho(s)^2} 
\leqs 1 + \frac{s-s_{k-1}}{\hat\rho(s)^2} \sup_{u\in[s_{k-1},s]}
\abs{(\hat\rho(u)^2)'} 
\leqs 1 + \Order{\Delta}
\end{equation} 
for all $s\in[s_{k-1},s_k]$. Together with~\eqref{p_esc_n:18}, this implies
\begin{equation}
 \label{p_esc_n:20}
\frac{\norm{\zeta^{k,1}_{s\wedge\tau_{\cS(h)}}}}{\hat\rho(s)}
\leqs \const \; h^2 \frac{\hat\rho(s_{k-1})}{s_{k-1}}\sqrt{\Delta}\e^\Delta  
\leqs \const \; \frac{h^2\sqrt{\Delta}\e^\Delta}{\mu^{3/4}} =: \frac{H_1}{2}\;, 
\end{equation} 
which yields $P_{k,1}(\zeta,H_1)=0$. 
Substituting $H_0=h+H_1$ in~\eqref{p_esc_n:17} thus yields 
\begin{equation}
 \label{p_esc_n:21}
P_k \leqs q(\Delta) := \frac{h^2}{\sigma^2} \frac{1}{\e^{2\Delta}-1}
\biggl[ 1 + \cO \biggl( \frac{1}{\e^{4\Delta}} \biggr)
+   \cO \biggl( \frac{h\sqrt{\Delta}\e^\Delta}{\mu^{3/4}} \biggr) \biggr] 
\end{equation}
for $k=1,\dots,N-1$. Finally, we make the choice 
\begin{equation}
 \label{p_esc_n:22}
\Delta = \frac{1+\nu}{2} \log 
\biggl( 1+\nu+\frac{h^2}{\sigma^2} \biggr)\;. 
\end{equation} 
Bounding $P_k$ above amounts to bounding $ q(\Delta)^{-1}$ below. 
For this we write 
\begin{align}
\label{p_esc_n:23}
q(\Delta)^{-1} 
&= \frac{\sigma^2}{h^2} \e^{2\Delta} 
\biggl[ 1 - \Order{\e^{-2\Delta}} - \Order{\e^{-4\Delta}} 
- \cO \biggl( \frac{h\sqrt{\Delta}\e^\Delta}{\mu^{3/4}}\biggr) \biggr] \\
&\geqs \biggl( 1+\nu+\frac{h^2}{\sigma^2} \biggr)^\nu
\biggl[ 1 - \cO \biggl( \biggl(\frac{\sigma^2}{h^2} \biggr)^{1+\nu}\biggr)
- \cO \biggl( \frac{h}{\mu^{3/4}} 
\biggl( \frac{h}{\sigma}\biggr)^{1+\nu}
\log \biggl( 1+\nu+ \frac{h^2}{\sigma^2}\biggr)^{1/2}
 \biggr) \biggr] \;.
\nonumber
\end{align}
Note that the error term $\Order{\e^{-4\Delta}}$ is negligible and no longer
appears in the last line. Now by Assumption~\eqref{p_esc_n:03}, for $c_0$ and
$c_1$ small enough, we get 
\begin{equation}
 \label{p_esc_n:24}
q(\Delta)^{-1} \geqs  \frac12 \biggl( 1+\nu+\frac{h^2}{\sigma^2} \biggr)^\nu
\end{equation} 
and 
\begin{equation}
 \label{p_esc_n:25}
\frac{\log  q(\Delta)^{-1}}{\Delta} \geqs \frac{2\nu}{1+\nu}
- \frac{2\log2}{(1+\nu)\log(1+\nu+h^2/\sigma^2)}
=: \kappa(\nu)\;.
\end{equation} 
Note that $\kappa(\nu)$ is indeed of the form~\eqref{p_esc_n:05}. The result
thus follows from \eqref{p_esc_n:10}. 
\end{proof}


\subsection{Averaging and Drift-Dominated Escape}
\label{appendix:proof_escape_drift} 

We consider again Equation~\eqref{eq:p_escape:01}, but this time for slightly
larger values of $\norm{\zeta}$. We start by transforming the system to polar
coordinates.

\begin{lem}
\label{lem:polar_coords}
Consider a system of the form
\begin{align}
\nonumber
\6\xi &= \frac{1}{\mu} f_\xi(\xi,\eta,z)\6z + \frac{\sigma}{\sqrt{\mu}}
F_\xi(z)\6W_z\;, \\
\6\eta &= \frac{1}{\mu} f_\eta(\xi,\eta,z)\6z + \frac{\sigma}{\sqrt{\mu}}
F_\eta(z)\6W_z\;,
\label{eq:esc_drift01}  
\end{align}
where $W_z$ denotes a two-dimensional Wiener process, and $F_\xi$ and $F_\eta$
are row vectors of dimension $2$. Then in polar coordinates
$(\xi=r\cos\varphi,\eta=r\sin\varphi)$, the system becomes 
\begin{align}
\nonumber
\6 r &= \frac{1}{\mu} f_r(r,\varphi,z)\6z + \frac{\sigma}{\sqrt{\mu}}
F_r(\varphi,z)\6W_z\;, \\
\6\varphi &= \frac{1}{\mu} \frac1r f_\varphi(r,\varphi,z)\6z +
\frac{\sigma}{\sqrt{\mu}} \frac1r
F_\varphi(\varphi,z)\6W_z\;,
\label{eq:esc_drift02}  
\end{align}
where the new and old diffusion coefficients are related via 
\begin{align}
\nonumber
F_r &= F_\xi\cos\varphi + F_\eta\sin\varphi\;, \\
F_\varphi &= F_\eta\cos\varphi - F_\xi\sin\varphi \;,
\label{eq:esc_drift03}  
\end{align}
while the drift coefficients are given by 
\begin{align}
\nonumber
f_r &= f_\xi\cos\varphi + f_\eta\sin\varphi
+ \frac{\sigma^2}{2r} F_\varphi F_\varphi^T\;, \\
f_\varphi &= f_\eta\cos\varphi - f_\xi\sin\varphi
- \frac{\sigma^2}{r} F_r F_\varphi^T \;.
\label{eq:esc_drift04}  
\end{align}
\end{lem}
\begin{proof}
The formulas can be checked directly by applying It\^o's formula
to~\eqref{eq:esc_drift02}.
\end{proof}

Applying this result to~\eqref{eq:p_escape:01}, we obtain a system of the form
\begin{align}
\nonumber
\6r &= \frac{1}{\mu}
\biggl[ a(z)r + r^2b_r(\varphi,z) + \cO \biggl( \frac{\sigma^2}{r}\biggr)
\biggr]
\6z 
+ \frac{\sigma}{\sqrt{\mu}} F_r(\varphi,z)\6W_z\;, \\
\6\varphi &= \frac{1}{\mu}
\biggl[ -\varpi(z) + r b_\varphi(\varphi,z) + \cO \biggl(
\frac{\sigma^2}{r^2}\biggr) \biggr]
\6z 
+ \frac{\sigma}{\sqrt{\mu}} \frac{1}{r} F_\varphi(\varphi,z)\6W_z\;.
\label{eq:esc_drift05} 
\end{align}
Note that the functions $b_r$ and $b_\varphi$ do not depend on $r$, owing to
the fact that the nonlinearity in the original equation is homogeneous of degree
$2$. Another important observation is that the average of $b_r$ (and
$b_\varphi$) over $\varphi$ is zero. This follows again from homogeneity,
combined with~\eqref{eq:esc_drift04}. This observation suggests to
simplify~\eqref{eq:esc_drift05} by an averaging transformation. 

\begin{prop}
 \label{prop:escape_average}
There exists a function $w(\varphi,z)$, which is bounded, smooth, and
$2\pi$-periodic in $\varphi$, such that $\bar r=r+r^2 w(\varphi,z)$ satisfies
the
SDE 
\begin{equation}
 \label{eq:esc_drift06}
\6\bar r = \frac{1}{\mu} \bigl[ a(z)\bar r + \beta(\bar r,\varphi,z) \bigr] \6z 
+ \frac{\sigma}{\sqrt{\mu}} \widetilde F_r(\bar r,\varphi,z) \6W_z\;,
\end{equation} 
where 
\begin{align}
\nonumber
\beta(\bar r,\varphi,z) &= \Order{\bar r^3} + \Order{\mu \bar r^2} 
+ \cO \biggl( \frac{\sigma^2}{\bar r}\biggr)\;, \\
\widetilde F_r(\bar r,\varphi,z)
&= F_r(\varphi,z) + \Order{\bar r}\;.
 \label{eq:esc_drift07}
\end{align}
\end{prop}

\begin{proof}
Using It\^o's formula and the fact that $r=\bar r-\bar r^2
w(\varphi,z)+\Order{\bar r^3}$, we obtain 
\begin{align}
\nonumber
\6\bar r ={}& \frac{1}{\mu} \biggl[ a(z) \bar r + \bar r^2 \biggl(
2zw(\varphi,z)
+ b_r(\varphi,z) - \varpi(z)\frac{\partial w}{\partial\varphi} + 
\mu \frac{\partial w}{\partial z}\biggr) + \Order{\bar r^3} 
+ \cO \biggl( \frac{\sigma^2}{\bar r}\biggr)
\biggr]\6z \\
&{}+ \frac{\sigma}{\sqrt{\mu}} 
\biggl[ F_r(\varphi,z) + \bar r \biggl( 2w(\varphi,z) F_r(\varphi,z) + 
\frac{\partial w}{\partial \varphi} F_\varphi(\varphi,z) 
\biggr) + \Order{\bar r^2} \biggr] \6W_z\;.
 \label{eq:esc_drift08}
\end{align}
It is thus sufficient to show that the equation 
\begin{equation}
 \label{eq:esc_drift09}
\frac{\partial w}{\partial \varphi} (\varphi,z) 
= \frac{a(z)}{\varpi(z)} w(\varphi,z) + \frac{b_r(\varphi,z)}{\varpi(z)} 
\end{equation} 
admits a bounded, $2\pi$-periodic solution. Letting $c=c(z)=a(z)/\varpi(z)$, the
general solution of~\eqref{eq:esc_drift09} can be written 
\begin{equation}
 \label{eq:esc_drift10}
w(\varphi,z) = w(0,z)\e^{c\varphi} + 
\int_0^\varphi \e^{c(\varphi-\varphi')}  \frac{b_r(\varphi',z)}{\varpi(z)}
\6\varphi'\;.
\end{equation} 
Thus choosing  
\begin{equation}
 \label{eq:esc_drift11}
w(0,z) = \frac{\e^{2\pi c}}{1-\e^{2\pi c}} \int_0^{2\pi}
\e^{-c\varphi}  \frac{b_r(\varphi,z)}{\varpi(z)}
\6\varphi\;,
\end{equation} 
the resulting $w(\varphi,z)$ is indeed $2\pi$-periodic as a function of $\varphi$. Finally note that 
\begin{equation}
 \label{eq:esc_drift12}
\lim_{z\to0} w(0,z) = \frac{1}{2\pi} \int_0^{2\pi} \varphi
\frac{b_r(\varphi,0)}{\varpi(0)}
\6\varphi\;,
\end{equation} 
showing that $w(\varphi,z)$ is also bounded as $z\to0$. 
\end{proof}

We now define the set 
\begin{equation}
 \label{eq:esc_drift13}
\cD(\eta) = \left\{ (\bar r,\varphi,z) \colon z\geqs\sqrt{\mu},
\bar r<\eta\sqrt{z} \right\}\;. 
\end{equation} 
Then the nonlinear term $\beta$ satisfies, on the set
$\cR=\cD(\eta)\setminus\cS(h)$, 
\begin{equation}
 \label{eq:esc_q_02}
\frac{\abs{\beta(\bar r,\varphi,z)}}{\bar r z} \leqs 
M \biggl( \frac{\bar r^2}{z} + \frac{\mu \bar r}{z} + \frac{\sigma^2}{\bar r^2
z}\biggr) 
\leqs M' \biggl( \eta^2 +\eta\mu^{3/4} + \frac{1}{\abs{\log\sigma}^2}\biggr)
\end{equation} 
for some constants $M$, $M'$. 
We can thus find, for any $0<\kappa<2$, an $\eta=\eta(\kappa)$ such that (for
sufficiently small $\sigma$ and $\mu$) 
\begin{equation}
 \label{eq:esc_q_03}
 a(z)\bar r + \beta(\bar r,\varphi,z) \geqs \kappa z \bar r
\end{equation}
holds in $\cR$. Thus we have 
 \begin{equation}
 \label{eq:esc_q_04}
\6\bar r = \frac{1}{\mu} \bigl[ \kappa z \bar r + \tilde\beta(\bar r,\varphi,z)
\bigr] \6z 
+ \frac{\sigma}{\sqrt{\mu}} \widetilde F_r(\bar r,\varphi,z) \6W_z\;,
\end{equation} 
where $\tilde\beta(\bar r,\varphi,z)\geqs 0$ in $\cR$, implying 
\begin{equation}
\bar r_z \geqs \bar r_{z_0} \e^{\kappa(z^2-z_0^2)/2\mu}
+ \frac{\sigma}{\sqrt{\mu}} \int_{z_0}^z \e^{\kappa(z^2-s^2)/2\mu}
\widetilde F_r(\bar r_s,\varphi_s,s)\6W_s 
 \label{eq:esc_q_05}
\end{equation} 
holds as long as the process stays in $\cR$. 

The following proposition shows that the process is unlikely to stay in $\cR$ for times significantly larger than $(\mu\abs{\log \sigma})^{1/2}$.

\begin{prop}
\label{prop:escape_drift}
There exists a constant $\kappa_2>0$ such that for any initial condition
$(\bar r_0,\varphi_0,z_0)\in\cR$,
\begin{equation}
 \label{eq:esc_drift14} 
\P^{(\bar r_0,\varphi_0,z_0)} \bigl\{ \tau_{\cR} > z \bigr\}
\leqs 2 \exp \biggl\{ - \kappa_2
\frac{z^2-z_0^2}{\mu\abs{\log\sigma}}\biggr\} \;.
\end{equation}  
\end{prop}
\begin{proof}
We introduce a partition $z_0<z_1<\dots<z_N=z$ of $[z_0,z]$, given by 
\begin{equation}
 \label{eq:esc_drift14:1}
z_{k+1}^2 - z_k^2 = \gamma \mu \abs{\log\sigma}
\qquad
\text{for }
0\leqs k < N = \biggl\lceil \frac{z^2-z_0^2}
{\gamma \mu \abs{\log\sigma}}\biggr\rceil \;.
\end{equation} 
The Markov property implies that 
\begin{equation}
 \label{eq:esc_drift14:2}
 \P^{(\bar r_0,\varphi_0,z_0)} \bigl\{ \tau_{\cR} > z \bigr\} \leqs 
\prod_{k=1}^{N-1} P_k\;,
\end{equation} 
where 
\begin{align}
 \nonumber
P_k &= \sup_{\bar r,\varphi \colon (\bar r,\varphi,z_k)\in\cR}  
P_k(\bar r,\varphi)\;, \\
P_k(\bar r,\varphi) &= 
\P^{(\bar r,\varphi,z_k)} \bigl\{ \tau_\cR > z_{k+1} \bigr\}\;.
 \label{eq:esc_drift14:3}
\end{align} 
Inequality~\eqref{eq:esc_q_05} (with $z_0$ replaced by $z_k$) shows
that  
\begin{equation}
 \label{eq:esc_drift14:4B}
\bar r_z \geqs \e^{\kappa(z^2-z_k^2)/2\mu}
\biggl[ \bar r_{z_k} + \frac{\sigma}{\sqrt{\mu}} M^k_z \biggr]
\end{equation} 
holds for $z_k \leqs z \leqs \tau_\cR$, 
where $M^k_z$ is the martingale 
\begin{equation}
 \label{eq:esc_drift14:4}
M^k_z = \int_{z_k}^z \e^{-\kappa(s^2-z_k^2)/2\mu} \widetilde F_r(\bar
r_s,\varphi_s,s) \,\6W_s\;. 
\end{equation}
It follows that 
\begin{align}
\nonumber 
 P_k(\bar r,\varphi)
&\leqs \P \biggl\{ \bar r +  \frac{\sigma}{\sqrt{\mu}}
 M^k_{z_{k+1}} 
< \eta\sqrt{z_{k+1}} \e^{-\kappa(z_{k+1}^2-z_k^2)/2\mu}\biggr\} \\
& \leqs \P \Bigl\{ 
 M^k_{z_{k+1}} 
< - c\abs{\log\sigma}\sqrt{\mu}\hat\rho(z_k) + 
\eta\sqrt{\mu}\,\sigma^{\gamma\kappa/2-1}\sqrt{z_{k+1}} \Bigr\}\;,
 \label{eq:esc_drift14:5}
\end{align} 
where we have used $\bar r \geqs h\hat\rho(z_k) =
c\sigma\abs{\log\sigma}\hat\rho(z_k)$ and 
$\e^{-\kappa(z_{k+1}^2-z_k^2)/2\mu}=\sigma^{\kappa\gamma/2}$. 
Choosing $\gamma > 2/\kappa$, we can guarantee that the term 
$c\abs{\log\sigma}\sqrt{\mu}\hat\rho(z_k)$ dominates. 

Since the noise acting on the system
is non-degenerate, we may assume the existence of constants $D_+\geqs D_->0$
such that 
\begin{equation}
 \label{eq:esc_drift14:7}
 D_- \leqs \widetilde F_r(\bar r,\varphi,z) \widetilde F_r(\bar r,\varphi,z)^T
\leqs D_+ \;.
\end{equation}
Thus the variance of $M^k_{z_{k+1}}$ is bounded above by 
\begin{equation}
V_+ = \int_{z_k}^{z_{k+1}} D_+ \e^{-\kappa(s^2-z_k^2)/\mu}\,\6s 
\leqs D_+ \biggl( -\frac{\mu}{2\kappa z_k} \biggr)
\e^{-\kappa(s^2-z_k^2)/\mu} \biggr|_{z_k}^{z_{k+1}}
\leqs \frac{D_+\mu}{2\kappa z_k}\;.
 \label{eq:esc_drift14:8}
\end{equation}
A Bernstein-type estimate (cf.~Lemma~\ref{lem:teclem01} in
Appendix~\ref{appendix:technical_lemmas}) provides the bound 
\begin{equation}
 \label{eq:esc_drift14:9}
\P \bigl\{ M^k_{z_{k+1}} < -x \bigr\} \leqs \e^{-x^2/2V_+}\;. 
\end{equation} 
Using~\eqref{eq:esc_drift14:8} and~\eqref{eq:esc_drift14:9}
in~\eqref{eq:esc_drift14:5} shows that we may assume $P_k<1/2$, and the result
follows from~\eqref{eq:esc_drift14:2} and the definition of $N$. Note that
$\kappa_2=\log2/\gamma < (\log2/2)\kappa$. 
\end{proof}


\subsection{Laplace Transforms}
\label{appendix:proof_escape_Laplace} 

In order to combine the results from the two previous subsections, we will use a
lemma based on Laplace transforms. In the following we let $\{x_t\}_{t\geqs0}$
be a time-homogeneous $\R^d$-valued Markov process with continuous sample paths.
All subsets $A\subset\R^d$ considered below are assumed to have smooth boundary,
and to be such that the first-exit time $\tau_A=\inf\{t\geqs0 \colon x_t\not\in
A\}$ is almost surely finite. The \emph{Laplace transform}\/ of $\tau_A$ is the
non-decreasing function 
\begin{equation}
 \label{eq:Laplace_transform}
\R \ni \lambda \mapsto 
 \E^{x} \bigl[ \e^{\lambda\tau_A} \bigr] 
= 1 + \lambda \int_0^\infty \P^x \bigl\{ \tau_A>t \bigr\} \e^{\lambda t} \,\6t 
\in [0,\infty] \;.
\end{equation} 
Note that $ \E^{x}[\e^{\lambda\tau_A}]\leqs 1$ for all $\lambda\leqs0$. 
Thus there exists a $\lambda_0\geqs0$ such that $ \E^{x}[\e^{\lambda\tau_A}] <
\infty$ for all $\lambda < \lambda_0$. 

\begin{lem}
\label{lem:Laplace} 
Choose nested bounded open sets $\cS_1\subset\cS_2\subset\cD\subset\R^d$. Let
$\cR=\cD\setminus\cS_1$ and consider the Laplace transforms 
\begin{align}
\nonumber
G_\cD(\lambda) &= \sup_{x\in\cS_1} 
\E^{x} \bigl[ \e^{\lambda\tau_{\cD}} \bigr]\;, \\
\nonumber
G_\cS(\lambda) &= \sup_{x\in\cS_1}
\E^{x} \bigl[ \e^{\lambda\tau_{\cS_2}} \bigr]\;, \\
\nonumber
G_\cR(\lambda) &= \sup_{x\in\partial\cS_2}
\E^{x} \bigl[ \e^{\lambda\tau_\cR} \bigr]\;, \\
Q(\lambda) &= \sup_{x\in\partial\cS_2}
\E^{x} \bigl[ 1_{\{\tau_{\cS_1^c}<\tau_\cD\}}\e^{\lambda\tau_\cR} \bigr]\;. 
\label{eq:Laplace02}
\end{align}
Let $\lambda$ be such that $G_\cS(\lambda)$ and $G_\cR(\lambda)$ are finite,
and assume that $Q(\lambda) G_\cS(\lambda) < 1$. Then $G_\cD(\lambda)$ is
also finite and satisfies 
\begin{equation}
 \label{eq:Laplace04}
G_\cD(\lambda) \leqs 
\frac{G_\cS(\lambda)G_\cR(\lambda)}{1 - Q(\lambda) G_\cS(\lambda)} \;.
\end{equation} 
\end{lem}
\begin{proof}
For $x_0\in\cS_1$, one necessarily has $\tau_{\cS_2}\leqs\tau_\cD$, and thus 
the strong Markov property implies
\begin{align}
\nonumber
\E^{x_0} \bigl[ \e^{\lambda\tau_{\cD}} \bigr]
&= \E^{x_0} \Bigl[ \e^{\lambda\tau_{\cS_2}} 
\E^{x_{\tau_{\cS_2}}} \bigl[
\e^{\lambda\tau_{\cD}} \bigr]
\Bigr] \\
&\leqs \E^{x_0} \bigl[ \e^{\lambda\tau_{\cS_2}}\bigr]
\sup_{x\in\partial\cS_2} \E^{x} \bigl[
\e^{\lambda\tau_\cD}\bigr]\;.
 \label{eq:Laplace05}
\end{align} 
Similarly, for $x\in\partial\cS_2$, since
$\tau_\cR=\tau_{\cS_1^c}\wedge\tau_\cD$
we have 
\begin{align}
\nonumber
\E^{x} \bigl[ \e^{\lambda\tau_\cD} \bigr]
&= \E^{x} \Bigl[
1_{\{\tau_{\cS_1^c}<\tau_\cD\}}
\e^{\lambda\tau_{\cS_1^c}}
\E^{x_{\tau_{\cS_1^c}}} \bigl[
\e^{\lambda\tau_\cD} \bigr]
\Bigr] 
+ \E^{x} \bigl[1_{\{\tau_\cD<\tau_{\cS_1^c}\}}
\e^{\lambda\tau_\cD}\bigr]\\
\nonumber
&\leqs \E^{x} \Bigl[
1_{\{\tau_{\cS_1^c}<\tau_\cD\}}
\e^{\lambda\tau_\cR}
\E^{x_{\tau_{\cS_1^c}}} \bigl[
\e^{\lambda\tau_\cD} \bigr]
\Bigr] 
+ \E^{x} \bigl[1_{\{\tau_\cD<\tau_{\cS_1^c}\}}
\e^{\lambda\tau_\cR}\bigr]\\
&\leqs Q(\lambda) G_\cD(\lambda) + G_\cR(\lambda)\;.
 \label{eq:Laplace06}
\end{align} 
Now, \eqref{eq:Laplace05} and~\eqref{eq:Laplace06} imply 
\begin{equation}
 \label{eq:Laplace07}
G_\cD(\lambda) \leqs G_\cS(\lambda) 
\bigl[Q(\lambda) G_\cD(\lambda) + G_\cR(\lambda) \bigr] \;,
\end{equation} 
which yields the result. 
\end{proof}

We will apply this lemma to sets $\cS_1=\cS(h_1)$, $\cS_2=\cS(h_2)$ and
$\cD=\cD(\kappa)$, where $h_i = c_i\sigma\abs{\log\sigma}$, $i=1,2$, with
$0<c_1<c_2$. We introduce a new time variable $t=z^2$, and let $x_t$ be the
time-homogeneous Markov process $(\sqrt{t},\zeta_{\sqrt{t}})$.  

Proposition~\ref{prop:escape_noise} yields a control of $G_\cS(\lambda)$ in the
following way. The bound \eqref{p_esc_n:04} translates in terms of the new
process $x_t$ as 
\begin{equation}
 \label{eq:Laplace08}
\P^{x_0} \bigl\{ \tau_{\cS_2} \geqs t \bigr\}
\leqs C_1 \e^{-\lambda_1 t}\;, 
\qquad
x_0=(z_0,\zeta_{z_0})\;,
\end{equation} 
where $C_1=C(\nu)(h_2/\sigma)^{2\nu}$ and $\lambda_1=\kappa(\nu)/\mu$. It
follows that 
\begin{equation}
 \label{eq:Laplace09}
G_\cS(\lambda) \leqs 
1 + \lambda\int_0^\infty \e^{\lambda t}  
\sup_{x_0\in\cS_1} \P^{x_0}\bigl\{ \tau_{\cS_2} \geqs t \bigr\} \6t
\leqs 1 + \frac{C_1\lambda}{\lambda_1 - \lambda}
\end{equation} 
holds for all $\lambda < \lambda_1 = \kappa(\nu)/\mu$. 
In a similar way, Proposition~\ref{prop:escape_drift} yields 
\begin{equation}
 \label{eq:Laplace09B}
G_\cR(\lambda)  
\leqs 1 + \frac{C_2\lambda}{\lambda_2 - \lambda}
\end{equation} 
for all $\lambda < \lambda_2=\kappa_2/\mu\abs{\log\sigma}$, where
$C_2=2$. It remains to estimate $Q(\lambda)$. Let us first show that
$Q(\lambda)$ can be bounded in terms of $Q(0)$. 

\begin{lem}
\label{lem:Q_lambda}
For all $\lambda<\lambda_2$, one has 
\begin{equation}
 \label{eq:Laplace10}
Q(\lambda) \leqs \frac{C_2^{\lambda/\lambda_2}}{1-\lambda/\lambda_2}
Q(0)^{1-\lambda/\lambda_2}\;. 
\end{equation}  
\end{lem}
\begin{proof}
First note that for all $T\geqs0$, 
\begin{equation}
 \label{eq:Laplace10:1}
\e^{\lambda\tau_\cR} 
\leqs \e^{\lambda T} + \lambda \int_T^\infty 
1_{\{\tau_\cR > t\}} \e^{\lambda t}\,\6t\;.
\end{equation} 
Plugging this into the definition of $Q(\lambda)$ yields 
\begin{equation}
 \label{eq:Laplace10:2}
Q(\lambda) \leqs  
\sup_{x\in\partial\cS_2} \biggl[
\P^x \bigl\{\tau_{\cS_1^c}<\tau_\cD\bigr\}\e^{\lambda T}
+ \lambda \int_T^\infty \P^x \bigl\{ \tau_\cR>t \bigr\} \e^{\lambda t} \,\6t
\biggr]\;.
\end{equation} 
The first term on the right-hand side is bounded by $Q(0)\e^{\lambda T}$. The
second one can be estimated by Proposition~\ref{prop:escape_drift}, yielding  
\begin{equation}
 \label{eq:Laplace10:3}
Q(\lambda) \leqs Q(0) \e^{\lambda T} 
+ \frac{C_2\lambda}{\lambda_2-\lambda} \e^{-(\lambda_2-\lambda)T}\;. 
\end{equation} 
Optimizing over $T$, we find that the optimal bound is obtained when
$\e^{\lambda_2T}=C_2/Q(0)$, which yields~\eqref{eq:Laplace10}. 
\end{proof}

Finally, $Q(0)$ can be estimated in a similar way as in the proof of 
Proposition~\ref{prop:escape_drift}.

\begin{prop}
\label{prop:Q(0)}
We have 
\begin{equation}
 \label{eq:Laplace11}
Q(0) = \sup_{x\in\partial\cS_2} 
\P^x \bigl\{\tau_{\cS_1^c}<\tau_\cD\bigr\}
 \leqs \sigma^{\kappa(c_2-c_1)^2\abs{\log\sigma}z_{0}\hat\rho(z_0)^{2} /D_+}\;.
\end{equation}
\end{prop}
\begin{proof}
Inequality~\eqref{eq:esc_q_05} and the fact that 
$\e^{\kappa z^2/2\mu}/\hat\rho(z)$ is increasing for sufficiently small~$z$
imply 
\begin{equation}
 \label{eq:esc_q_06a}
\bar r_z -h_1 \hat\rho(z) \geqs \e^{\kappa(z^2-z_0^2)/2\mu}
\biggl[ (h_2-h_1) \hat\rho(z_0) + 
\frac{\sigma}{\sqrt{\mu}} M^0_z
\biggr]
\end{equation} 
for $z\leqs\tau_\cR$, where $M^0_z$ is the martingale introduced
in~\eqref{eq:esc_drift14:4}. The variance of $M^0_z$ being bounded by
$D_+\mu/2\kappa z_0$ (cf.~\eqref{eq:esc_drift14:8}), the Gaussian tail
estimate of Lemma~\ref{lem:teclem01} allows to write
\begin{align}
\nonumber
\P \bigl\{ \tau_{\cS_1^c} < z \wedge \tau_\cR \bigr\}
&= \P \biggl\{ \inf_{z_0\leqs s\leqs z \wedge \tau_\cR} \bigl( \bar r_s -
h_1 \hat\rho(z)\bigr) \leqs 0 \biggr\} \\
&\leqs \exp \biggl\{ -\kappa\frac{(c_2-c_1)^2}{D_+} \abs{\log\sigma}^2\biggr\}
= \sigma^{\kappa(c_2-c_1)^2\abs{\log\sigma}z_{0}\hat\rho(z_0)^{2}/D_+}\;.
\label{eq:esc_q_08} 
\end{align}
Note that the right-hand side of~\eqref{eq:esc_q_08} does not depend on $z$.
The result thus follows from taking the limit $z\to\infty$. 
\end{proof}

\begin{proof}[Proof of Theorem~\ref{thm:escape_canards}]
Since $\lambda_2<\lambda_1$ for sufficiently small $\sigma$, we set
$\lambda=(1-\theta)\lambda_2$ for a fixed $0<\theta<1$. For this $\lambda$, we
have 
\begin{equation}
 \label{eq:esc_proof:1}
G_\cR(\lambda) \leqs 1+\frac{C_2}{\theta}\;,
\qquad
G_\cS(\lambda) \leqs 1+\frac{C_1\lambda_2}{\lambda_1-\lambda_2}
= 1 + \cO \biggl( \frac{1}{\abs{\log\sigma}} \biggr)\;.
\end{equation} 
Furthermore Lemma~\ref{lem:Q_lambda} and Proposition~\ref{prop:Q(0)} yield 
\begin{equation}
 \label{eq:esc_proof:2}
Q(\lambda) \leqs \frac{C_2}{\theta} \sigma^{c\theta\abs{\log\sigma}} 
\end{equation} 
for some constant $c>0$. Thus Lemma~\ref{lem:Laplace} can be applied to show
that $G_\cD(\lambda)$ is finite. 
Finally, by Markov's inequality, 
\begin{equation}
 \label{eq:Laplace100}
\P^x\bigl\{\tau_\cD\geqs t\bigr\}
= \P^x\bigl\{\e^{\lambda\tau_\cD}\geqs \e^{\lambda t}\bigr\}
\leqs \e^{-\lambda t} G_\cD(\lambda)\;,
\end{equation} 
which gives the theorem when translated back to the process $\zeta_z$. 
\end{proof}


\subsection{A Gaussian Tail Estimate for Martingales}
\label{appendix:technical_lemmas}

Let $W_z$ be an $n$-dimensional standard Brownian motion, and 
consider the martingale 
\begin{equation}
 \label{eq:teclem01}
M_z = \int_0^z g(X_s,s) \,\6W_s
= \sum_{i=1}^{n}  \int_0^z g_{i}(X_s,s) \,\6W^{(i)}_s \;,
\end{equation}  
where $g=(g_{1},\dots,g_{n})$ takes values in $\R^{n}$ and the process $X_z$ is assumed to be adapted to the filtration generated by $W_z$. We
will assume that the integrand satisfies 
\begin{equation}
 \label{eq:teclem02}
G_-(z)^2 \leqs g(X_z,z) g(X_z,z)^T \leqs G_+(z)^2
\end{equation} 
almost surely, for deterministic functions $G_\pm(z)$, and that the integrals 
\begin{equation}
 \label{eq:teclem03}
V_\pm(z) = \int_0^z G_\pm(s)^2\,\6s
\end{equation} 
are finite. Note that under these conditions, $M_{z}$ is indeed a continuous martingale.

\begin{lem}
\label{lem:teclem01}
For any $x>0$, 
\begin{equation}
 \label{eq:teclem04}
\P \biggl\{ \sup_{0\leqs s\leqs z} M_s > x\biggr\}
\leqs \e^{-x^2/2V_+(z)}\;.
\end{equation} 
\end{lem}
\begin{proof}
Let 
\begin{equation}
 \label{eq:teclem05:1}
[M]_z = \int_0^z  g(X_s,s) g(X_s,s)^T \,\6s
\end{equation} 
be the increasing process associated with $M_z$. Then, for any $\gamma\in\R$, the
Dol\'eans exponential
\begin{equation}
 \label{eq:teclem05:2}
\e^{\gamma M_z - \gamma^2 [M]_z/2} 
\end{equation} 
is a martingale. It follows that 
\begin{align}
\nonumber
 \P \biggl\{ \sup_{0\leqs s\leqs z} M_s > x\biggr\} 
&= \P \biggl\{ \sup_{0\leqs s\leqs z} \e^{\gamma M_s}  >
\e^{\gamma x}\biggr\} \\
\nonumber
&\leqs \P \biggl\{ \sup_{0\leqs s\leqs z} \e^{\gamma M_s - \gamma^2 [M]_s/2}  >
\e^{\gamma x - \gamma^2 [M]_z/2}\biggr\} \\
\nonumber
&\leqs \P \biggl\{ \sup_{0\leqs s\leqs z} \e^{\gamma M_s - \gamma^2 [M]_s/2}  >
\e^{\gamma x - \gamma^2 V_+(z)/2}\biggr\} \\
&\leqs \e^{-\gamma x + \gamma^2 V_+(z)/2} \E \Bigl[ \e^{\gamma M_z - \gamma^2
[M]_z/2}\Bigr]
 \label{eq:teclem05:3}
\end{align} 
by Doob's submartingale inequality. Now the expectation in the last line is
equal to $1$, and the result follows by optimizing over $\gamma$, that is,
choosing $\gamma=x/V_+(z)$. 
\end{proof}

\bibliographystyle{plain}
\bibliography{../BDGK}

\newpage
\tableofcontents

\end{document}